\documentclass[12pt,a4paper,reqno]{amsart}
\usepackage{amssymb}
\usepackage{amscd}
\numberwithin{equation}{section}

\theoremstyle{plain}

\newtheorem{theorem}{Theorem}[section]
\newtheorem{proposition}[theorem]{Proposition}
\newtheorem{lemma}[theorem]{Lemma}
\newtheorem{corollary}[theorem]{Corollary}

\theoremstyle{definition}

\newtheorem{definition}[theorem]{Definition}
\newtheorem{remark}[theorem]{Remark}

\newtheorem{example}[theorem]{Example}
\newtheorem{examples}[theorem]{Examples}

\renewcommand{\leq}{\leqslant}
\renewcommand{\geq}{\geqslant}

\newsavebox{\proofbox}
\savebox{\proofbox}{\begin{picture}(7,7)%
  \put(0,0){\framebox(7,7){}}\end{picture}}

\newcommand\E{\mathbb{E}}
\newcommand\Z{\mathbb{Z}}
\newcommand\R{\mathbb{R}}
\newcommand\T{\mathbb{T}}
\newcommand\C{\mathbb{C}}

\newcommand\Zcal{\mathcal{Z}}
\newcommand\X{\mathrm{X}}

\newcommand\Y{\mathrm{Y}}
\newcommand\D{\mathcal{D}}
\newcommand\B{\mathcal{B}}

\newcommand\G{\mathcal{G}}
\newcommand\I{\mathcal{I}}

\newcommand\F{\mathbb{F}}
\newcommand\Fw{\mathbb{F}^{\omega}}
\newcommand\Q{\mathbb{Q}}

\newcommand\x{{\bf x}}

\newcommand\w{{\bf w}}
\newcommand\1{{\bf 1}}
\newcommand\2{{\bf 2}}
\newcommand\z{{\bf z}}
\newcommand\s{{\bf s}}
\newcommand\y{{\bf y}}
\newcommand\g{{\bf g}}

\newcommand\m{{\bf m}}
\renewcommand\u{{\bf u}}
\newcommand\eps{\varepsilon}
\newcommand\ader{\Delta}
\newcommand\mder{{\Delta\!\!\!\!\!\hbox{\raisebox{0.2ex}{\tiny\ \textbullet}}\ \!}}
\newcommand\mdersmall{{\Delta\!\!\!\!\cdot\ \!}}

\newcommand\sgn{\operatorname{sgn}}

\newcommand\Abr{\operatorname{Abr}}

\newcommand\Phase{{\mathcal P}} 
\newcommand\charac{\operatorname{char}}
\newcommand\diag{\operatorname{diag}}
\newcommand\id{\operatorname{id}}
\newcommand\pt{\operatorname{pt}}
\renewcommand\th{{\operatorname{th}}}

\begin{document}

\title[An inverse theorem for the ergodic $\Fw$ uniformity seminorms ]{An inverse theorem for the uniformity seminorms associated with the action of $\F^\infty_p$}

\author{Vitaly Bergelson}
\address{Department of Mathematics, Ohio State University, Columbus OH 43210-1174}
\email{vitaly@math.ohio-state.edu}

\author{Terence Tao}
\address{UCLA Department of Mathematics, Los Angeles, CA 90095-1555.}
\email{tao@math.ucla.edu}

\author{Tamar Ziegler}
\address{Department of Mathematics, Technion, Haifa 32000, Israel.}
\email{tamarzr@tx.technion.ac.il}

\subjclass{}

\begin{abstract}  Let $\F$ a finite field.
We show that the universal characteristic factor for the Gowers-Host-Kra uniformity seminorm $U^k(\X)$ for an ergodic action $(T_g)_{g \in \Fw}$ of the infinite abelian group $\Fw$ on a probability space $X = (X,\B,\mu)$ is generated by phase polynomials $\phi: X \to S^1$ of degree less than $C(k)$ on $X$, where $C(k)$ depends only on $k$.
In the case where $k \leq \charac(\F)$ we obtain the sharp result $C(k)=k$.
This is a  finite field counterpart of an analogous result for $\Z$ by Host and Kra \cite{hk-cubes}.  In a companion paper \cite{tz-correspondence} to this paper, we shall combine this result with a correspondence principle to establish the inverse theorem for the Gowers norm in finite fields in the high characteristic case $k \leq \charac(\F)$, with a partial result in low characteristic.
\end{abstract}

\maketitle


\section{Introduction}

\subsection{Gowers-Host-Kra seminorms}

This paper is concerned with the structural theory of measure preserving actions of abelian 
groups.  We begin with some general definitions.

\begin{definition}[$G$-systems] Let $(G,+)$ be a locally compact abelian group.
A \emph{$G$-system} $\X = (X, \B_X, \mu_X, (T_g)_{g \in G})$ is a probability space $X = (X,\B_X,\mu_X)$ which is \emph{separable modulo null sets} (i.e. $\B_X$ is countably generated modulo null sets)\footnote{In order to carry out certain measure-theoretical constructions such as the disintegration of measures with respect to a factor, we will tacitly assume that the underlying measure spaces of the $G$-systems we will be dealing with are {\em regular}, meaning that in the triple $(X,\B,\mu)$, $X$ is a compact metric space, $\B$ is the (completion of the) $\sigma$-algebra of Borel sets, and $\mu$ is a Borel measure. Since every separable (modulo null sets) probability measure space is equivalent to a regular space (see for example \cite[Proposition 5.3]{Fur2}), this assumption can be made without any loss of generality.  Here, equivalence means equivalence of abstract $\sigma$-algebras modulo null sets; see \cite[Definition 5.2]{Fur2} for a precise definition.

Moreover, when dealing with factor maps between $G$-systems (see Definition \ref{factor-def} below) we will be assuming without specifically mentioning this that this regularity assumption applies simultaneously to a $G$-system and its factor. (Cf. \cite[Theorem 5.15]{Fur2})}, together with an action $g \mapsto T_g$ of $G$ on $X$ by measure-preserving transformations $T_g: X \to X$, and such that the map $(g,x) \mapsto T_g x$ is jointly measurable in $g$ and $x$.  We define the $L^p$ spaces $L^p(\X) = L^p(X,\B_X,\mu_X)$ for $1 \leq p \leq \infty$ in the usual manner (in particular, we identify any two functions in $L^p(\X)$ which agree $\mu$-almost everywhere).  If $X$ is a point, we write $\X = \pt$.
Given any measurable $\phi: X \to \C$ and $h \in G$, we define the shift
$T_h \phi := \phi \circ T_h$
and the multiplicative derivative
$\mder_h \phi := \overline{\phi} \cdot T_h \phi$.
Similarly, if $(U,+)$ is an abelian group and $\phi: X \to U$ is a measurable function, we define the shift $T_h \phi := \phi \circ T_h$ and the additive derivative
$\ader_h \phi := T_h \phi - \phi$.
We observe the commutativity relations $\mder_h \mder_k = \mder_k \mder_h$ and $\ader_h \ader_k = \ader_k \ader_h$ for all $h,k \in G$.  Observe also that for any $h \in G$, $\mder_h$ and $\ader_h$ are multiplicative and additive homomorphisms  from the group of measurable functions from $X  \to \C$ or  $X \to U$ endowed with pointwise multiplication or pointwise addition respectively, to itself.

We say that a $G$-system is \emph{ergodic} if the only functions in $L^2(\X)$ which are invariant under the $G$-action (by the shifts introduced above) are the constants.
\end{definition}

\begin{remark} Most of our analysis will take place in the setting of ergodic systems, but for various technical reasons we will sometimes have to work with non-ergodic systems.  In some (but not all) cases, results on ergodic systems can be extended satisfactorily to the non-ergodic case using the ergodic decomposition.  The hypothesis of separability is a technical one (used in particular in Appendix \ref{measure-sec} to obtain a certain measurability property), but can often be removed in applications by restricting the $\sigma$-algebra $\B_X$ to the sub-algebra generated by the functions one is interested in studying, together with all of their shifts.
\end{remark}

In most of our analysis, the group $G$ will be countable, discrete and abelian and hence has a 
\emph{F{\o}lner sequence}, i.e. a sequence $(\Phi_n)_{n=1}^\infty$ of finite subsets of $G$ satisfying\footnote{We use $|E|$ to denote the cardinality of a finite set $E$, and $\Delta$ to denote symmetric difference.} $|(\Phi_n+h) \Delta \Phi_n|/|\Phi_n| \to 0$ as $n \to \infty$ for all $h \in G$.  It is well known that one can always choose a F{\o}lner  sequence to be nested and to satisfy the condition $G = \bigcup_{n=1}^\infty \Phi_n$, and we will assume throughout this paper that the F{\o}lner sequences we deal with have this additional property.   Model examples include the integers $\Z$ (with $\Phi_n = [-n,n]$), and the (additive group of) countably infinite vector space $\Fw := \oplus \F \equiv \bigcup_n \F^n$ over a finite field $\F$ (with $\Phi_n = \F^n$).  For our initial discussion we will allow $G$ to be any countable abelian group, but we will eventually restrict\footnote{We will also need to consider the actions of various compact abelian groups, and in particular closed subgroups of the Pontryagin dual $\widehat{\Fw} \equiv \prod \F$ of $\Fw$.} our attention to the vector space $\Fw$.  However, it may be useful for future applications to note that several of the tools used here are in fact valid for arbitrary countable discrete\footnote{In fact, it seems likely that the hypothesis that $G$ be discrete could be dropped in much of the theory.  It may also be possible to generalize from abelian groups $G$ to nilpotent groups $G$.  We will not pursue these matters here.} abelian $G$.

This paper is concerned with the following seminorms for $G$-systems:

\begin{definition}[Gowers-Host-Kra uniformity seminorms, cf. \cite{hk-cubes}] \label{ghk-uni}   Let $G$ be a countable abelian  group, let $\X = (X,\B_X, \mu_X, (T_g)_{g \in G})$ be a $G$-system, let $\phi \in L^\infty(\X)$, and let $k \geq 1$ be an integer.  We define the \emph{Gowers-Host-Kra seminorm} $\|\phi\|_{U^k(\X)}$ of order $k$ of $\phi$ recursively by the formula
$$
 \| \phi\|_{U^1(\X)} := \lim_{n \to \infty} \| \E_{h \in \Phi^1_n} T_h \phi \|_{L^2(\X)}
$$
for $k=1$, and
$$ \| \phi\|_{U^k(\X)} := \lim_{n \to \infty} (\E_{h \in \Phi^k_n} \| \mder_h \phi \|_{U^{k-1}(\X)}^{2^{k-1}})^{1/2^k}$$
for $k \geq 1$, where for each $k$, $\Phi^k_1 \subset \Phi^k_2 \subset \ldots$ is a F{\o}lner sequence, and we use the expectation notation
$\E_{h \in H} f(h) := \frac{1}{|H|} \sum_{h \in H} f(h)$ for finite non-empty sets $H$ and functions $f: H \to \C$ throughout this paper.  
\end{definition}

\begin{remark} In the case of ergodic $G$-systems, we can use the mean ergodic theorem to simplify the $U^1$ norm as $\|\phi\|_{U^1(\X)} = |\int_X \phi\ d\mu_X|$.  In the non-ergodic case, the mean ergodic theorem gives the formula
$$ \|\phi\|_{U^1(\X)} = (\lim_{n \to \infty} \int_\X \E_{h \in \Phi^k_n} \mder_h \phi\ d\mu_X)^{1/2}$$
and more generally
$$ \|\phi\|_{U^k(\X)} = ( \lim_{n_k \to \infty} \ldots \lim_{n_1 \to \infty} \int_X \E_{h_k \in \Phi^k_{n_k}} \ldots \E_{h_1 \in \Phi^1_{n_1}} \mder_{h_k} \ldots \mder_{h_1} \phi\ d\mu_X )^{1/2^k}.
$$
The existence of the limits 
(independently of the choice of F{\o}lner sequences)
as well as similar integral formulae for the higher order Gowers-Host-Kra norms are given for $\Z$-actions in \cite{hk-cubes}, and for actions of a general countable abelian group in   Lemma \ref{welldefined}. One can also show that the Gowers-Host-Kra seminorms are indeed seminorms on $L^\infty(\X)$; see Lemma \ref{uk-basic}.
\end{remark}

\begin{example}\label{eigenf}  Let $\phi: X \to S^1$ be an \emph{eigenfunction} for an ergodic $G$-system, thus $\phi$ is measurable and $T_h \phi = \lambda(h) \phi$ for all $h \in G$ and some character $\lambda: G \to S^1$, where $S^1 := \{ z \in \C: |z|=1\}$ is the unit circle.  Then $\|\phi\|_{U^k(\X)} = 1$ for all $k \geq 2$.  If $\lambda$ is trivial, then $\|\phi\|_{U^1(\X)}=1$ as well, otherwise $\|\phi\|_{U^1(\X)} = 0$.
\end{example}

\begin{remark} In the case when $G$ is finite and $\X$ is just $G$ with normalized counting measure, the $\sigma$-algebra that consists of all subsets of $G$, and the translation action, the Gowers-Host-Kra seminorms simplify to the \emph{Gowers uniformity norms}
$$ \|f\|_{U^k(G)} = \left(\E_{x, h_1, \ldots, h_k \in G} \mder_{h_k} \ldots \mder_{h_1} f(x)\right)^{1/2^k}.$$
These norms (in the special case $G = \Z/N\Z$) were first introduced by Gowers in \cite{gowers},  where he derives quantitative bounds for  Szemer\'edi's theorem on arithmetic progressions in sets of positive upper density in the integers.  The above seminorms in the context of ergodic $\Z$-systems were introduced by Host and Kra in \cite{hk-cubes}, as a tool in the study of the ergodic averages related to Furstenberg's ergodic theoretic proof \cite{fu-szemeredi} of Szemer\'edi's theorem \cite{szemeredi}.   The Gowers uniformity norms for other finite abelian groups, such as finite-dimensional vector spaces $\F^n$ over a finite field $\F$, were studied in \cite{gt:inverse-u3}, \cite{sam}, \cite{green-tao-finfieldAP4s}, \cite{lms}, \cite{gt-ff-ratner}, \cite{gowwolf}.
\end{remark}

\subsection{Universal characteristic factors}

A fundamental concept in the study of the Gowers-Host-Kra uniformity seminorms is that of the \emph{universal characteristic factor} for such norms.  To describe this concept we need some notation.

\begin{definition}[Factors]\label{factor-def} A \emph{factor} $\Y = (Y, \B_Y, \mu_Y, (S_g)_{g \in G}, \pi^X_Y)$ of a $G$-system $\X = (X, \B_X, \mu_X, (T_g)_{g \in G})$ is another $G$-system $(Y, \B_Y, \mu_Y, (S_g)_{g \in G})$, together with a measurable \emph{factor map} $\pi^X_Y: X \to Y$, such that the push-forward $(\pi^X_Y)_* \mu_X$ of $\mu_X$ by $\pi^X_Y$ is equal to $\mu_Y$, and such that $\pi^X_Y \circ T_g = S_g \circ \pi^X_Y$ $\mu_X$-a.e. for all $g \in G$. 
We will often write  $(\Y, (S_g)_{g \in G}, \pi^X_Y)$, $(\Y, \pi^X_Y)$ or just $\Y$, for the factor 
$(Y, \B_Y, \mu_Y, (S_g)_{g \in G}, \pi^X_Y)$.
If $U$ is a measure space and $f: Y \to U$ is a measurable map, we write $(\pi^X_Y)^* f: X \to U$ for the pullback $(\pi^X_Y)^* f := f \circ \pi^X_Y$.  Conversely, if $f \in L^2(\X)$, we write $(\pi^X_Y)_* f \in L^2(\Y)$ for the pushforward of $f$, and $\E(f|\Y) := (\pi^X_Y)^* (\pi^X_Y)_* f \in L^2(\X)$ for the conditional expectation of $f$ to $\Y$.  We say that $f \in L^2(\X)$ is \emph{$\B_Y$-measurable} if $f = \E(f|\Y)$, of equivalently if $f = (\pi^X_Y) F$ for some $F \in L^2(\Y)$.  We refer to $(\X, \pi^X_Y)$ as an \emph{extension} of the $G$-system $(Y, \B_Y, \mu_Y, (S_g)_{g \in G})$.  If $Y$ is a point,  $(\pi^X_{\pt})_* f = \int_X f\ d\mu_X$.

One factor $\Y = (Y, \B_Y, \mu_Y,(S_g)_{g \in \G}, \pi^X_Y)$ of $\X$  is said to \emph{extend} another factor $\Y' = (\Y', \B_{\Y'}, \mu_{\Y'},(S'_g)_{g \in \G}, \pi^X_{Y'})$ of $\X$ (or equivalently, $\Y'$ is a \emph{sub-factor} of $\Y$) if every $\Y'$-measurable function is also $\Y$-measurable; in this case, we write $\Y \geq \Y'$.  Two factors are said to be \emph{isomorphic} if they extend each other.  Note that the notion of extension is a partial order modulo measure equivalence.  When we say that a factor is \emph{maximal} (resp. \emph{minimal}) with respect to some property, we mean that there is no extension (resp. sub-factor) of this factor, that obeys that property, which is not already equivalent to that factor.

We say that a factor $\Y = (Y,  \B_Y, \mu_Y,(S_g)_{g \in G}, \pi^X_Y)$ is \emph{generated} by a collection ${\mathcal F}$ of measurable functions $f: X \to \C$ if the $\sigma$-algebra $(\pi^X_Y)^{-1}(\B_Y)$ is generated (modulo $\mu_X$-null sets) by the pre-images of level sets $T_g f^{-1}(V)$ of functions $f \in {\mathcal F}$, where $g \in G$ and $V$ ranges over Borel subsets of $\C$.  Equivalently, $\Y$ is the minimal factor such that all functions in ${\mathcal F}$ are $\B_Y$-measurable.
\end{definition}

\begin{remark} Observe that any $\sigma$-algebra $\B \subset \B_X$ which is preserved by the $G$-action induces a factor $\Y = (X, \B, \mu_X, (T_g)_{g \in G},\operatorname{id})$ of 
$\X = (X, \B_X, \mu_X,(T_g)_{g \in G})$ (indeed, up to isomorphism, all factors arise in this manner, and we will often abuse notation by identifying factors with invariant $\sigma$-algebras).  In this case we see that $(\pi^X_Y)_* f = \E(f|\B)$ and $(\pi^X_Y)^* f = f$.  Note that if $X$ is separable modulo null sets, then $L^2(\X)$ is separable, hence on taking orthogonal projections $L^2(\Y)$ is separable, hence $Y$ is also separable modulo null sets.
\end{remark}

\begin{remark} Observe that any factor of an ergodic $G$-system is also ergodic.  The converse, of course, is not true.
\end{remark}

\begin{proposition}[Universal characteristic factor]\label{ucf-prop}  Let $G$ be a countable abelian group, let $\X$ be a $G$-system, and let $k \geq 1$.  Then there exists a factor 
$\Zcal_{<k} = \Zcal_{<k}(\X) = (Z_{<k}(\X), \B_{\Zcal_{<k}}, \mu_{\Zcal_{<k}}, (S_g)_{g \in G}, \pi^X_{Z_{<k}(\X)})$ of $\X$ with the property that for every $f \in L^\infty(\X)$, $\|f\|_{U^k(\X)} = 0$ if and only if $(\pi^X_{Z_{<k}(\X)})_* f = 0$ (or equivalently $\E(f|\Zcal_{<k}) = 0$).  This factor is unique up to  isomorphism.
\end{proposition}

\begin{proof} The uniqueness of $\Zcal_{<k}$ is clear; the existence follows immediately from Lemma \ref{softfactor}.
\end{proof}

\begin{remark}\label{wham} As $U^k$ is a seminorm on $L^{\infty}(\X)$, an equivalent characterisation of $\Zcal_{<k}$ is that it is the maximal factor for which $\| f\|_{U^k(X)} = \| \E(f|\Zcal_{<k}) \|_{U^k(X)}$ for all $f \in L^\infty(\X)$.  The factor $\Zcal_{<k}$ is also referred to as $\Zcal_{k-1}$ in the literature (and in particular in \cite{hk-cubes}).  From \eqref{mono-2} we have the monotonicity $\Zcal_{<j} \leq \Zcal_{<k}$ for $k \geq j$. 
\end{remark}

\begin{example}[Universal characteristic factors for small $k$]\label{smallk}  Let $\X$ be a $G$-system, then from the ergodic theorem we see that $\Zcal_{<1}(\X)$ is generated by the $G$-invariant functions on $\X$; in particular, for ergodic $G$-systems $Z_{<1}(X)$ is simply a point.  Some spectral theory also reveals (in the ergodic case) that $\Zcal_{<2}(\X)$ is the \emph{Kronecker factor} of $\X$, that is, the factor generated by the eigenfunctions of $\X$ (see Example \ref{eigenf}); this system is isomorphic 
 to $(H,\B,\mu,(T_g)_{g \in G})$ where $H$ is a closed subgroup of the Pontryagin dual $\hat G$ of $G$, 
 $\B$ the Borel $\sigma$-algebra, $\mu$ the Haar measure, and  the action of $G$  being given by a homomorphism from $G$ to $H$, acting on $H$ by translation.
\end{example}

In view of Proposition \ref{ucf-prop} and the decomposition $f = \E(f|\Zcal_{<k}) + (f - \E(f|\Zcal_{<k}))$, we see that any function $f \in L^\infty(\X)$ can be decomposed into a $\Zcal_{<k}$-measurable function, plus a function with vanishing $U^k$ norm.  When coupled with an explicit description of $\Zcal_{<k}$ (as was done for $k=1,2$ in Example \ref{smallk}), this decomposition leads to some highly non-trivial multiple recurrence and convergence theorems in ergodic theory: see for instance \cite{hk-cubes}, \cite{zieg-jams}, \cite{fhk}, \cite{fw}.  (See also \cite{gt-primes}, \cite{gt:inverse-u3}, \cite{green-tao-linearprimes}, \cite{gowwolf} for some finitary analogues of this decomposition, and some applications to additive combinatorics and analytic number theory.)
 
It is thus of interest to describe the universal characteristic factors $\Zcal_{<k}(\X)$ as explicitly as possible.  One particular class of functions related to such factors are the \emph{phase polynomials}:

\begin{definition}[Phase polynomials]\label{phase-def} Let $G$ be a countable discrete abelian group, $\X$ be a $G$-system, let $\phi \in L^\infty(\X)$, and let $k \geq 0$ be an integer. We say that $\phi$ is a \emph{phase polynomial of degree less than $k$} if we have $\mder_{h_1} \ldots \mder_{h_k} \phi = 1$ $\mu_X$-a.e.. for all $h_1,\ldots,h_k \in G$.  (In particular, setting $h_1=\ldots=h_k=0$, we see that phase polynomials must take values in the unit circle $S^1$ $\mu_X$-almost everywhere.)  We let $\Phase_{<k}(\X)$ denote the set of phase polynomials of degree less than $k$. 

We write $\Abr_{<k}(\X)$ for the factor of $\X$ generated by $\Phase_{<k}(\X)$, and say that $\X$ is an \emph{Abramov\footnote{It was Abramov who studied (under the name ``systems with quasi-discrete spectrum'' and for $\Z$-actions, see \cite{abramov}) systems of this type.} system of order $<k$} if $\X$ is ``generated'' by $\Phase_{<k}(\X)$, or equivalently if $\Phase_{<k}(\X)$ spans $L^2(\X)$.
\end{definition}

\begin{example} $\Phase_{<0}(\X)$ consists only of the constant function $1$, so the only Abramov system of order $<0$ is a trivial (one point) system.  $\Phase_{<1}(\X)$ consists of the $G$-invariant functions from $\X$ to $S^1$ (which, in the ergodic case, are just the constants), and only the Abramov systems of order $<1$ are those for which the action of $G$ is trivial.  In the ergodic case, $\Phase_{<2}(\X)$ consists of the eigenfunctions from Example \ref{eigenf}, and so the ergodic Abramov systems of order $<2$ are precisely the Kronecker systems (i.e. systems generated by translations on compact abelian groups).  There is an analogous relationship between higher degree phase polynomials and higher order eigenfunctions of $\X$.  Observe that every phase polynomial $\phi \in \Phase_{<k}(\X)$ takes the form $\phi = e(P)$, where $e: \R/\Z \to S^1$ is the standard character $e(x) := e^{2\pi i x}$, and $P: X \to \R/\Z$ is a \emph{polynomial of degree $<k$} in the sense that $\ader_{h_1} \ldots \ader_{h_{k}} P = 0$ for all $h_1,\ldots,h_{k} \in G$.    
\end{example}

The following observations are immediate:

\begin{lemma}[Trivial facts about phase polynomials]\label{trivpoly-lem} Let $G$ be a countable discrete abelian group, $\X$ be a $G$-system, and let $k \geq 0$.
\begin{itemize}
\item[(i)] (Monotonicity) We have $\Phase_{<k}(\X) \subseteq \Phase_{<k+1}(\X)$.  In particular, $\Abr_{<k}(\X) \leq \Abr_{<k+1}(\X)$, and an Abramov system of order $<k$ is also an Abramov system of order $<k+1$.
\item[(ii)] (Homomorphism) $\Phase_{<k}(\X)$ is an abelian group under pointwise multiplication, and for each $h \in G$, $\mder_h$ is a homomorphism from $\Phase_{<k+1}(\X)$ to $\Phase_{<k}(\X)$.
\item[(iii)] (Polynomiality criterion) Conversely, if $\phi: X \to \C$ is a measurable function such that $\mder_h \phi \in \Phase_{<k}(\X)$ for all $h \in G$, then $\phi \in \Phase_{<k+1}(\X)$.
\item[(iv)] (Functoriality) If $\Y$ is a factor of $\X$, then the pullback map $(\pi^Y_X)^*$ is a homomorphism from $\Phase_{<k}(\Y)$ to $\Phase_{<k}(\X)$.  Conversely, if $f: \Y \to \C$ is such that $(\pi^Y_X)^* f \in \Phase_{<k}(\X)$, then $f \in \Phase_{<k}(\Y)$.
\end{itemize}
\end{lemma}

It is not hard to show that every phase polynomial of degree $<k$ is in fact $\Zcal_{<k}(\X)$-measurable (see Lemma \ref{easy-thm}); thus $\Abr_{<k}(\X) \leq \Zcal_{<k}(\X)$.  However, in the case of $\Z$-systems, the characteristic factor $\Zcal_{<k}(\X)$ also contains some functions which do not arise from phase polynomials, even when one assumes ergodicity; this fact was essentially first observed by Furstenberg and Weiss \cite{fw-char}.  Indeed, it is not too difficult to show (see \cite{hk-cubes}) that any factor $\Y = (N/\Gamma, (x \mapsto a^g x)_{g \in \Z}, \pi^X_Y)$ of $\X$ which is a \emph{$<k$-step nilsystem}, thus $N$ is a nilpotent Lie group of step $<k$, $\Gamma$ is a discrete cocompact subgroup, $N/\Gamma$ is given Haar measure, and $a \in N$ is a group element, is a sub-factor of $\Zcal_{<k}(\X)$.  The converse statement is much deeper, and is due to Host and Kra \cite{hk-cubes}:

\begin{theorem}[Description of $\Zcal_{<k}(\X)$ for ergodic $\Z$-systems]\label{hk-thm}\cite[Theorem 10.1]{hk-cubes} Let $\X$ be an ergodic $\Z$-system.  Then $\Zcal_{<k}(\X)$ is the minimal factor that extends all $<k$-step nilsystem factors of $\X$.  Indeed, $\Zcal_{<k}(\X)$ is itself the inverse limit of $<k$-step nilsystems.
\end{theorem}

\begin{remark} In the case of ergodic $\Z$-systems, every Abramov system of order $<k$ is the inverse limit of $<k$-step nilsystems (this is implicit from \cite{hk-cubes}).  However, the converse is not true\footnote{For instance, the Heisenberg system $H(\R)/H(\Z)$ discussed in Section \ref{heisen} is a $2$-step nilsystem which is not an Abramov system of any order, if the underlying shifts $\alpha, \beta$ and $1$ are independent over $\Q$.}:  see \cite{fw-char} for further discussion.
\end{remark}

\begin{remark} The finitary counterpart to Theorem \ref{hk-thm}, where $G$ and $X$ are $\Z/N\Z$,  is known as the \emph{inverse conjecture for the Gowers norm} for cyclic groups $\Z/N\Z$, and would have a number of applications to additive combinatorics and analytic number theory, see e.g. \cite{green-tao-linearprimes}.  It is currently only proven for $k \leq 3$ \cite{gt:inverse-u3}.
\end{remark}

\subsection{Main result}

In view of Theorem \ref{hk-thm}, it is natural to ask what the universal characteristic factors $\Zcal_{<k}(\X)$ are for ergodic\footnote{The non-ergodic case can then be recovered, at least in principle, from the ergodic case by the ergodic decomposition, though we will not attempt to do so here.} $G$-systems, when $G$ is an abelian group other than the integers $\Z$.  

The main results of this paper give a sharp answer to this question in the case when $G$ is the (additive group of the) countably infinite vector space $G = \Fw$ over a finite field $\F$ of characteristic $\ge k$ which we will refer to as the {\em high characteristic} case. We also partially answer this question for general $k$:
 
\begin{theorem}[Sharp description of $\Zcal_{<k}(\X)$ for ergodic $\Fw$-systems in high char]\label{main-thm-high}  Let $\F$ be a finite field, and let $\X$ be an ergodic $\Fw$-system.  Let $k \leq \charac(\F)$. Then for each $k \geq 1$, we have $\Abr_{<k}(\X) =\Zcal_{<k}(\X)$.
\end{theorem}
 
\begin{theorem}[Partial description of $\Zcal_{<k}(\X)$ for ergodic $\Fw$-systems in low char]\label{main-thm}  Let $\F$ be a finite field, and let $\X$ be an ergodic $\Fw$-system.  Then for each $k \geq 1$, we have
$\Abr_{<k}(\X) \leq \Zcal_{<k}(\X) \leq \Abr_{<C(k)}(\X)$
for some $C(k)$ depending only on $k$.
\end{theorem}

\begin{remark} We also have a slightly more precise structural description of $\Zcal_{<k}(\X)$, as a tower of abelian extensions by polynomial cocycles; see Theorems \ref{struc-thm}, \ref{struc-high}.  The quantity $C(k)$ can in principle be computed explicitly from the proof of Theorem \ref{main-thm}, but we have not sought to obtain the best possible value of $C(k)$.  We believe in fact that $\Zcal_{<k}(\X)$ should equal $\Abr_{<k}(\X)$ for all $k$ (not just in the high characteristic case $k \leq \charac(\F)$).
\end{remark}

From Theorems \ref{main-thm}, \ref{main-thm-high} and Proposition \ref{ucf-prop} we have the following immediate corollaries:

\begin{corollary}[Ergodic inverse Gowers conjecture - high characteristic]\label{main-cor-high}  Let $\F$ be a finite field, and let $\X$ be an ergodic $\Fw$-system.  Let $1 \leq k \leq \charac(\F)$ and $f \in L^\infty(\X)$ be such that $\|f\|_{U^k(\X)} > 0$.  Then there exists $\phi \in \Phase_{<k}(\X)$ such that $\int_X f \overline{\phi}\ d\mu_X \neq 0$.
\end{corollary}

\begin{corollary}[Partial ergodic inverse Gowers conjecture]\label{main-cor}  Let $\F$ be a finite field, and let $\X$ be an ergodic $\Fw$-system.  Let $k \geq 1$ and $f \in L^\infty(\X)$ be such that $\|f\|_{U^k(\X)} > 0$.  Then there exists $\phi \in \Phase_{<C(k)}(\X)$ such that $\int_X f \overline{\phi}\ d\mu_X \neq 0$.
\end{corollary}

In a companion paper \cite{tz-correspondence}, we will combine Corollary \ref{main-cor-high} with a version of the Furstenberg correspondence principle, as well as the equidistribution theory in \cite{gt-ff-ratner}, to obtain a finitary counterpart to this theorem:

\begin{theorem}[Inverse theorem for the Gowers norm over finite fields in high characteristic \cite{tz-correspondence} ]\label{inv-thm} Let $\F$ be a finite field of characteristic $p$, let $1 \leq k \leq \charac(\F)$  be an integer, and let $\delta > 0$.  Then there exists $c = c(p,k,\delta) > 0$ such that for finite-dimensional vector space $G$ over $\F$ and any function $f: G \to \C$ with $\|f\|_{L^\infty(G)} \leq 1$ and $\|f\|_{U^k(G)} \geq \delta$, one has $|\E_{x \in G} f(x) \overline{\phi(x)}| \geq c$ for some $\phi \in \Phase_{<k}(G)$.
\end{theorem}

\begin{remark} We conjecture that this result should in fact hold without the restriction on the characteristic  depending on $k$.  Once this restriction is removed, it becomes important here that the values of the phase polynomial $\phi$ are allowed to range freely in the unit circle $S^1$.  If one constrains $\phi$ to take values in the $p^\th$ roots of unity $C_p$, then the claim can fail for small $p$, as first observed in \cite{lms}, \cite{gt-ff-ratner}.  However, such examples do not obstruct Theorem \ref{inv-thm} from holding when $\phi$ takes values in $S^1$ (see \cite{tz-correspondence} for further discussion).  By using Corollary \ref{main-cor} instead of Corollary \ref{main-cor-high}, one can obtain a partial analogue of Theorem \ref{inv-thm} in the low characteristic case $k > \charac(\F)$, in which $\Phase_{<k}(G)$ is replaced by $\Phase_{<C(k)}(G)$; see \cite{tz-correspondence}.
\end{remark}

\begin{remark} Theorem \ref{main-thm-high} should also allow one (assuming sufficiently high characteristic) to obtain a formula for the limit of multiple ergodic averages of quantities such as $c(g) := \mu( A \cap T^g A \cap \ldots \cap T^{(k-1) g} A )$ (as in \cite{zieg-jams}), and to be able to show that $c(g)$ can be approximated by a function of polynomials in $g$, in the spirit of the results in \cite{bhk}.  We hope to report on these and other applications in a subsequent paper.
\end{remark}

\subsection{The Heisenberg example}\label{heisen}

To illustrate the above results we now pause to describe the model case of a \emph{Heisenberg system}.  (The discussion in this section is not directly used in the remainder of the paper.)
To simplify the discussion we restrict attention to the $k=3$ case.

We first review the more familiar case of $\Z$-systems.  For any commutative ring $R$, let $H(R)$ be the Heisenberg group
\[ 
   H(R) :=
   \left( \begin{smallmatrix}
   1 & R & R \\
   0 &          1 & R \\
   0 & 0          &          1 \\
   \end{smallmatrix}\right).\]
This is clearly a $2$-step nilpotent group.  The quotient space $H(\R)/H(\Z)$ is then a $2$-step nilmanifold that has a natural Haar measure.  For fixed $\alpha, \beta \in \R$, the function $\psi: \Z \to H(\R)$ defined by
\begin{equation}\label{psin}
\psi(n) :=  \left( \begin{smallmatrix}
   1 & n\alpha & n \gamma+\binom{n}{2}\alpha\beta \\
   0 &          1 & n \beta \\
   0 & 0          &          1 \\
   \end{smallmatrix}\right)
\end{equation}
can be easily checked to be a homomorphism $\Z \to H(\R)$, and thus defines an action $(T_g)_{g \in \Z}$ on $H(\R)/H(\Z)$; if $1, \alpha, \beta$ are linearly independent over $\Q$, one can show that this action is ergodic (see e.g. \cite{parry}). The function $f: H(\R) \to \C$ defined by
\begin{equation}\label{heisenberg2}
f   \left( \begin{smallmatrix}
   1 & x & z \\
  0 & 1 & y \\
   0 &   0        &          1 \\
   \end{smallmatrix}\right) = e(z-\{x\}y),
\end{equation}
where $\{x\} := x - \lfloor x\rfloor$ is the fractional part of $x$, induces  a function $f: H(\R)/H(\Z) \to \C$, with
\begin{equation}\label{anh}
f(a^n x_0)=  e(\binom{n}{2}\alpha\beta+n\gamma-\{n\alpha\}n \beta)
\end{equation}
where $x_0 := H(\Z) \in H(\R)/H(\Z)$ (one can easily check that this function is well defined as a function on  $H(\R)/H(\Z)$).

The argument on the right-hand side of \eqref{anh} is an example of a  {\em generalized polynomial} (see \cite{berg-leib-poly}). It can be shown that the function $f$ is asymptotically orthogonal to all phase polynomials of degree $<3$ (in fact, it is asymptotically orthogonal to phase polynomials of all degrees), but that  $\|f\|_{U^3(\X)}>0$.  For the system $\X=(H(\R)/H(\Z),(T_g)_{g \in \Z})$, the algebra generated by the function $f$ and the eigenfunctions of the $\Z$ action (which take the form $e(ax+by+\theta)$ for integers $a,b$ and $\theta \in \R/\Z$) is dense in $L^2(\X)$. 

One can imitate this construction for an $\Fw$-action, for $\F$ with $\charac\F >2$ . The analogues of the ring of integers $\Z$ and the field of reals $\R$ will be the ring $\F[t] := \{ \sum_{i=0}^m a_i t^i: a_i \in \F \}$ of polynomials in one variable and the field $\F((t)) := \{ \sum_{i=-\infty}^m a_i t^i: a_i \in \F \}$ of Laurent polynomials of one variable, respectively. The analogue of the torus $\R/\Z$ is the abelian group $\F((t))/\F[t]$.  If $x = \sum_{i=-\infty}^m a_i t^i$ is an element of $\F((t))$, we write $\lfloor x \rfloor := \sum_{i=0}^m a_i t^i \in \F[t]$ for the ``integer part'', and $\{x\} := x - \lfloor x \rfloor$ for the ``fractional part''.

Observe that \emph{as an additive group}, $\F[t]$ is isomorphic to $\Fw$, and so an ergodic $\F[t]$-system is also an ergodic $\Fw$-system.  Consider the homogeneous space $X := H(\F((t)))/H(\F[t])$; this is a ``manifold'' over the base field $\F((t))$ and comes with a natural Haar measure. For fixed $\alpha, \beta \in \F((t))$, the function $\psi: \F[t] \to H(\F((t)))$ defined by \eqref{psin} is still a homomorphism, and thus defines an action $(T_g)_{g \in \F[t]}$ on $X$ as before.  If we define $f: H(\F((t))) \to \C$ by \eqref{heisenberg2}, then this function again induces a  well defined function on $X$ and we have the formula \eqref{anh} as before with $x_0 := \F[t] \in X$.  

Thus far, the $\Fw$-action case has proceeded in exact analogy with the $\Z$-action case.  However, the key difference between the two case is that the phase $P(n) := n^2\alpha\beta+n\gamma-\{n\alpha\}n \beta$ on the right-hand side of \eqref{anh} is not just a generalised polynomial - it is a genuine polynomial, indeed one easily checks the identity $\ader_{h_1} \ader_{h_2} \ader_{h_3} P = 0$ for all $h_1,h_2,h_3 \in \F[t]$.  (This ultimately stems from the fact that the maps $x \mapsto \lfloor x \rfloor$ and $x \mapsto \{x\}$ are genuine homomorphisms in the $\Fw$-action case (where there is no ``carry'' operation), whereas they are only ``approximate'' homomorphisms in the integer case.)  Thus the system $\X = (X, (T_g)_{g \in \F[t]})$ is isomorphic to the system $((\T(\F))^2 \times \T(\F),(T_g)_{g \in \Fw})$, where $\T(\F) := \F^{\Z} \equiv \F((t))/\F[t]$ is the Pontryagin dual $\hat{\Fw}$ of $\Fw$, and the $\Fw$-action is given by 
 $$T_g((x,y),z) = ((x,y)+(g\alpha,g\beta), z+g\gamma+\binom{g}{2}\alpha\beta+[g\alpha]y-g\beta\{x\}-\{g\alpha\} g\beta)$$
 for $x,y \in \T(\F) \equiv \F((t))/\F[t]$ and $g \in \Fw \equiv \F[t]$.  For any character $\chi$ of $\T(\F)$, $\chi \in \widehat{\T}(\F)$, the function $z \mapsto \chi(z)$ lies in $\Phase_2(\T(\F))$.   
By Fourier decomposition, any function in $L^2(\X)$ can be written as $\sum_{\chi \in \widehat{\T}(\F)} f_{\chi}(x,y)\chi(z)$, where $f_{\chi}(x,y) \in L^2( \T(\F)^2)$.
Since for any $\chi$, the function $f_{\chi}(x,y)$ is defined on a Kronecker system (arising from the eigenfunctions $(x,y) \mapsto e( ax+by )$ for $a,b \in \F[t]$), we have $f_{\chi} \in \Zcal_{<2}(\X)$.  It follows that $\X$ is an Abramov system of order $<3$ and is thus equal to $\Zcal_{<3}(\X)$ by Theorem \ref{main-thm}.  This latter fact can also be checked  by calculating the Gowers-Host-Kra seminorms directly.  The reader may also wish to verify Theorem \ref{struc-thm} explicitly for this example.

The main point of the example above is that it demonstrates the peculiarities of the polynomial aspect of nilpotency in the case of an $\Fw$-action.

\subsection{Overview of the structure of the paper}

\begin{figure}
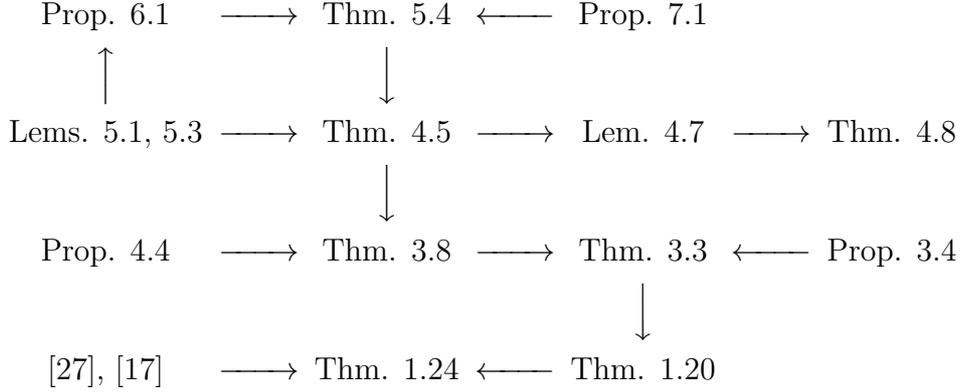
\label{logic}
$$
\begin{CD}
\hbox{Prop. \ref{fin-prop}} @>>>  \hbox{Thm. \ref{main-6}} @<<< \hbox{Prop. \ref{propfin}} @. \\
@AAA                                  @VVV                              @.                     @. \\
\hbox{Lems. \ref{descent}, \ref{vert-lem}} @>>> \hbox{Thm. \ref{main-4}} @>>> \hbox{Lem. \ref{torlemma}} @>>> \hbox{Thm. \ref{struc-thm}} \\
@.   @VVV @.  @. \\
\hbox{Prop. \ref{ext-type-k}} @>>> \hbox{Thm. \ref{main-3}} @>>> \hbox{Thm. \ref{main-2}} @<<< \hbox{Prop. \ref{typek-prop}} \\
@. @. @VVV @.\\
\hbox{\cite{tz-correspondence}, \cite{gt-ff-ratner}} @>>> \hbox{Thm. \ref{inv-thm}} @<<< \hbox{Thm. \ref{main-thm}} @.
\end{CD}
$$
\caption{Logical dependencies in the proofs of key theorems in the paper in the general characteristic case (ignoring the material in the appendices, which are used throughout the paper).  Theorem \ref{struc-thm} is also used inductively to establish Propositions \ref{fin-prop}, \ref{propfin}.  The high-characteristic case follows a similar logic.}  
\end{figure}

The bulk of the paper will be devoted to the proof of Theorems \ref{main-thm}, \ref{main-thm-high}.  
We first prove Theorem \ref{main-thm}, and then specialize to the high characteristic case.
The proof will be established via a series of reductions.  The main task is to establish the inequality
$\Zcal_{<k}(\X) \leq \Abr_{<C(k)}(\X)$
for some $C(k)$ depending on $k$, which will imply that $\Zcal_{<k}(\X)$ is generated by polynomials of high degree.  Thus we do not need to keep careful track of the degree of the polynomials that arise in our analysis, except to ensure that they are bounded.  The next few steps closely follow the approach of Host and Kra \cite{hk-cubes}.  Namely, by following the methods in \cite{hk-cubes}, one can reduce matters to understanding a certain type of cocycle $f: \Fw \times X \to S^1$ on a structured type of ergodic $\Fw$-system $\X$; specifically, one needs to show that any $(\Fw,\X,S^1)$-cocycle of type\footnote{All definitions for the terms used here are given later in the paper, when these results are formalized.} $<k$ on an ergodic $\Fw$-system of order $<k$ (see Definition \ref{sysk-def}) is cohomologous to a $(\Fw,\X,S^1)$-phase polynomial; see Theorem \ref{main-4}.   As mentioned earlier, the degree of this polynomial is not too important so long as it is bounded.

By an inductive hypothesis one can understand the underlying system $\X$ reasonably well; it turns out to be a finite  tower of abelian group extensions, each of which is given by a cocycle which is cohomologous to a polynomial.

To proceed further one needs to understand the condition that a cocycle $f: \Fw \times \X \to S^1$ has bounded type.  By definition, this means that a certain ``iterated derivative'' $d^{[k]} f$ of that cocycle is a coboundary.  It turns out that the underlying system $\X$ can be expressed as an abelian extension $\X = \Y \times_\rho U$ of a simpler system $\Y$, where $U = (U,\cdot)$ is a compact abelian group.  The group $U$ then acts freely on $\X$ by the action $V_u: (y,v) \mapsto (y,uv)$, with this action commuting with the action of $\Fw$.  One can  differentiate the cocycle $f$ repeatedly in the ``vertical'' direction by using the operation $\mder_u f(g,x) := f(g, V_u x) / f(g,x)$.  It turns out that each such vertical differentiation reduces the type of the cocycle; iterating this, one concludes that one can find an $m$ such that $\mder_{t_1} \ldots \mder_{t_m} f$ is a cocycle of type $<0$ (i.e. a coboundary) for every $t_1,\ldots,t_m \in U$.  In other words, one has an equation of the form
\begin{equation}\label{mder-iter}
 \mder_{t_1} \ldots \mder_{t_m} f(g,x) = \mder_g F_{t_1,\ldots,t_m}(x) := \frac{ F_{t_1,\ldots,t_m}( T_g x ) }{ F_{t_1,\ldots,t_m}( x ) }
\end{equation}
for some functions $F_{t_1,\ldots,t_m}(x)$.  A key technical point is that while the function $F_{t_1,\ldots,t_m}(x)$ is \emph{a priori} only measurable in $x$, it can be made to be measurable in the parameters $t_1,\ldots,t_m$ also (see Lemma \ref{measurable_choice}).  This will be rather important for us as we will be relying quite heavily on the measurability property\footnote{For instance, we will need a variant of the classical Steinhaus theorem that asserts that if $A$ is a measurable subset of a compact abelian group $U$ with positive measure, then the difference set $A-A$ contains a neighborhood of the origin (cf. Lemma \ref{tor-lem}).  Curiously, analogous results are exploited in the additive-combinatorial approach to the Gowers inverse problem (see e.g. \cite{gt:inverse-u3}), where they go by the name of ``Bogolyubov-type lemmas''.} in our arguments.

We would like to use the equation \eqref{mder-iter} to show that $f$ itself is a coboundary, up to a polynomial error.  The idea is to ``integrate'' the derivatives $\mder_{t_1},\ldots,\mder_{t_m}$ one at a time.  At any intermediate stage of this process, one will have obtained another function $\tilde f: \Fw \times \X \to S^1$ which differs from $f$ (multiplicatively) by a phase polynomial, and which obeys an equation of the form
\begin{equation}\label{mderj}
\mder_{t_1} \ldots \mder_{t_j} \tilde f(g,x) = c_{t_1,\ldots,t_j}(g,x) \mder_g F_{t_1,\ldots,t_j}(x)
\end{equation} 
for some $0 \leq j \leq m$, some functions $F_{t_1,\ldots,t_j}(x)$, some phase polynomials $c_{t_1,\ldots,t_j}$, any $t_1,\ldots,t_j \in U$, $g \in G$, and almost all $x \in X$.  A technical point here is that whereas the original function $f$ was a cocycle in $\Fw$, the new function $\tilde f$ need not be; however, it turns out that the $\Fw$-cocycle property\footnote{Indeed, there seem to be significant technical difficulties if one attempts to work purely within the world of cocycles, stemming ultimately from the fact that $\Fw$ is not finitely generated.  For instance, it appears quite difficult to show that the abstract nilpotent group ${\mathcal G}^{[k]}$ defined in \cite{hk-cubes} acts transitively in the non-finitely generated case.} is not actually needed in the rest of our analysis; indeed, as we shall shortly see, we can rely primarily on the cocycle behavior in $t_j$ instead. 

The left hand side of \eqref{mderj} exhibits some ``linearity'' in $t_j$, thanks to the cocycle equation $\mder_{t_j + t'_j} f = (\mder_{t_j} V_{t'_j} f) \mder_{t'_j} f$.  This induces some approximate linearity properties on $F_{t_1,\ldots,t_j}$.  One can conjugate $F_{t_1,\ldots,t_j}$ to upgrade this approximate linearity to genuine linearity (here we crucially exploit the measurability of all our objects with respect to $t_1,\ldots,t_j$), at least for $t_1,\ldots,t_j$ in an open subgroup $U'$ of $U$ (here we rely heavily on the finite characteristic of $\Fw$, which forces $U$ to be a torsion group, so that every open neighborhood of $U$ contains an open subgroup).  It still remains to handle the behavior in the quotient group $U/U'$, but this is a finite abelian group and can be worked out explicitly, splitting this group as the product of finite cyclic groups and ``straightening'' the various actors in \eqref{mderj} along each such group.  At the end of the day, one is able to ``integrate'' and remove one of the derivatives in \eqref{mderj}, at the cost of introducing an additional $U$-invariant factor on the right-hand side (cf. the classical ``$+C$'' ambiguity when solving an equation $\frac{d}{dx} f = F$ by integration).  However, this additional factor (which basically lives on $Y$ rather than on $X$) can be shown to be cohomologous to a polynomial by an induction hypothesis, and so can be absorbed into the other terms on the right-hand side.  One can then iterate this procedure until the number $j$ of derivatives in \eqref{mderj} reaches zero, at which point we obtain the desired characterization of $f$.  

The high characteristic case $k \leq \charac(\F)$ (i.e. Theorem \ref{main-thm-high}) follows a similar logic, but one has to be much more careful in keeping track of the degrees and types of the various objects that arise in the proof, thus requiring ``exact'' counterparts of many of the lemmas used in the above analysis.  The high characteristic hypothesis is used to ensure that all the polynomials which appear take values in a coset of the cyclic group $C_p$, and have constant ``integral'' along various lines (see Lemma \ref{L:values-in-F_p} for a more precise statement.)

\begin{remark} As mentioned above, our approach largely follows that of Host and Kra \cite{hk-cubes}, although the abstract nilpotent structure group ${\mathcal G}^{[k]}$, which plays a central role in \cite{hk-cubes}, is not used explicitly in this paper (although of course we will be exploiting several symmetries of the cubic measure $\mu^{[k]}$, particularly with respect to ``vertical rotations'', which can be viewed as special elements of this structure group). Furthermore, the heart of the Host-Kra argument
(the ``lifting'' proposition in \cite[Proposition 10.10]{hk-cubes}) does not have an equivalent in our context due to the fact that
the group $\Fw$ is not finitely generated.
As discussed above, the emphasis is instead on solving various equations of ``Conze-Lesigne'' type\cite{cl}.  In this respect, the arguments here share some features in common with \cite{zieg-jams} as well as \cite{hk-cubes}. 
\end{remark}

\begin{remark}
One curious distinction between the finite characteristic case and the $\Z$-action case is that in the former case (and assuming sufficiently high characteristic), the systems $Z_k(\X)$ are toral systems rather than inverse limits of such systems. 
\end{remark}

\subsection{Acknowledgements}

The authors would like to thank Tim Austin, Ben Green, Bernard Host, Bryna Kra, and Trevor Wooley for many enlightening conversations and suggestions.  The first author is supported by NSF grant DMS-0600042. The second author is supported by a grant from the MacArthur Foundation, and by NSF grant DMS-0649473. The first and third authors are supported by BSF grant  No. 2006094. The third author is supported
by a Landau fellowship of the Taub foundations, and by an Alon fellowship.   The authors also thank the anonymous referees for many useful suggestions and corrections.

\section{Abelian cohomology}\label{abcom0}

Throughout the paper we will be relying heavily on the language of abelian cohomology of dynamical systems.  We record the key definitions here; further discussion of these concepts can be found in Appendix \ref{abcom}.

\begin{definition}[Abelian cohomology]\label{cohom} Let $G$ be a countable discrete abelian group, let $\X=(X,\B_X,\mu_X, (T_g)_{g \in G})$ be a $G$-system, and let $(U,\cdot)$ be a compact abelian group. 
\begin{itemize}
\item We let $M(\X,U)$ denote the set of all measurable functions $\phi: X \to U$, with two functions $\phi, \phi': \X \to U$ identified if they agree $\mu_X$-almost everywhere; this is an abelian group under pointwise multiplication.  We refer to elements of $M(\X,U)$ as \emph{$(\X,U)$-functions}.
\item We let $M(G,\X,U)$ denote the set of all measurable functions $\rho: G \times X \to U$, with two functions $\rho, \rho': G \times \X \to U$ identified if $\rho(g,x) = \rho'(g,x)$ for all
$g \in G$ and almost every  $x \in X$; this is an abelian group under pointwise multiplication.  We refer to elements of $M(G,\X,U)$ as \emph{$(G,\X,U)$-functions}.
\item We let $Z^1(G,\X,U)$ denote the subgroup of $M(G,\X,U)$ consisting of those $(G,\X,U)$-functions that obey the cocycle equation
$\rho(g+g', x) = \rho(g,T_{g'} x) \rho(g', x)$
for all  $g,g' \in G$ and almost every $x \in X$.  We refer to elements of $Z^1(G,\X,U)$ as \emph{abelian $(G,\X,U)$-cocycles}, or \emph{cocycles} for short.
\item If $\rho$ is an abelian $(G,\X,U)$-cocycle, we define the \emph{abelian extension} $\X \times_\rho U$ of $\X$ by $\rho$ to be the product space $(X \times U, \B_X \times \B_U, \mu_X \times \mu_U)$ with shift maps $(x,u) \mapsto (T_g x, \rho(g,x) u)$ for $g \in G$, where $\mu_U$ is normalised Haar measure on $U$.  Note that this is indeed an extension of $\X$, with the obvious factor map $\pi^{\X \times_\rho U}{\X}: (x,u) \mapsto x$.  
\item If $F$ is a $(\X,U)$-function, we define the \emph{derivative} $\mder F$ of $F$ to be the $(G,\X,U)$ function $\mder F(g,x) := \mder_g F(x) = F(T_g x) / F(x)$.  We refer to $F$ as an \emph{antiderivative} of $\mder F$.  We write $B^1(G,\X,U) = \mder M(\X,U)$ denote the space of all derivatives; this is a subgroup of $Z^1(G,\X,U)$.  We refer to elements of $B^1(G,\X,U)$ as \emph{$(G,\X,U)$-coboundaries}, or \emph{coboundaries} for short.
\item More generally, if $B$ is a $G$-invariant subset of $\X$, we say that a $(G,\X,U)$-function $\rho$ is a \emph{$(G,B,U)$-coboundary} if there exists some measurable $F: B \to U$ such that $\rho(g,x) = \mder_g F(x) = \frac{F(T_g x)}{F(x)}$ 
for all $g \in G$ and $\mu_X$-almost every $x \in B$.  We refer to $F$ as an \emph{antiderivative} of the coboundary.
\item We say that two $(G,\X,U)$-functions $\rho, \rho'$ are \emph{$(G,\X,U)$-cohomologous} (or \emph{cohomologous} for short), if $\rho/\rho' \in B^1(G,\X,U)$.  Note that we do not require $\rho, \rho'$ to be cocycles here, though clearly any function cohomologous to a cocycle is again a cocycle.
\end{itemize}
\end{definition}

\begin{remark}\label{measure-equiv} Observe that if $\rho$ and $\tilde \rho$ are cohomologous, then $\X \times_\rho U$ and $\X \times_{\tilde \rho} U$ are measure-equivalent systems.  Thus, from the perspective of measure equivalence, $(X,G,U)$ cocycles $\rho$ are only determined up to their representative $[\rho]_{G,X,U}$ in the cohomology group $H^1(G,\X,U) := Z^1(G,\X,U)/B^1(G,\X,U)$.  As usual we have the short exact sequence
$$
\begin{CD}
0 @>>> B^1(G,\X,U) @>>> Z^1(G,\X,U) @>>>  H^1(G,\X,U) @>>> 0.
\end{CD}
$$
\end{remark}

\begin{remark} We will primarily be working with $(G,\X,U)$-cocycles, but for technical reasons related to the fact that $\Fw$ is not countably generated, we will also need to work in the more general setting of $(G,\X,U)$-functions.  In practice, however, these functions will be ``close'' to coboundaries in various senses (for instance, they may differ from a coboundary by a polynomial function).
\end{remark}

\begin{remark} We caution that abelian extensions of ergodic $G$-systems are not necessarily ergodic.
\end{remark}

We will need several technical results concerning abelian cohomology groups, which we have collected in Appendix \ref{abcom}, and which we will refer to as necessary in the main text of the paper.

\section{Reduction to abelian extensions of order $<k+1$}

Recall the definition of $\Zcal_{<k}(\X)$ from Proposition \ref{ucf-prop}. 
Before starting the proof of Theorem \ref{main-thm}, we make a basic definition:

\begin{definition}[System of order $<k$]\label{sysk-def} Let $k \geq 1$, and let $G$ be a countable discrete abelian group.  A $G$-system $\X$ is said to be \emph{of order $<k$} if $\Zcal_{<k}(\X) = \X$.
\end{definition}

\begin{example} Trivial systems are of order $<1$; Kronecker systems are of order $<2$.  From \eqref{mono-2} we see that for any system $\X$, $\Zcal_{<k}(\X)$ is of order $<k$, and that any system of order $<k$ is automatically of order $<k+1$.
\end{example}

By Theorem \ref{easy-thm}, every Abramov system of order $<k$ is also a $G$-system of order $<k$.  This and \eqref{mono-2}, will allow us to immediately derive Theorem \ref{main-thm}  from the following claim:

\begin{theorem}[First reduction]\label{main-2}
Let $\F$ be a finite field, let $k \geq 1$, and let $\X$ be an ergodic $\Fw$-system of order $<k$.  Then $\X$ is an Abramov system of order\footnote{Here and in the sequel, we use $O_{k}(1)$ to denote any quantity bounded by $C(k)$ for some constant $C(k)$ depending only on $k$, and similarly for other choices of subscripts in the $O()$ notation.} $<O_{k}(1)$.  
\end{theorem}

To prove Theorem \ref{main-2}, we will in fact prove a more precise statement.  We will use the language of abelian cohomology, in particular the notions of $(G,\X,U)$-functions, (abelian) $(G,\X,U)$-cocycles, $(G,\X,U)$-coboundaries, and extensions $\X \times_\rho U$ of a $G$-system $\X$ by a cocycle $\rho: G \times X \to U$, where $U$ is a compact abelian group; see Definition \ref{cohom} for full details.

The following basic fact was established by Host and Kra\cite{hk-cubes}:

\begin{proposition}[Order $<k+1$ systems are abelian extensions of order $<k$ systems]\label{typek-prop}\cite[Proposition 6.3]{hk-cubes}  Let $G$ be a discrete countable abelian group, let $k \geq 1$, and let $\X$ be an ergodic $G$-system of order $<k+1$.  Then $\X$ is an abelian extension $\X \equiv \Zcal_{<k}(\X) \times_\rho U$ of the order $<k$ system $\Zcal_{<k}(\X)$ for some compact abelian group\footnote{All topological groups in this paper are assumed to be Hausdorff.} $U$ and some $(G,\Zcal_{<k}(\X),U)$-cocycle $\rho$.
\end{proposition}

\begin{definition}[Phase polynomials, II]\label{phase2}  Let $(G,+)$ be a discrete countable abelian group, let $\X=(X,\B_X,\mu_X, (T_g)_{g \in G})$ be a $G$-system, let $(U,\cdot)$ be a compact abelian group, and let $k \geq 1$.
\begin{itemize}
\item We let $\Phase_{<k}(\X,U)$ denote the set of all $(\X,U)$-functions $\phi$ satisfying the equation $\mder_{h_1} \ldots \mder_{h_k} \phi(x) = 1$ for $\mu_X$-a.e. $x$ and for all $h_1,\ldots,h_k \in G$, where $\mder_h \phi(x) := \phi(T_h x) / \phi(x)$, and we identify functions that agree $\mu_X$-almost everywhere.  We refer to elements of $\Phase_{<k}(\X,U)$ as \emph{$(\X,U)$-phase polynomials of degree $<k$.}  
\item We let $\Phase_{<k}(G,\X,U)$ denote the set of all $(G,\X,U)$-functions $\rho$ satisfying the equation $\mder_{h_1} \ldots \mder_{h_k} \rho(g,x) = 1$ for $\mu_X$-a.e. $x$ and for all $g,h_1,\ldots,h_k \in G$, where $\mder_h \rho(g,x) := \rho(g,T_h x) / \rho(g, x)$, and where we identify functions that agree $\mu_X$-almost everywhere for each $g \in G$.  We refer to elements of $\Phase_{<k}(G,\X,U)$ as \emph{$(G,\X,U)$-phase polynomials of degree $<k$}.  We refer to elements of the intersection $\Phase_{<k}(G,\X,U) \cap Z^1(G,\X,U)$ as \emph{$(G,\X,U)$-phase polynomial cocycles of degree $<k$}.  
\end{itemize}
\end{definition}

\begin{example} If $\X$ is ergodic, then all phase polynomials of degree $<1$ are constant in the $X$ variable, and so
$\Phase_{<1}(\X,U) \equiv U$
and
$\Phase_{<1}(G,\X,U) \equiv M(G,U)$
in this case (for the definition of $M(G,U)$ see Definition \ref{cohom}).  
\end{example}

\begin{remark} Observe that $\Phase_{<k}(\X,U)$ and $\Phase_{<k}(G,\X,U)$ are subgroups of $M(\X,U)$ and $M(G,\X,U)$ respectively. This definition is of course closely related to Definition \ref{phase-def}; indeed, we observe that
$\Phase_{<k}(\X) \equiv \Phase_{<k}(\X,S^1)$.
Also observe that a $(G,\X,U)$-function $\rho$ is a phase polynomial coboundary of degree $<k$ if and only if it is a derivative $\rho(g,x) = \mder_g F(x)$ of a $(\X,U)$-phase polynomial $F$ of degree $<k+1$. Indeed, we have the short exact sequence
\begin{equation}\label{integ}
\begin{CD}
0 @>>> \Phase_{<1}(\X,U) @>>> \Phase_{<k+1}(\X,U) @>\mdersmall>> \Phase_{<k}(G,\X,U) \cap B^1(G,\X,U) @>>> 0
\end{CD}
\end{equation}
and similarly
$$
\begin{CD}
0 @>>> \Phase_{<1}(\X,U) @>>> M(\X,U) @>\mdersmall>> B^1(G,\X,U) @>>> 0.
\end{CD}
$$
(for the definitions of $M(\X,U), B^1(G,\X,U) $ see Definition \ref{cohom}). 
\end{remark}

Using Proposition \ref{typek-prop} and Fourier analysis, we can reduce matters to studying projections of abelian cocycles to the unit circle.  More precisely, in future sections we will show the following result. 

\begin{theorem}[Second reduction]\label{main-3} Let $\F$ be a finite field, let $k \geq 1$, and let $\X$ be an ergodic $\Fw$-system of order $<k$, and let $\X \times_\rho U$ be a (possibly non-ergodic) abelian extension of $\X$ by a $(\Fw,\X,U)$-cocycle which is of order $<k+1$.  Then for every character $\chi \in \hat U$ (i.e. every continuous homomorphism $\chi: U \to S^1$), the $(\Fw,\X,S^1)$-cocycle $\chi \circ \rho$ is cohomologous to a $(\Fw,\X,S^1)$-phase polynomial of degree $<O_{k}(1)$, i.e. $\chi \circ \rho \in \Phase_{<O_{k}(1)}(\Fw,\X,S^1) \cdot B^1(\Fw,\X,S^1)$ for all $\chi \in \hat U$.  
\end{theorem}

\begin{proof}[Proof of Theorem \ref{main-2} assuming Theorem \ref{main-3}]
We induct on $k$.  The claim is trivial for $k=1$, so suppose that $k \geq 2$ and that Theorem \ref{main-2} has already been proven for $k-1$.  By Proposition \ref{typek-prop}, we may assume that $\X = \Zcal_{<k-1}(\X) \times_\rho U$ for some compact abelian group $U$ and some  $(\Fw,Z_{<k-1}(X),U)$-cocycle $\rho$.

Let $\chi \in \hat U$ be any element of the Pontryagin dual $\hat U$ of $U$, thus $\chi: U \to S^1$ is a character (i.e. a continuous homomorphism).  Applying Theorem \ref{main-3}, we see that there exists a measurable function $F_\chi \in M(\Zcal_{<k-1}(\X), S^1)$ and a $(\Fw, \Zcal_{<k-1}(\X), S^1)$-phase polynomial cocycle $\tilde \rho_\chi$ of degree $<O_{k}(1)$, for which we have the equation
\begin{equation}\label{chirho-eq} \chi \circ \rho = \tilde \rho_\chi \mder F_\chi.
\end{equation}
If we define $\phi_\chi \in L^\infty(\X)$ to be the function $\phi_\chi(x,u) := \overline{F_\chi}(x) \chi(u)$, then we see from \eqref{chirho-eq} that $\mder \phi_\chi = \tilde \rho_\chi$, and thus (by \eqref{integ}) $\phi_\chi$ is a phase polynomial of degree $<O_{k}(1)$.  On the other hand, $F_\chi(x)$ lies in $L^2(\Zcal_{<k-1}(\X))$ and can thus be approximated in that space by finite linear combinations of phase polynomials in $\Phase_{<O_{k}(1)}(\X)$, by the induction hypothesis.  Thus we see that the characters $(x,u) \mapsto \chi(u)$ can also be approximated in $L^2(\X)$ by finite linear combinations of phase polynomials in $\Phase_{<O_{k}(1)}(\X)$.  Since $L^2(\X)$ is generated by $L^2(\Zcal_{<k-1}(\X))$ and these characters, we see that $\X$ is an Abramov system of order $<O_{k}(1)$ as claimed.
\end{proof}

It remains to prove Theorem \ref{main-3}.  This will be the objective of the next few sections.

\section{Functions of type $<k$}\label{typek-sec}

We make a further reduction, introducing the useful notion\footnote{This concept is essentially that of a \emph{cocycle of type $k$} from \cite{hk-cubes}, but generalized to non-cocycles and to more general group actions.  We have replaced ``$k$'' by ``$<k$'' as such cocycles will have ``degree'' strictly less than $k$ in some sense.} of a \emph{function of type $<k$}.

\begin{definition}[Functions of type $<k$]\label{fin-def} Let $(G,+)$ be a discrete countable abelian  group, let $\X = (X, \B_X, \mu_X, (T_g)_{g \in G})$ be a $G$-system, and let $U = (U,\cdot)$ be a compact abelian group.  Let $k \geq 0$, and let $\X^{[k]} = (X^{[k]}, \B^{[k]}, \mu^{[k]}, (T_g^{[k]})_{g \in G} )$ be the $G$-system defined in Definition \ref{cubichk-def}.  
\begin{itemize}
\item For each $f \in M(\X,U)$, we define $d^{[k]} f \in M(\X^{[k]},U)$ to be the function
$$ d^{[k]} f((x_\w)_{\w \in \2^k}) := \prod_{\w \in \2^k} f(x_\w)^{\sgn(\w)}$$
where $\sgn(w_1,\ldots,w_k) := w_1 \ldots w_k \in \{-1,+1\}$.    
\item Similarly, for each $\rho \in M(G,\X,U)$, we define $d^{[k]} \rho \in M(G, \X^{[k]}, U)$ to be the function
$$ d^{[k]} \rho(g,(x_\w)_{\w \in \2^k}) := \prod_{\w \in \2^k} \rho(g, x_\w)^{\sgn(\w)}.$$
\item A \emph{$(G,\X,U)$-function of type $<k$} is any function $\rho \in M(G,\X,U)$ such that $d^{[k]} \rho$ is a $(G,\X^{[k]},U)$-coboundary.  We let $M_{<k}(G,\X,U)$ denote the space of $(G,\X,U)$-functions of type $<k$, and let $Z^1_{<k}(G,\X,U) := Z^1(G,\X,U) \cap M_{<k}(G,\X,U)$ 
denote the space of $(G,\X,U)$-cocycles of type $<k$.  
\end{itemize}
\end{definition}

\begin{figure}\label{include}
$$
\begin{CD}
\mder \Phase_{<k+1}(\X,U) @>>>         B^1(G,\X,U)    @. \\
@VVV                                          @VVV                             @.       \\
\Phase_{<k}(G,\X,U) \cap Z^1(G,\X,U)    @>>>    Z^1_{<k}(G,\X,U)    @>>>  Z^1(G,\X,U) \\
@VVV                                         @VVV                             @VVV       \\
\Phase_{<k}(G,\X,U)   @>>>        M_{<k}(G,\X,U) @>>>   M(G,\X,U)
\end{CD}
$$
\caption{Inclusions between various groups of $(G,\X,U)$ functions.}  
\end{figure}

\begin{example} A $(G,\X,U)$-function is of type $<0$ if and only if it is a $(G,\X,U)$-coboundary, thus $M_{<0}(G,\X,U) = Z^1_{<0}(G,\X,U) = B^1(G,\X,U)$.  
\end{example}

We make some easy observations (see Figure \ref{include}):

\begin{lemma}[Basic facts about functions of type $<k$]\label{basic} Let $G$ be a countable abelian group, let $\X = (X, \B_X, \mu_X, (T_g)_{g \in G})$ be an ergodic $G$-system, let $U = (U,\cdot)$ be a compact abelian group, and let $k \geq 0$.  
\begin{itemize}
\item[(i)] Every $(G,\X,U)$-function of type $<k$ is also a $(G,\X,U)$-function of type $<k+1$.
\item[(ii)] The set of $(G,\X,U)$-functions of type $<k$ is a subgroup of $M(G,\X,U)$ that contains the group $B^1(G,\X,U)$ of coboundaries.  In particular, any function $(G,\X,U)$-cohomologous to a function of type $<k$, is also of type $<k$.  
\item[(iii)] Let $f$ be a $(G,\X,U)$-function.  Then $f$ is a $(G,\X,U)$-phase polynomial of degree $<k$ if and only if $d^{[k]} f = 0$ $\mu^{[k]}$-almost everywhere.  In particular, every $(G,\X,U)$-phase polynomial of degree $<k$ is of type $<k$. 
\item[(iv)] If $f$ is a $(G,\X,U)$-coboundary (resp. a $(G,\X,U)$-cocycle), then $d^{[k]} f$ is a $(G,\X^{[k]},U)$-coboundary (resp. a $(G,\X^{[k]},U)$-cocycle).  Equivalently, we have the commuting diagram
$$
\begin{CD}
0  @>>>   B^1(G,\X,U)  @>>> Z^1(G,\X,U) @>>>  H^1(G,\X,U) @>>> 0 \\
@.         @V{d^{[k]}}VV          @V{d^{[k]}}VV        @V{d^{[k]}}VV  \\
0  @>>>   B^1(G,\X^{[k]},U)  @>>> Z^1(G,\X^{[k]},U) @>>>  H^1(G,\X^{[k]},U) @>>> 0
\end{CD}
$$
of short exact sequences.
\end{itemize}
\end{lemma}

\begin{proof} We first show (i).  Let $f$ be a $(G,\X,U)$-function of type $<k$, then we can find a $(\X^{[k]},U)$-function $F$ such that $d^{[k]} f(g,\x) = \mder_{g^{[k]}} F(\x)$ for all $g \in G$ and $\mu^{[k]}$-almost all $\x$.  Expressing $\X^{[k+1]} = \X^{[k]} \times \X^{[k]}$ and using \eqref{def_measures} we conclude that $d^{[k+1]} f(g,\x,\x') = \mder_{g^{[k+1]}} [F(\x)/F(\x')]$ for all $g \in G$ and $\mu^{[k+1]}$-almost all $(\x,\x')$, and so $f$ is of type $<k+1$ as desired.

From (i) we see in particular that coboundaries, being of type $<0$, are of type $<k$.  The claims in (ii) are then easily verified.

To show (iii), we induct on $k$.  The claim is easy for $k=0,1$ (using ergodicity), so suppose that $k \geq 2$ and the claim has already been shown for $k-1$.

Suppose $f$ is a $(G,\X,U)$-phase polynomial of degree $k$.
For any $g \in G$, $\mder_g f$ is a phase polynomial of degree $<k-1$, and so by induction hypothesis $d^{[k-1]}(\mder_g f) = 1$ $\mu^{[k-1]}$-a.e.  Since
$d^{[k-1]}(\mder_g f) = \mder_{g^{[k-1]}}(d^{[k-1]} f)$,
we conclude that $d^{[k-1]} f$ is invariant under the action of the diagonal group $\diag(G^{[k-1]})$.  By Remark \ref{iterm}, $d^{[k-1]}f$ is constant $(\mu^{[k-1]})_s$-almost everywhere for  $P_{k-1}$-a.e $s$.  By Remark \ref{iterm} again, we conclude that $d^{[k]} f=1$ $\mu^{[k]}$-a.e., as desired.  The converse claim follows by reversing all the above steps.

The claim (iv) for coboundaries follows from the identity
$d^{[k]} \mder_g F = \mder_{g^{[k]}} d^{[k]} F$,
valid for any $(\X,U)$-function $F$ and any $g \in G$.  The claim (iv) for cocycles is clear from direct computation.
\end{proof}

The relevance of the type $<k$ concept to us lies in the important observation that abelian extensions of order $<k+1$ arise from functions of type $<k$:

\begin{proposition}\label{ext-type-k}
Let $G$ be a countable abelian group, and let $\X = (X,\B_X,\mu_X, (T_g)_{g \in G})$ be an ergodic $G$-system of order $<k$ for some $k \geq 1$.  Let $U$ be a compact abelian group, and let $\rho$ be a $(G,\X,U)$-cocycle such that the abelian extension $\X \times_\rho U$ is an ergodic $G$-system of order $<k+1$.  Then $\rho$ is a $(G,\X,U)$-cocycle of type $<k$.  In particular, for every character $\chi \in \hat U$, $\chi \circ \rho$ is a $(G,\X,S^1)$-cocycle of type $<k$.
\end{proposition}

\begin{proof}  See \cite[Proposition 6.4]{hk-cubes} (the result there is only stated for $\Z$-actions, but the proof generalizes without difficulty).  We give a brief sketch of the proof here, to indicate why the type $<k$ condition arises naturally.  For simplicity let us take $U=S^1$.  Let $\phi$ be the $(X \times_\rho U,S^1)$-function $\phi(z,u) := u$, then the function $J: L^2(\X^{[k]}) \to L^2((\X \times_\rho U)^{[k]})$ defined by $Jf(\z,\u) = f(\z) d^{[k]} \phi(\z,\u)$ is an isometry whose range is a closed $\diag(G^{[k]})$-invariant space containing $d^{[k]} \phi$.  On the other hand, since $\X \times_\rho U$ has order $<k+1$, we see from Lemma \ref{softfactor} that $\|\phi\|_{U^{k+1}(\X \times_\rho U)} \neq 0$.  Applying Corollary \ref{uk-norm-I_k-cor}, this implies that
$(\pi^{(\X \times_\rho U)^{[k]}}_{\I_k(\X \times_\rho U)})_* d^{[k]} \phi \neq 0$.
By the ergodic theorem, we conclude that the image of $J$ contains a $\diag(G^{[k]})$-invariant function $Jf$.  Unpacking the definitions, this implies that
$f(T^{[k]}_g \z) d^{[k]} \rho(g,\z)=f(\z)$
for $\mu^{[k]}$ a.e $\z$, which roughly speaking asserts that $\rho$ is of type $<k$ on the support of $f$, which is a set of positive measure.  By various cocycle identities one can also conclude that $\rho$ is of type $<k$ on various shifts of the support of $f$, which by ergodicity can be glued together to establish a global type $<k$ condition; see the proof of \cite[Proposition 6.4]{hk-cubes} for details. 
\end{proof}

In view of Proposition \ref{ext-type-k}, Theorem \ref{main-3} will now follow by applying the following result to the $(\Fw,X,S^1)$-cocycle $\chi \circ \rho$.

\begin{theorem}[Third reduction]\label{main-4} Let $\F$ be a finite field, let $m, k \geq 1$, and let $\X$ be an ergodic $\Fw$-system of order $<k$.  Let $f \in M_{<m}(\Fw, \X, S^1)$ be a $(\Fw, \X,S^1)$-function of type $<m$.  Then $f$ is $(\Fw, \X, S^1)$-cohomologous to a $(\Fw, \X,S^1)$-phase polynomial of degree $<O_{k,m}(1)$.  In other words, 
$$ M_{<m}(\Fw, \X, S^1) \subset \Phase_{<O_{k,m}(1)}(\Fw,\X,S^1) \cdot B^1(\Fw,\X,S^1).$$
\end{theorem}

\begin{remark} When $k=1$, the only ergodic $\Fw$-systems of order $0$ are trivial, and Theorem \ref{main-4} is easily verified in this case.  Unwinding all the previous reductions, this already establishes the $k=1$ case of Theorem \ref{main-thm}; thus $\Zcal_{<2}(\X)$ is precisely the Kronecker system spanned by the eigenfunctions of $\X$ (i.e. the elements of $\Phase_{<2}(\X)$).  Of course, as is well known, this fact can also be established more directly by classical spectral theory methods; see for instance the discussion just before \cite[Lemma 4.2]{hk-cubes}.  

While we will only need Theorem \ref{main-4} in the case when $f$ is a $(\Fw, \X,S^1)$-cocycle of type $<m$, for inductive purposes it is important that we generalize the claim to the case where $f$ is a $(\Fw, \X,S^1)$-function of type $<m$.
\end{remark}

We will prove Theorem \ref{main-4} in later sections.  For now, let us observe that we can use Theorem \ref{main-thm} to obtain some structural control on systems of order $<k$.  We say that a group $U$ is \emph{$m$-torsion} for some $m \geq 1$ if we have $u^m = 1$ for all $u \in U$.

\begin{lemma}[Structure groups for $\Fw$-systems are finite torsion]\label{torlemma}  Let $k \geq 2$ be such that Theorem \ref{main-4} holds for $k-1$.  Let $\F$ be a finite field of characteristic $p$, let $\X$ be an ergodic $\Fw$-system of order $<k-1$, and let $\X \times_\rho U$ be an abelian extension of order $<k$ by a compact abelian group which is also ergodic.  Then $U$ is $p^m$-torsion for some $m = O_{k}(1)$.  If $p$ is sufficiently large depending on $k$, then we can take $m=1$.
\end{lemma}

\begin{proof}  Since characters separate points, it suffices to show that $\chi(U)$ is a $p^m$-torsion group for every character $\chi \in \hat U$. By Proposition \ref{ext-type-k}, $\chi \circ \rho$ is a $(\Fw,\X,S^1)$-cocycle of type $<k-1$, and thus by Theorem \ref{main-4} we have $\chi \circ \rho = q_\chi \mder F$ for some $(\Fw,\X,S^1)$-phase polynomial $q_\chi$ of degree $<O_{k}(1)$ and some $(\X,S^1)$-function $F$.  Clearly $q_\chi$ is also a cocycle.  By Lemma \ref{L:values-in-F_p}, $q_\chi$ takes values in the cyclic group $C_{p^m}$ for some $m = O_{k}(1)$, and for $p$ sufficiently large depending on $k$, one can take $m=1$.

As $\chi$ is a character, $\chi(U)$ is a compact subgroup of $S^1$.  The system $\X \times_{\chi \circ \rho} \chi(U)$ is then a factor of $\X \times_\rho U$ and is therefore ergodic.  Since  $\chi \circ \rho = q_\chi \mder F$, the system $\X'$ defined with the same set, $\sigma$-algebra, and action as the system $\X \times_{q_\chi} S^1$, but with the measure $\mu_X \times \mu_{S^1}$ replaced with the measure 
$$\int_{\X'} f(x,u)\ d\mu_{\X'} := \int_\X \int_{\chi(U)} f( x, F(x) u )\ d\mu_{\chi(U)}(u) d\mu_X(x),$$
is equivalent to $\X \times_{\chi \circ \rho} \chi(U)$ and thus also ergodic.  On the other hand, since $q_\chi$ takes values in $C_{p^m}$, any set of the form $X \times A$, where $A$ is invariant under $C_{p^m}$, will be invariant in $\X'$.  These two statements are only consistent with each other if $\chi(U)$ is a subgroup of $C_{p^m}$, and is thus $p^m$-torsion, as desired.
\end{proof}


\begin{theorem}[Structure theorem]\label{struc-thm} Let $k \geq 1$ be such that Theorem \ref{main-4} holds for all smaller values of $k$.  Let $\F$ be a finite field of characteristic $p$, and let $\X$ be an ergodic $\Fw$-system.  Then $\Zcal_{<1}(\X)$ is trivial, and for all $2 \leq j \leq k$, we can write
$ \Zcal_{<j}(\X) \equiv \Zcal_{<j-1}(\X) \times_{\rho_{j-1}} U_{j-1}$
where $U_{j-1}$ is $p^m$-torsion for some $m =O_{k,p}(1)$, and $\rho_{j-1}$ is a $(\Fw,\Zcal_{<j-1}(\X),U_{j-1})$-phase polynomial cocycle of degree $O_{k}(1)$.  In particular, we have
$$ \Zcal_{<k}(\X) \equiv U_0 \times_{\rho_1} U_1 \times_{\rho_2} \ldots \times_{\rho_{k-1}} U_{k-1}$$
where $U_0$ is trivial.
\end{theorem}

\begin{proof}  Fix $\F$, $\X$. The triviality of $\Zcal_{<1}(\X)$ follows from ergodicity.
 Now suppose inductively that $k \geq 2$, and the claim has already been proven for smaller values of $k$. By Proposition \ref{typek-prop}, we can write
$\Zcal_{<k}(\X) \equiv \Zcal_{<k-1}(\X) \times_{\rho_{k-1}} U_{k-1}$
for some compact abelian group $U_{k-1}$ and some $(\Fw,\Zcal_{<k-1}(\X),U_{k-1})$-cocycle $\rho_{k-1}$.  By Lemma \ref{torlemma}, the compact abelian group $U_{k-1}$ is thus $p^m$-torsion for some $m=O_{k}(1)$ (with $m=1$ if $p$ is large enough).   This implies (see \cite[Chapter 5, Theorem 18]{morris}) that $U_{k-1}$ is topologically isomorphic to the direct product of some cyclic $p$-groups $C_{p^m}$ for $m = O_k(1)$.  Let $\chi: U_{k-1} \to C_{p^m}$ be one of the coordinate maps.  By Theorem \ref{main-4}, $\chi \circ \rho_{k-1}$ is cohomologous to a $(\Fw,\Zcal_{<k-1}(\X),S^1)$-phase polynomial cocycle $q_{\chi}$ of degree $<O_{k}(1)$, which by Lemma \ref{L:values-in-F_p} also takes values some $p$-group $C_{p^n}$ for some $n = O_{k}(1)$.  By increasing $n$ if necessary we can take $n \geq m$; but note that if $p$ is sufficiently large depending on $k$ then we can take $n=m=1$.  

By construction, there exists a $(\Zcal_{<k-1}(\X),S^1)$-function $F_\chi$ such that
$\chi \circ \rho_{k-1} = q_{\chi} \mder F_\chi$.  Since $(\chi \circ \rho_{k-1})^{p^m} = 1$, we conclude that $1 = q_\chi^{p^m} \mder F_\chi^{p^m}$.

We claim that there exists a $(\Zcal_{<k-1}(\X),S^1)$-phase polynomial $\tilde F_\chi$ of degree $<O_{k}(1)$ such that $F_\chi^{p^m} = \tilde F_\chi^{p^m}$.  There are two cases.  If $p$ is sufficiently large depending on $k$, then $q_\chi^{p^m}=1$, so by ergodicity $F_\chi^{p^m}$ is constant, and the claim is trivial.  Now suppose instead that $p = O_k(1)$.  Then $\mder F_\chi^{p^m}$ is a $(\Fw,\Zcal_{<k-1}(\X),S^1)$ phase polynomial of degree $<O_{k}(1)$, and so (by \eqref{integ}) $F_\chi^{p^m}$ is a $(\Zcal_{<k-1}(\X),S^1)$-phase polynomial of degree $<O_{k}(1)$.  By Lemma \ref{root}, we can find a $(\Zcal_{<k-1}(\X),S^1)$-phase polynomial $\tilde F_\chi$ of degree $<O_{k}(1)$ such that $F_\chi^{p^m} = \tilde F_\chi^{p^m}$, as claimed.  

Now write $q'_\chi := q_\chi \mder \tilde F_\chi$ and $F'_\chi := F_\chi/\tilde F_\chi$, then we have
$\chi \circ \rho = q'_\chi \mder F'_\chi$.
Since $F'_\chi$ takes values in $C_{p^m}$, and $q'_\chi$ does also.  By construction, $q'_\chi$ is thus a $(\Fw,\X,C_{p^m})$-phase polynomial of degree $<O_{k}(1)$.

If we let $q' := (q'_\chi)_\chi$ and $F' := (F_\chi)_{\chi}$, where $\chi$ ranges over all the coordinate projection maps, then $q'$ is a $(\Fw,\X,U)$-phase polynomial of degree $<O_{k}(1)$ and
$\rho = q' \mder F'$.
Thus $\rho$ is cohomologous to a phase polynomial of degree $<O_{k}(1)$, and thus we may assume 
(see Remark \ref{measure-equiv}) that $\rho$ is equal to a phase polynomial of degree $<O_{k}(1)$.  The claim then follows from the induction hypothesis.
\end{proof}

Our remaining task is to prove Theorem \ref{main-4}.

\section{Reduction to solving a Conze-Lesigne type equation}

To prove Theorem \ref{main-4} we will use two lemmas to reduce matters to solving a certain equation of Conze-Lesigne type.
The first lemma allows one to descend a type condition on an extension to a type condition on a base, worsening the type if necessary:

\begin{lemma}[Descent of type]\label{descent}  Let $G$ be a countable abelian group, let $\Y$ be a $G$-system, let $k, m \geq 1$, and let $\X = \Y \times_\rho U$ be an ergodic abelian extension of $\Y$ by a $(G,\Y,U)$-phase polynomial cocycle $\rho$ of degree $<m$.  Let $\pi: X \to Y$ be the factor map, and let $f$ be a $(\Y,S^1)$-function such that $\pi^* f$ is of type $<k$.  Then $f$ is of type $<k+m+1$.
\end{lemma}

\begin{remark}\label{descent-remark} For cocycles, a more general (and stronger) statement appears in \cite[Corollary 7.8]{hk-cubes}.  However, for technical reasons, it is necessary for us to work with more general functions than just cocycles.  (But see Corollary \ref{exact_descent} below.)
\end{remark}

\begin{proof}  Note that $\pi^* d^{[k]} f = d^{[k]} \pi^* f$ is a $(G,\X^{[k]},S^1)$-coboundary, hence a $(G,\X^{[k]},S^1)$-cocycle, and so (by Lemma \ref{cocycle}) $d^{[k]} f$ is a $(G,\Y^{[k]},S^1)$-cocycle.

It is convenient to move from $k$ to $k+1$. By Lemma \ref{basic}(i), $\pi^* f$ is of type $<k+1$, thus we have an $(\X^{[k+1]},S^1)$-function $F$ such that
$ d^{[k+1]}(\pi^* f)(g, \x) = \mder_{g^{[k+1]}} F(\x)$
for all $g \in G$ and $\mu^{[k+1]}$-almost every $\x \in X^{[k+1]}$.  

We view $\X^{[k+1]}$ as $\Y^{[k+1]} \times U^{[k+1]}$. Since the projection of $\mu^{[k+1]}$ on
 $\Y^{[k+1]}$  is $(\pi_* \mu)^{[k+1]}$, 
every ergodic component of $\X^{[k+1]}$ is an abelian extension of an ergodic component of $\Y^{[k+1]}$, with a Mackey group $M \leq U^{[k+1]}$ (see for example \cite{fw-char}).  In other words, for each ergodic component of $\Y^{[k+1]}$, there exists a map $G: Y^{[k+1]} \to U^{[k+1]}$ such that the $(\Fw,\Y^{[k+1]},U^{[k+1]})$-cocycle
$$ \tilde \rho(g, \y) := (\rho(g, \y_\w))_{\w \in \2^{k+1}} \mder_{g^{[k+1]}} G(\y)$$
takes values in $M$ (i.e. it is a $(\Fw,\Y^{[k+1]},M)$-cocycle).  On each ergodic component of $\X^{[k+1]}$, if we set
$K(\y,\m) := F( \y, G(\y) \m )$
then we have
$$ d^{[k+1]} f(g, \y ) = \frac{ K( S_g^{[k+1]} \y, \tilde \rho(g,\y) \m ) }{K( \y, \m)}.$$
If we expand into a Fourier series
$ K(\y,\m) = \sum_{\chi \in \hat M} a_\chi(\y) \chi(\m)$
and compare Fourier coefficients, we conclude
$$ a_\chi( S_{g^{[k+1]}} \y ) \chi( \tilde \rho(g,\y) ) = d^{[k+1]} f(g,\y) a_\chi(\y)$$
almost everywhere on any ergodic component.  In particular, $|a_\chi|$ is invariant and thus constant a.e. on any ergodic component.  We extend $\chi$ arbitrarily to a character $\chi \in \widehat{U^{[k+1]}} \equiv {\hat U}^{[k+1]}$.

On almost every component, at least one of the $a_\chi$ must be non-vanishing; since there are at most countably many characters (by the separability of $\X$), we conclude that for some $\chi \in \hat U^{[k+1]}$, we have $|a_\chi| \neq 0$ on a positive measure collection of ergodic components.  On each such component, if one sets $b := a_\chi / |a_\chi|$, we thus have
$$ d^{[k+1]} f(g,\y) = \chi( \tilde \rho(g,\y) ) \mder_{g^{[k+1]}} b(\y)
= [\prod_{\w \in \2^{k+1}} \chi_\w( \rho(g,\y_\w) )] \mder_{g^{[k+1]}} Gb(\y),$$
where $\chi_\w$ are the components of $\chi$.  Thus the $(G,\Y^{[k+1]},S^1)$-function
$$ H(g,\y) := d^{[k+1]} f(g,\y) / \prod_{\w \in \2^{k+1}} \chi_\w( \rho(g,\y_\w) )$$
is a $(G,A,S^1)$-coboundary for some $\diag(G^{[k]})$-invariant subset $A$ of $\Y^{[k+1]}$ of positive measure.  

We now claim that $(T_h)^{[k+1]}_\alpha H$ is also a $(G,A,S^1)$-coboundary for every positive side transformation $(T_h)^{[k+1]}_\alpha H$.  But a computation shows that
$$ \frac{(T_h)^{[k+1]}_\alpha H}{H}(g,y) = \Delta_{h^{[k]}} d^{[k]} f(g, \partial(\alpha)_*(\y)) / 
\prod_{\w \in \alpha} \chi_\w( \Delta_h \rho(g,\y_\w) ).$$
Now we crucially use the fact that $d^{[k]} f$ and $\rho$ are cocycles to write this as
$$ \frac{(T_h)^{[k+1]}_\alpha H}{H}(g,y) = \Delta_{g^{[k]}} d^{[k]} f(h, \partial(\alpha)_*(\y)) / 
\prod_{\w \in \alpha} \chi_\w( \Delta_g \rho(h,\y_\w) ).$$
The right-hand side is then a $(G,\Y^{[k+1]},S^1)$-coboundary, and the claim follows.

Arguing as in the proof of Proposition \ref{ext-type-k}, we can glue various translated coboundaries together and conclude that $H$ is a $(G,\Y,S^1)$-coboundary, thus $d^{[k+1]} f(g,\y)$ is $(G,\Y^{[k+1]},S^1)$-cohomologous to $\prod_{\w \in \2^{k+1}} \chi_\w( \rho(g,\y_\w) )$.  On the other hand, the factors $\chi_\w( \rho(g,\y_\w) )$ are of degree $<m$ and thus $d^{[m]} \chi_\w( \rho(g,\y_\w) ) = 0$.  It then follows that $d^{[m+k+1]} f(g,\y)$ is a $(G,\Y^{[m+k+1]},S^1)$-coboundary, and so $f$ is of type $<m+k+1$ as claimed. 
\end{proof}

The second lemma allows us to reduce the type of a function by differentiation in the vertical direction.

\begin{lemma}[Vertical differentiation lemma]\label{vert-lem} Let $G$ be a countable abelian group, let $k \geq 1$, and let $\X = \Y \times_\rho U$ be a $G$-system of order $<k$ for some compact abelian $U$ and a $(G,\Y,U)$-cocycle $\rho$.  Let $f$ be a $(G,\X,S^1)$-function of type $<m$ for some $m \geq 1$.  Then for every $t \in U$, the $(G,\X,S^1)$-function $\mder_t f$ is of type $<m-\min(k,m)$, where $t$ acts on $\X$ by the action $V_t: (y,u) \mapsto (y,tu)$ for $y \in Y$ and $u \in U$.
\end{lemma}

\begin{proof}  By hypothesis, there exists a $(\X^{[m]},S^1)$-function $F$ such that
$d^{[m]} f = \mder F$.  Let $\alpha$ be a $m-\min(k,m)$-dimensional side of $\2^m$.  By Lemma \ref{erglem}(iv), the side transformation $(V_t)^{[m]}_\alpha$ leaves $\mu^{[m]}$ invariant, and commutes with the $\diag(G^{[k]})$-action.  Thus we have
$d^{[m]} (V_t)^{[m]}_\alpha f = \mder (V_t)^{[m]}_\alpha F$.
Dividing through, we conclude
$$ d^{[m]} \frac{(V_t)^{[m]}_\alpha f}{f} = \mder \frac{(V_t)^{[m]}_\alpha F}{F}$$
and so $d^{[m]} \frac{(V_t)^{[m]}_\alpha f}{f}$ is a $(G, \X^{[m]}, S^1)$-coboundary.  This expression can be rewritten as
$$ d^{[m]} \frac{(V_t)^{[m]}_\alpha f}{f} = (\partial(\alpha)_*)^* (d^{[m-\min(k,m)]} \mder_t f),$$
and so $(\partial(\alpha)_*)^* (d^{[m-\min(k,m)]} \mder_t f)$ is a $(G, \X^{[m]}, S^1)$-coboundary, and thus a $(G,\X^{[m]},S^1)$-cocycle, which (by Lemma \ref{cocycle}) implies that $d^{[m-\min(k,m)]}\mder_t f$ is a $(G,\X^{[m-1]},S^1)$-cocycle.  Applying Lemma \ref{hkc-lem}, we conclude that $d^{[m-\min(k,m)]} \mder_t f$ is in fact a $(G, \X^{[m-\min(k,m)]}, S^1)$-coboundary, and so $\mder_t f$ is of type $<m-\min(k,m)$ as desired.
\end{proof}

Because of these two lemmas, Theorem \ref{main-4} will follow from

\begin{theorem}[Fourth reduction]\label{main-6} Let $k \geq 2$ and $m \geq 1$, and assume that Theorem \ref{main-4} has already been proven for all smaller values of $k$. Let $\F$ be a finite field, and let $\X$ be an ergodic $\Fw$-system of order $<k$.  Write $\X = \Zcal_{<k-1}(\X) \times_\rho U$ for some abelian $(\Fw,\Zcal_{<k-1}(\X),U)$-cocycle $\rho$.  Let $f \in M_{<m}(\Fw, \X, S^1)$ be a function of type $<m$, be such that
$$ \mder_t f \in \Phase_{<m}(\Fw,\X,S^1) \cdot B^1(\Fw,\X,S^1)$$
for all $t \in U$ (i.e. $\mder_t f$ is always cohomologous to a phase polynomial of degree $<m$).  Then $f$ is cohomologous to $P \pi^* \tilde f$ for some $P \in \Phase_{<O_{k,m}(1)}(\Fw,\X,S^1)$, and some $(\Fw,\Zcal_{<k-1}(\X),S^1)$-function $\tilde f$, where $\pi$ is the factor map from $\X$ to $\Zcal_{<k-1}(\X)$.
\end{theorem}

\begin{proof}[Proof of Theorem \ref{main-4} assuming Theorem \ref{main-6}] Let $k$, $m$, $\F$, $\X$, $f$ be as in Theorem \ref{main-4}.  We can assume inductively that Theorem \ref{main-4} has already been proven for smaller values of $k$.

We claim for each $0 \leq j \leq m$ that
\begin{equation}\label{main5-induct}
\mder_{t_1} \ldots \mder_{t_j} f \in \Phase_{<O_{k,m,j}(1)}(\Fw,\X,S^1) \cdot B^1(\Fw,\X,S^1)
\end{equation}
for all $t_1,\ldots,t_j \in U$.

We establish this by downward induction on $j$.  When $j=m$, the claim follows by $m$ applications of Lemma \ref{vert-lem} (and no $\Phase_{<O_{k,m,j}(1)}(\Fw,\X,S^1)$ term appears in this case).  Now suppose that $0 \leq j < m$ and the claim \eqref{main5-induct} has already been shown for $j+1$.  Applying Theorem \ref{main-6}, we see that for all $t_1,\ldots,t_j$,
$\mder_{t_1} \ldots \mder_{t_j} f$ is $(\Fw,\X,S^1)$-cohomologous to $P_{t_1,\ldots,t_j} \pi^* \tilde f_{t_1,\ldots,t_j}$ for some $P_{t_1,\ldots,t_j} \in \Phase_{<O_{k,m,j}(1)}(\Fw,\X,S^1)$ and $\tilde f \in M(\Fw,\Zcal_{<k-1}(\X),S^1)$.

Since $f$ is of type $<m$, and the $U$ action commutes with the $G$ action, we see that $\mder_{t_1} \ldots \mder_{t_j} f$ is also of type $<m$.  By Lemma \ref{basic}(ii), (iii) we conclude that the $(\Fw,\X,S^1)$-function $\pi^* \tilde f_{t_1,\ldots,t_j}$ is of type $<O_{k,m,j}(1)$.  Applying Lemma \ref{descent}, we conclude that the $(\Fw,\Zcal_{<k-1}(\X),S^1)$-function $\tilde f_{t_1,\ldots,t_j}$ is of type $<O_{k,m,j}(1)$.  Applying the inductive hypothesis, we conclude that $\tilde f_{t_1,\ldots,t_j}$ is $(\Fw,\Zcal_{<k-1}(\X),S^1)$-cohomologous to a $(\Fw,\Zcal_{<k-1}(\X),S^1)$-phase polynomial of degree $<O_{k,m,j}(1)$.  By functoriality, this implies that $\pi^* \tilde f_{t_1,\ldots,t_j}$ is $(\Fw,\X,S^1)$-cohomologous to a $(\Fw,\X,S^1)$-phase polynomial of degree $<O_{k,m,j}(1)$.  Since $\mder_{t_1} \ldots \mder_{t_j} f$ was $(\Fw,\X,S^1)$-cohomologous to $P_{t_1,\ldots,t_j} \pi^* \tilde f_{t_1,\ldots,t_j}$, the claim \eqref{main5-induct} follows.

Theorem \ref{main-4} now follows by specializing \eqref{main5-induct} to the case $j=0$.
\end{proof}

It remains to establish Theorem \ref{main-6}.

\section{Reduction to a finite $U$}

The purpose of this section is to obtain the following reduction.

\begin{proposition}[Reduction to finite $U$]\label{fin-prop} In order to prove Theorem \ref{main-6}, it suffices to do so in the case when $U$ is finite.
\end{proposition}

\begin{proof} Fix $k$, and assume that Theorem \ref{main-6} has already been proven in the case of finite $U$.  Let $m, \F, \X, \rho, U, f$ be as in Theorem \ref{main-6}.  By Theorem \ref{struc-thm} we may assume that $\rho$ is a $(\Fw,\Zcal_{<k-1}(\X),U)$-phase polynomial of degree $<O_{k,p}(1)$.  By hypothesis, for each $t \in U$ there exists $q_t \in \Phase_{<m}(\Fw,\X,S^1)$ and $F_t \in M(\X,S^1)$ such that
\begin{equation}\label{mdert}
\mder_t f = q_t \mder F_t.
\end{equation}
By Lemma \ref{measurable_choice} we can take $q_t, F_t$ to be measurable with respect to $t$.

The next step is to linearize $q_t$ on an open subgroup of $U$, by arguing as follows.
Let $t, u \in U$.  Then the cocycle identity $\mder_{tu} f = (\mder_t (V_u f)) \mder_u f$ and \eqref{mdert} give
$$
q_{tu} \mder F_{tu} = (V_u q_t) q_u \mder((V_u F_t) F_u)$$
and hence
$$ \mder \frac{F_{tu}}{(V_u F_t) F_u} \in \Phase_{<m}(\Fw,\X,S^1)$$
and thus (by \eqref{integ}), the function
\begin{equation}\label{phitu}
\phi_{t,u} := \frac{F_{tu}}{(V_u F_t) F_u}
\end{equation}
is a $(\X,S^1)$-phase polynomial of degree $<m+1$.

Let $\eps > 0$ be a small number.  By Lusin's theorem, the function $t \to F_t$ is equal to a uniformly continuous function (in $L^2(\X)$) outside of an open set $E$ of measure $O(\eps)$ in $U$.  In particular, there exists an open neighborhood $U'$ of the identity in $U$ such that $F_{tu}$ and $V_u F_t$ both lie within $O(\eps)$ of $F_t$ whenever $u \in U'$ and $t, tu \not \in E$.  In particular, we see that $\phi_{t,u}$ lies within $O(\eps)$ of $\overline{F_u}$ when $u \in U'$ and $t, tu \not \in E$.  Applying Lemma \ref{sep-lem} and taking $\eps$ small enough, we conclude that for each $u \in U'$ there exists $\phi'_u \in \Phase_{<m+1}(\X,S^1)$ such that $\phi_{t,u}/\phi'_u$ is constant whenever $t, tu \not \in E$.  Arguing as in the proof of Lemma \ref{measurable_choice} we can ensure that $u \mapsto \phi'_u$ is measurable.

Now let $u, v \in U'$.  From \eqref{phitu} we have the second-order cocycle identity
$\phi_{uv,w} \phi_{u,v} = \phi_{u,vw} V_u \phi_{v,w}$
for all $w \in U$.  If $\eps$ is small enough, we can find $w$ so that $w, vw,uvw \not \in E$.  We conclude that 
\begin{equation}\label{phuv}
 \phi_{u,v} = c(u,v) \frac{\phi'_u V_u \phi'_v}{\phi'_{uv}}
\end{equation}
for some constant $c(u,v)$.  If we then set $F'_u := F_u / \phi'_u$ and $q'_u := q_u \mder \phi'_u$ for $u \in U$, then $q'_u$ is a $(G,\X,S^1)$-phase polynomial of degree $<m$, and we conclude from \eqref{mdert} that
\begin{equation}\label{mdertu}
\mder_u f = q'_u \mder F'_u 
\end{equation}
for all $u \in U$, while from \eqref{phitu}, \eqref{phuv} we have
$\mder \frac{F'_{uv}}{(V_u F'_v) F'_u} = 1$
and thus (by \eqref{mdertu})
\begin{equation}\label{coca}
\frac{q'_{uv}}{(V_u q'_v) q'_u} = 1.
\end{equation}
for $u,v \in U'$.

Now we make a crucial use of the finite characteristic hypothesis.  By Lemma \ref{torlemma}, $U$ is $p^n$-torsion for some $n = O_{k}(1)$.  By Lemma \ref{tor-lem}, we conclude that $U'$ contains an open subgroup of $U$; by reducing $U'$ if necessary, we may assume that $U'$ is in fact equal to an open subgroup.  

Now we pass from $U$ to $U'$ as follows.  By Lemma \ref{subgroup}, we may write $U = U' \times W$ for some finite group $W$ (shrinking $U'$ if necessary).  We can then factorise $\X = \Zcal_{<k-1}(\X) \times_\rho U$ as $\X = \Y \times_{\rho'} U'$, where $\Y := \Zcal_{<k-1}(\X) \times_{\rho'} W$ and $\rho'$, $\rho''$ are the projections of $\rho$ to $U'$ and $W$ respectively.  Note that $\rho'$ is a $(\Fw,\Y,U')$-phase polynomial of degree $<O_{k}(1)$ (since $\rho$ is also).

Applying Lemma \ref{poly-integ} once for each $g \in \Fw$ and then pasting together, we may write $q'_u = \mder_u Q$ for all $u \in U$ and some $(G,\X,S^1)$-phase polynomial $Q$ of degree $<O_{k,m}(1)$.
If we let $f'$ be the $(\Fw,\X,S^1)$-function $f' := f/Q$, we thus conclude from \eqref{mdertu} that
$
\mder_u f' = \mder F'_u
$
for all $u \in U'$.

The compact abelian group $U$ acts freely on $\X$ in a manner commuting with the $G$ action, and so the compact abelian subgroup $U'$ does also.  We can thus apply Lemma \ref{straighten-lemma} and conclude that $f'$ is $(\Fw,\X,S^1)$-cohomologous to a $(\Fw,\X,S^1)$-function $f''$ which is invariant with respect to some open subgroup $U''$ of $U'$.

Let $\sigma: U \to U/U''$ be the quotient map, let $\X' := \Zcal_{<k-1}(\X) \times_{\sigma \circ \rho} U/U''$, and let $\pi: \X \to \X'$ be the associated factor map.  Then we can write $f'' = \pi^* \tilde f$ for some $(\Fw, \X',S^1)$-function $\tilde f$.  By functoriality, $\X'$ is of order $<k$. 

By construction, $\pi^* \tilde f$ is cohomologous to $f'=f/Q$.  Since $f$ has type $<m$ and $Q$ is a phase polynomial of degree $<O_{k,m}(1)$, we see from Lemma \ref{basic}(ii), (iii) that $\pi^* \tilde f$ is a $(\Fw,\X,S^1)$-function of type $<O_{k,m}(1)$.

From \eqref{mdert} and the polynomial nature of $Q$, we know that $\pi^* \mder_t \tilde f = \mder_t \pi^* \tilde f$ is $(\Fw,\X,S^1)$-cohomologous to a $(\Fw,\X,S^1)$-phase polynomial of degree $<O_{k,m}(1)$.  By Lemma \ref{descent-lem}, $\mder_t \tilde f$ is thus $(\Fw,\X',S^1)$-cohomologous to a $(\Fw,\X',S^1)$-phase polynomial of degree $<O_{k,m}(1)$, times a function of the form $\chi \circ \rho \circ \pi^{\X'}_{\Zcal_{<k-1}(\X)}$ for some $\chi \in \hat U$.  But recall that $\rho$ is a $(\Fw,\Zcal_{<k-1}(\X),U)$-phase polynomial of degree $<O_{k}(1)$.  We conclude that $\mder_t \tilde f$ is $(\Fw,\X',S^1)$-cohomologous to a $(\Fw,\X',S^1)$-phase polynomial of degree $<O_{k,m}(1)$.  We can now invoke Theorem \ref{main-6} for the finite group $U/U''$ and conclude that $\tilde f$ is $(\Fw,\X',S^1)$-cohomologous to $P (\pi^{\X'}_{\Zcal_{<k-1}(\X)})^* \tilde f'$ for some $(\Fw,\X',S^1)$-phase polynomial $P$ of degree $<O_{k,m}(1)$, and some $(\Fw,\Zcal_{<k-1}(\X),S^1)$-function $\tilde f'$.  Since $f$ is cohomologous to $Q \pi^* \tilde f$, we obtain the desired factorization of $f$ required to establish Theorem \ref{main-6}.
\end{proof}

\section{The finite group case}\label{finite-sec}

\begin{proposition}[The finite case]\label{propfin}  Theorem \ref{main-6} is true when $U$ is finite.
\end{proposition}

\begin{proof}
Let $k, m,\F,\X,\rho,U,f$ be as in Theorem \ref{main-6}, and let $p$ be the characteristic of $\F$.  By Lemma \ref{torlemma}, $U$ is a finite abelian group which is $p^n$-torsion for some $n = O_{k}(1)$, and so by the classification of finite abelian groups we may assume that 
$$ U = C_{p^{n_1}} \times \ldots \times C_{p^{n_N}}$$
for some $1 \leq n_1,\ldots,n_N \leq O_{k}(1)$ and some finite (but unbounded) $N$.  When $p$ is large enough, depending on $k$, we can take $n_1=\ldots=n_N=1$.

Let $e_1,\ldots,e_N$ be the standard set of generators of $U$.  
By hypothesis, for every $1 \leq j \leq N$ we can find $p_j \in \Phase_{<m}(\Fw,\X,S^1)$ and $F_j \in M(\X,S^1)$ such that
\begin{equation}\label{mderif}
\mder_{e_j} f = p_j \mder F_j.
\end{equation}
The idea here is to express $F_j$ as $\mder_{e_j} F$ times a polynomial error, for some $F$ independent of $j$; we will then ``integrate'' this to express $f$ as $\mder F$ times a polynomial error, times a function invariant under $e_1,\ldots,e_N$; this is basically what we need to establish Theorem \ref{main-6}.

We turn to the details.  
The first task is to measure two potential obstructions to $F_j$ being expressible as $\mder_{e_j} F$ (modulo polynomial errors), namely the obstruction coming from the torsion of $U$, and the obstruction coming from the multi-dimensionality of $U$.

Observe that we have the telescoping identity
\begin{equation}\label{Tele}
\prod_{t=0}^{p^{n_j}-1} V_{e_j^t} \mder_{e_j} f = 1
\end{equation}
and thus by \eqref{mderif}
$$ \mder \prod_{t=0}^{p^{n_j}-1} V_{e_j^t} F_j \in \Phase_{<m}(\Fw,\X,S^1)$$
and so (by \eqref{integ}) we have
\begin{equation}\label{vfj}
\prod_{t=0}^{p^{n_j}-1} V_{e_j^t} F_j \in \Phase_{<m+1}(\X,S^1).
\end{equation}
Note, conversely, that $\prod_{t=0}^{p^{n_j}-1} V_{e_j^t} F_j$ needed to be polynomial in order to have any chance to express $F_j$ as $\mder_{e_j} F$ times a polynomial.

Suppose now that $p$ is sufficiently large depending on $k$, so that $n_j=1$.  By Lemma \ref{L:values-in-F_p} we have
$\prod_{t=0}^{p^{n_j}-1} V_{e_j^t} p_j = 1$,
and hence by \eqref{Tele}, \eqref{mderif} we may now strengthen \eqref{vfj} to
$$ \mder \prod_{t=0}^{p^{n_j}-1} V_{e_j^t} F_j = 1.$$

In a similar fashion, from the commutation identity
$\frac{\mdersmall_{e_i} \mdersmall_{e_j} f}{\mdersmall_{e_j} \mdersmall_{e_i} f} = 1$
for any $1 \leq i,j \leq N$, we see from Lemma \ref{polycomp}(i) and \eqref{mderif} that
$$ \mder ( \frac{\mder_{e_i} F_j}{\mder_{e_j} F_i} ) \in \Phase_{<m-1}(\Fw,\X,S^1)$$
and hence by \eqref{integ} we have
\begin{equation}\label{midori}
\frac{\mder_{e_i} F_j}{\mder_{e_j} F_i} \in \Phase_{<m}(\Fw,\X,S^1).
\end{equation}
Again, observe that $\frac{\mdersmall_{e_i} F_j}{\mdersmall_{e_j} F_i}$ had to be polynomial in order to have a chance to express $F_i$, $F_j$ as $\mder_{e_i} F$, $\mder_{e_j} F$ modulo polynomials.

We now clean up the first condition \eqref{vfj}.  We claim that there exists a $e_j$-invariant phase polynomial $\Psi_j \in \Phase_{<O_{k,m}(1)}(\X,S^1)$ such that
$\prod_{t=0}^{p^{n_j}-1} V_{e_j^t} F_j = \Psi_j^{p^{n_j}}.$

When $p$ is sufficiently large depending on $k$, then as observed before, $\prod_{t=0}^{p^{n_j}-1} V_{e_j^t} F_j$ is constant, and the claim is trivial, so suppose instead then $p=O_k(1)$.  By \eqref{vfj}, the $(\X,S^1)$-function $\prod_{t=0}^{p^{n_j}-1} V_{e_j^t} F_j$ is polynomial of degree $<m+1$.  By inspection, it is also $e_j$-invariant.  Quotienting out by the $e_j$ action, applying Corollary \ref{root}, and pulling back, we obtain the claim in this case.

If we then write
$ \tilde F_j := F_j / \Psi_j$,
then by construction we have
\begin{equation}\label{vfj2}
\prod_{t=0}^{p^{n_j}-1} V_{e_j^t} \tilde F_j = 1
\end{equation}
and
\begin{equation}\label{mderif2}
\mder_{e_j} f \in  (\mder F_j) \cdot \Phase_{<O_{k,m}(1)}(\Fw,\X,S^1)
\end{equation}
while from \eqref{midori} we have
\begin{equation}\label{midori2}
\frac{\mder_{e_i} \tilde F_j}{\mder_{e_j} \tilde F_i} \in \Phase_{<O_{k,m}(1)}(\Fw,\X,S^1)
\end{equation}
Note from \eqref{vfj2} and Lemma \ref{L:values-in-F_p} that the $\tilde F_i$ take values in $C_{p^M}$ for some $M = O_{k,m}(1)$; when $p$ is sufficiently large depending on $k$, $m$, then we can take $M=1$.

Recall that in the $N$-dimensional Euclidean space $\R^N$, a (smooth) vector field $(\tilde F_i)_{1 \leq i \leq N}$ which obeys the curl-free condition $\frac{\partial F_i}{\partial x_j} - \frac{\partial F_j}{\partial x_i} = 0$ can be expressed as a gradient $\tilde F_i = \frac{\partial F}{\partial x_i}$, where $F$ can be given explicitly by the formula
$$ F( (t_1,\ldots,t_N) ) := \sum_{i=1}^N \int_0^{t_i} \tilde F_i( (t_1,\ldots,t_{i-1},t'_i,0,\ldots,0) )\ dt'_i.$$
If in addition the $\tilde F_i$ are periodic modulo $\Z^N$, and obey the necessary condition $\int_0^1 \tilde F_i( x + t_i e_i )\ dt_i = 0$ for all $1 \leq i \leq N$ and $x \in \R^N$, then the function $F$ defined above is also periodic modulo $\Z^N$ and descends to the torus $\R^N/\Z^N$.

We can perform exactly the same construction in our finite multiplicative setting.  By Theorem \ref{struc-thm}, we may write $\X = \Zcal_{<k-1}(\X) \times_\rho U$ where $\rho$ is a $(\Fw,\Zcal_{<k-1}(\X),U)$-phase polynomial cocycle of degree $<O_{k}(1)$.  We adopt the notation
$$ [t_1,\ldots,t_N] := e_1^{t_1} \ldots e_N^{t_N}$$
for any integers $t_1,\ldots,t_N$, thus for instance $[t_1,\ldots,t_N]$ is periodic in $t_j$ with period $p^{n_j}$.  We let $F$ be the $(\X,S^1)$ function defined by the formula
$$ F( y, [t_1,\ldots,t_N] ) := \prod_{i=1}^N \prod_{0 \leq t'_i < t_i} \tilde F_i( y, [t_1,\ldots,t_{i-1},t'_i,0,\ldots,0] )$$
for $y \in \Zcal_{<k-1}(\X)$ and $t_1,\ldots,t_N \in \Z$, with the convention that $\prod_{0 \leq t'_i < t_i} a_{t'_i} := (\prod_{t_i \leq t'_i < 0} a_{t'_i})^{-1}$ when $t_i$ is negative.  Note from \eqref{vfj2} that the right-hand side here is periodic in $t_j$ with period $p^{n_j}$, and so $F$ is well-defined.  Since the $\tilde F_i$ take values in a cyclic group $C_{p^M}$, $F$ does also.

Now we compute a derivative of $F$.  We clearly have
$$ \mder_{e_j} F( y, [t_1,\ldots,t_N] ) = \prod_{i=1}^N \prod_{0 \leq t'_i < t_i} \mder_{e_j} \tilde F_i( y, [t_1,\ldots,t_{i-1},t'_i,0,\ldots,0] )$$
for any $1 \leq j \leq N$.  On the other hand, we have the telescoping identity
$$ \prod_{i=1}^N \prod_{0 \leq t'_i < t_i} \mder_{e_i} \tilde F_j( y, [t_1,\ldots,t_{i-1},t'_i,0,\ldots,0] ) = \tilde F_j( y, [t_1,\ldots,t_N] ) / \tilde F_j(y, 1)$$
and thus
\begin{equation}\label{mderej}
 \mder_{e_j} F( y, [t_1,\ldots,t_N] ) = \frac{ \tilde F_j( y, [t_1,\ldots,t_N] ) }{\tilde F_j(y, 1)} \prod_{i=1}^N \prod_{0 \leq t'_i < t_i} \omega_{ij}( y, [t_1,\ldots,t_{i-1},t'_i,0,\ldots,0] )
\end{equation}
where $\omega_{ij} := \frac{\mdersmall_{e_j} \tilde F_i}{\mdersmall_{e_i} \tilde F_j}$.

Since $\rho$ has degree $<O_{k}(1)$, the map $(y,u) \mapsto u$ is a $(\X,U)$-phase polynomial of degree $<O_{k}(1)$, which implies for any fixed $1 \leq i \leq N$ and $0 \leq t'_i < p^{n_i}$ that the map $(y,[t_1,\ldots,t_N]) \mapsto [t_1,\ldots,t_{i-1},t'_i,0,\ldots,0]$ is also a $(\X,U)$-phase polynomial of degree $<O_{k}(1)$, since the map from $[t_1,\ldots,t_N]$ to $[t_1,\ldots,t_{i-1},t'_i,0,\ldots,0]$ is a homomorphism.
By Lemma \ref{polycomp}(iii), we conclude that the functions $(y,[t_1,\ldots,t_N]) \mapsto \tilde F_i( y, [t_1,\ldots,t_{i-1},t'_i,0,\ldots,0] )$ are $(\X,S^1)$-phase polynomials of degree $<O_{k,m}(1)$ for all $1 \leq i \leq N$.  The map $(y,[t_1,\ldots,t_N]) \mapsto t_i$ is also a $(\X,S^1)$-phase polynomial of degree $<O_{k,m}(1)$.  Also, by \eqref{midori2}, the $\omega_{ij}$ are $(\X,S^1)$-phase polynomials of degree $<O_{k,m}(1)$.
We now claim that
$$ (y,[t_1,\ldots,t_N]) \mapsto \prod_{0 \leq t'_i < t_i} \omega_{ij}( y, [t_1,\ldots,t_{i-1},t'_i,0,\ldots,0] )$$
is also a $(\X,S^1)$-phase polynomial of degree $<O_{k,m}(1)$.  

When $p=O_{k,m}(1)$, this claim follows from Corollary \ref{func}, so suppose now that $p$ is sufficiently large depending on $k, m$.  Then $\omega_{ij}( y, [t_1,\ldots,t_{i-1},t'_i,0,\ldots,0] )$ is a phase polynomial of degree $O_{k,m}(1)$ in $t'_i$ that takes values in $C_p$.  By Taylor expansion we may thus write 
$$\omega_{ij}( y, [t_1,\ldots,t_{i-1},t'_i,0,\ldots,0] ) = \prod_{0 \leq j \leq O_{k,m}(1)} [\mder_{e_i}^j \omega_{ij}( y, [t_1,\ldots,t_{i-1},0,0,\ldots,0] )]^{\binom{t'_i}{j}}$$
and thus
$$ \prod_{0 \leq t'_i < t_i} \omega_{ij}( y, [t_1,\ldots,t_{i-1},t'_i,0,\ldots,0] ) = \prod_{0 \leq j \leq O_{k,m}(1)}
[\mder_{e_i}^j \omega_{ij}( y, [t_1,\ldots,t_{i-1},0,0,\ldots,0] )]^{\binom{t_i}{j+1}}.$$
The claim now follows from Lemma \ref{polycomp}.

Inserting the above claim into \eqref{mderej} we conclude that
$$ \tilde F_j \in \mder_{e_j} F (\pi^* F'_j) \cdot \Phase_{<O_{k,m}(1)}(\Fw,\X,S^1)$$
where $F'_j(y,u) := \tilde F_j(y,1)$, and hence by \eqref{mderif2}
$$ \mder_{e_j} f \in  (\mder_{e_j} \mder F) \cdot (\pi^* \mder F'_j) \cdot \Phase_{<O_{k,m}(1)}(\Fw,\X,S^1).$$
Thus if we set $f' := f / \mder F$, then $f$ is cohomologous to $f'$ and 
\begin{equation}\label{mderp-eq} \mder_{e_j} f' \in (\pi^* \mder F'_j) \cdot \Phase_{<O_{k,m}(1)}(\Fw,\X,S^1).
\end{equation}
Now we need to work on the $F'_j$ term.  From the telescoping identity
$\prod_{0 \leq t_j < p^{n_j}} V_{e_j^{t_j}} \mder_{e_j} f' = 1$
and \eqref{mderp-eq}, we have
$$ \pi^* \mder (F'_j)^{p^{n_j}} \in \Phase_{<O_{k,m}(1)}(\Fw,\X,S^1);$$
pushing forward by $\pi$ and then using \eqref{integ}, we conclude
$$ (F'_j)^{p^{n_j}} \in \Phase_{<O_{k,m}(1)}(\Fw,\Zcal_{<k-1}(\X),S^1).$$
In the case that $p$ is sufficiently large depending on $k, m$, we see from Lemma \ref{L:values-in-F_p} that we have the improvement
$ \pi^* \mder (F'_j)^{p^{n_j}}  = 1$,
and so $(F'_j)^{p^{n_j}}$ is constant in this case.  

We now claim that $F'_j = q_j F''_j$, where $q_j \in \Phase_{<O_{k,m}(1)}(\Fw,\Zcal_{<k-1}(\X),S^1)$ and $F''_j$ takes values in $C_{p^{n_j}}$.  For $p$ sufficiently large depending on $k,m$, this is immediate from the previous discussion; for $p=O_{k,m}(1)$, the claim follows instead from Corollary \ref{root}.

Inserting this claim back into \eqref{mderp-eq} we obtain
\begin{equation}\label{mderp-eq2} \mder_{e_j} f' \in (\pi^* \mder F''_j) \cdot \Phase_{<O_{k,m}(1)}(\Fw,\X,S^1).
\end{equation}
Now let $F^*$ be the $(\X,S^1)$-function
$ F^*(y, [t_1,\ldots,t_N]) := \prod_{j=1}^N F''_j(y)^{t_j};$
this is well-defined since $F''_j$ takes values in $C_{p^{n_j}}$.  We observe that $\pi^* \mder F''_j = \mder_{e_j} \mder F^*$.  Thus if $f'' := f'/F^*$, then $f''$ is $(\X,S^1)$-cohomologous to $f$ and
$\mder_{e_j} f'' \in \Phase_{<O_{k,m}(1)}(\Fw,\X,S^1)$
for all $1 \leq j \leq N$.  By repeated use of the cocycle identity $\mder_{uu'} f'' = (\mder_u f) V_u (\mder_{u'} f)$ we conclude that
$\mder_u f'' \in \Phase_{<O_{k,m}(1)}(\Fw,\X,S^1)$
for all $u \in U$.  Applying Lemma \ref{poly-integ} once for each $g \in \Fw$ we may thus write
$\mder_u f'' = \mder_u P$
for some $P \in \Phase_{<O_{k,m}(1)}(\Fw,\X,S^1)$.  Thus $f''/P$ is $U$-invariant, and so $f'' = P \pi^* \tilde f$ for some $(\Fw,\Zcal_{<k-1}(\X),S^1)$-function $\tilde f$, and Proposition \ref{propfin} follows.
\end{proof}

From Proposition \ref{propfin} and Proposition \ref{fin-prop}, we obtain Theorem \ref{main-6}, and thus Theorem \ref{main-2}.

%
%
\section{The high characteristic case}

We now develop high characteristic analogues of the above theory, establishing the sharp Theorem \ref{main-thm-high} instead of Theorem \ref{main-thm} in this setting.  The arguments here will be similar to those used to prove Theorem \ref{main-thm}; the main new difficulty is to be careful to not lose anything in the degree of various functions beyond what is absolutely necessary.

\subsection{Preliminary reductions}

Just as Theorem \ref{main-thm} follows from Theorem \ref{main-4}, Theorem \ref{main-thm-high} will follow from

\begin{theorem}[First reduction of high characteristic case]\label{main-high-char} 
Let $\F$ be a finite field, let $1 \leq k \leq \charac(\F)$, and let $\X$ be an ergodic $\Fw$-system of order $<k-1$.  Let $f \in Z^1_{<k}(\Fw, \X, S^1)$ be a $(\Fw, \X,S^1)$-cocycle of type $<k$.  Then $f$ is $(\Fw, \X, S^1)$-cohomologous to a $(\Fw, \X,S^1)$-phase polynomial of degree $<k$.
\end{theorem}

The deduction of Theorem \ref{main-thm-high} from Theorem \ref{main-high-char} is completely analogous to the corresponding derivation of Theorem \ref{main-thm} from Theorem \ref{main-4} and is omitted.  It remains to establish Theorem \ref{main-high-char}.  It should not be surprising that this will be accomplished by an induction on $k$.  But it will also be convenient to induct on a secondary parameter $j$, measuring the order of $\X$.  For technical reasons, in this induction the function $f$ will no longer be a cocycle, but instead have two weaker properties, which we refer to as being a \emph{quasi-cocycle} and a \emph{line cocycle} respectively:

\begin{definition}[Quasi-cocycles]\label{quasicocycle}  Let $G$ be a countable abelian group, let $\X$ be an ergodic $G$-system, and let $f$ be a $(G,\X,S^1)$-function.  Let $k \ge 0$.  We say that $f$ is a \emph{$(G,\X,S^1)$-quasi-cocycle of order $<k$} if, for every $g, g' \in G$, one has
$$ f(g+g',x) = f(g,T_{g'} x) f(g',x) p_{g,g'}(x)$$
for some $p_{g,g'} \in \Phase_{<k}(\X,S^1)$. 
\end{definition}

\begin{examples}  A $(G,\X,S^1)$-cocycle is precisely a $(G,\X,S^1)$-quasi-cocycle of order $<0$.  Every quasi-cocycle of order $<k$ is of course a quasi-cocycle of order $<k+1$; in particular, cocycles are quasi-cocycles of every order.  Every $(G,\X,S^1)$-phase polynomial of order $<k$ is also a $(G,\X,S^1)$-quasi-cocycle of order $<k$.  The space of $(G,\X,S^1)$-quasi-cocycles of order $<k$ form a group.  One can of course define this concept for  compact abelian groups other than $S^1$, but we will only need the $S^1$ quasi-cocycles in our arguments.
\end{examples}

\begin{definition}[Line cocycle]  Let $\F$ be a finite field of characteristic $p$, let $\X$ be an ergodic $\Fw$-system, and let $f$ be a $(\Fw,\X,S^1)$-function.  We say that $f$ is a \emph{line cocycle} if for every $g \in \Fw$, $\prod_{j=0}^{p-1} f(g, T_{g^j} x) = 1$ for $\mu$-a.e. $x$.
\end{definition}

\begin{example}  Every $(\Fw,\X,S^1)$-cocycle is a line cocycle, and the set of line cocycles forms a group.  In particular, any function cohomologous to a line cocycle is again a line cocycle.
\end{example}

It is clear that Theorem \ref{main-high-char} then follows from the $j=k$ case of

\begin{theorem}[Second reduction of high characteristic case]\label{main-high-char2} 
Let $\F$ be a finite field, let $1 \leq j, k \leq \charac(\F)$, and let $\X$ be an ergodic $\Fw$-system of order $<j$.  Let $f \in M_{<k}(\Fw, \X, S^1)$ be a $(\Fw,\X,S^1)$-function of type $<k$ which is also a line cocycle and a quasi-cocycle of order $<k-1$. Then $f$ is $(\Fw, \X, S^1)$-cohomologous to a $(\Fw, \X,S^1)$-phase polynomial $P$ of degree $<k$; furthermore, $P$ takes values in the cyclic group $C_p = \{ z \in S^1: z^p=1 \}$.
\end{theorem}

\subsection{Vertical differentiation}

We now begin the proof of Theorem \ref{main-high-char2}.  We first deal with the easy case $k=1$.  In this case, $f$ is a quasi-cocycle of order $<0$, and is thus a cocycle.  By hypothesis, $d^{[1]} f$ is a $(\Fw,\X^{[1]},S^1)$-coboundary, which by Lemma \ref{split} implies that $f$ is cohomologous to a constant $c(g)$, thus $f(g,x) = c(g) \mder_g F(x)$ for some $(\X,S^1)$-function $F$.  Since $f$ is a cocycle, $c$ is a character, and thus takes values in $C_p$, and the claim follows.

Now suppose that $2 \leq k \leq \charac(\F)$ and assume inductively that the claim has already been proven for smaller values of $k$.  In particular, Theorem \ref{main-high-char} holds for smaller values of $k$.  On the other hand, by repeating the proof of Theorem \ref{struc-thm}, we have the following structure theorem:

\begin{corollary}[Exact structure theorem, high characteristic]\label{struc-high}  Let $\F$ be a finite field, and let $1 \leq k \leq \charac(\F)$ be such that Theorem \ref{main-high-char} holds for all values smaller equal than $k$. Let $\X$ be an ergodic $\Fw$-system.  Then $\Zcal_{<1}(\X)$ is trivial, and for all $2 \leq j \leq k$, we can write
$ \Zcal_{<j}(\X) \equiv \Zcal_{<j-1}(\X) \times_{\rho_{j-1}} U_{j-1}$,
where $U_{j-1}$ is $\charac(\F)$-torsion and $\rho_{j-1}$ is a $(\Zcal_{<j-1}(\X),U_{j-1})$-phase polynomial of degree $<j-1$.  In particular, we have
$$ \Zcal_{<k}(\X) \equiv U_0 \times_{\rho_1} U_1 \times_{\rho_2} \ldots \times_{\rho_{k-1}} U_{k-1}$$
where $U_0$ is trivial.
\end{corollary}

Note that the torsion of the groups $U_j$ here is just $\charac(\F)$ rather than a power of $\charac(\F)$, due to the high characteristic hypothesis (see Lemma \ref{torlemma}).  By hypothesis, Corollary \ref{struc-high} is applicable for our fixed value of $k$.

When $j=1$, then $\X$ is a point, and the claim is trivial.  Now suppose $2 \leq j \leq \charac(\F)$ and assume inductively that the claim has already been proven for the same value of $k$ and all smaller values of $j$.

We first deal with the low order case $j \leq k$, returning to the high order case $j > k$ later.  We use Corollary \ref{struc-high} to write $\X = \Zcal_{<j-1}(\X) \times_{\rho_{j-1}} U_{j-1}$.

Let $t \in U_{j-1}$.  We observe the following properties of the $(\Fw,\X,S^1)$-function $\mder_t f$:

\begin{lemma}[Exact differentiation]\label{exactdiff} $\mder_t f$ is a line cocycle, is of type $<k-j+1$, and is a quasi-cocycle of order $<k-j$.
\end{lemma}

\begin{proof} The first claim follows from the fact that $f$ is a line cycle, and that the action $V_t$ of $t$ on $\X$ commutes with the $\Fw$ action.
The second claim follows from Lemma \ref{vert-lem}.

Finally, we prove the quasi-cocycle claim.  As the action $V_t$ of $t$ commutes with the action of $\Fw$, it suffices from Definition \ref{quasicocycle} to show that $\mder_t$ maps $\Phase_{<k-1}(\X,S^1)$ to $\Phase_{<k-j}(\X,S^1)$.

Let $f \in \Phase_{<k-1}(\X,S^1)$, then by Lemma \ref{basic}(iii), $d^{[k-1]} f = 0$ $\mu^{[k-1]}$-a.e..  Let $\alpha$ be a $(k-j)$-face of $\2^{k-1}$.  By Lemma \ref{erglem}(iv), $(V_t)^{[k-1]}_\alpha$ preserves $\mu^{[k-1]}$.  We conclude that $(V_t)^{[k-1]}_\alpha(d^{[k-1]} f) = 0$ $\mu^{[k-1]}$-a.e..  Dividing these two equations, we conclude that
$$ (V_t)^{[k-1]}_\alpha(d^{[k-1]} f) / (d^{[k-1]} f) = (\partial(\alpha)_*)^*( d^{[k-j]} \mder_t f ) = 0$$
$\mu^{[k-1]}$-a.e.  Since $\partial(\alpha)_*$ pushes forward $\mu^{[k]}$ to $\mu^{[k-j]}$ (see Lemma \ref{ccs}), we conclude that $d^{[k-j]} \mder_t f=0$ $\mu^{[k-j]}$-a.e..    Applying Lemma \ref{basic}(iii) again we conclude that $\mder_t f \in \Phase_{<k-j}(\X,S^1)$ as required.
\end{proof}

By the induction hypothesis, $\mder_t f$ is $(\Fw,\X,S^1)$-cohomologous to a $(\Fw,\X,S^1)$-phase polynomial $q_t$ of degree $<k-j+1$ which takes values in $C_p$.  Since $\mder_t f$ is a line cocycle and a quasi-cocycle of order $<k-j$, $q_t$ is also.


\subsection{Reduction to the finite $U$ case}\label{Red}

We now argue (as in Proposition \ref{fin-prop}) that in order to conclude the proof of Theorem \ref{main-high-char2}, it suffices to do so in the case when the vertical structure group $U_{j-1}$ is finite.

To do this, we follow the procedure in as in Proposition \ref{fin-prop}, to make $q_t$ a $U'$-cocycle 
for an open subgroup $U'$ of $U_{j-1}$.  Note that the modifications done to $q_t$
in this procedure are by $\Fw$-cocycles of degree $<k-j+1$, so none of the properties of 
$q_t$ are damaged (it is still a $k-j+1$ side cocycle, a line cocycle, and of the correct degree).

As in Proposition \ref{fin-prop}, we can write $U_{j-1} = U' \times W$ for some finite $W$, and write $\X = \Y \times_{\rho'} U'$, where $\Y := \Zcal_{<j-1}(\X) \times_{\rho''} W$ and $\rho', \rho''$ are the projections of $\rho_{j-1}$.  Note that as $\rho_{j-1}$ is of degree $<j$, $\rho'$ is also.  We write $x \in X$ as $x = (y,u)$ for $y \in Y$ and $u \in U'$.

We now invoke the following variant of Lemma \ref{poly-integ} which is more efficient with the degree.

\begin{proposition}[Exact integration]\label{exact_int}  Let $G$ be a countable abelian group, let $j \geq 0$, let $U$ be a compact abelian group, and let $\X= \Y \times_\rho U$ be an ergodic $G$-system with $\Y \geq \Zcal_{<j}(\X)$, where $\rho$ a $(G,\Y,U)$-phase polynomial cocycle of degree $<j$.  For  any $t \in U$, let $p_t(x)$ be a $(\X,S^1)$-phase polynomial of degree $<l$, and suppose that for any $t,s \in U$, $p_{t\cdot s}(x)=p_t(V_sx)\cdot p_s(x)$. 
Then there exists a $(\X,S^1)$-phase polynomial $Q(x)$ of degree $<l+j$ such that $\mder_t Q(x)=p_t(x)$. 
Furthermore we can take $Q(y,u u_0):=p_u(y,u_0)$ for some $u_0 \in U$. 
\end{proposition}

\begin{remark} In the converse direction, one can show (by using the properties of the nilpotent group ${\mathcal G}^{[k]}$ studied in \cite{hk-cubes}) that if $Q$ has degree $<l+j$, then $\mder_t Q$ has degree $<l$.  This may explain the terminology ``exact''.
\end{remark}

We will prove this proposition in Section \ref{xac}.  Assuming it for now, we see (by applying it once for each $g$) that we can write $q_t = \mder_t Q$ for all $t \in U'$ and some $(\Fw,\X,S^1)$-phase polynomial of degree $<k$, such that $Q(g,y,uu_0) = q_u(g,y,u_0)$ for all $y \in Y$, $u \in U'$, and some $u_0 \in U'$.  Since $q_t$ takes values in $C_p$, we see that $Q$ does also.

We now claim that $Q$ is a $(\Fw,\X,S^1)$-quasicocycle of order $<k-1$.  Indeed, for any $g, h \in G$ and $x = (y,uu_0) \in \X$, we have
\begin{align*}
\frac{Q(g+h,x)}{Q(g,x) Q(h,T_g x)} &= \frac{Q(g+h,x)}{Q(g,x) Q(h,x) \mder_g Q(h,x)} \\ 
&= \frac{q_u(g+h,y,u_0)}{q_u(g,y,u_0) q_{u}(h, y, u_0) \mder_g Q(h,x)} \\
&= \frac{q_u(g+h,y,u_0)}{q_u(g,y,u_0) q_u(h,T_g y,\rho_{j-1}(g,y)u_0)} \frac{q_u(h,T_g y,\rho_{j-1}(g,y)u_0)}{q_{u}(h, y, u_0) \mder_g Q(h,x)} \\
&= P_{u,g,h}(y, u_0) \frac{\mder_g q_u(h,x)}{\mder_g Q(h,x)} 
\end{align*}
where
$ P_{u,g,h}(x) := \frac{q_u(g+h,x)}{q_u(g,x) q_u(h,T_g x)}.$
Fix $g,h$.  Since $q_u$ and $Q$ are $(\Fw,\X,S^1)$-phase polynomials of degree $<k-j+1$ and $<k$ respectively, we see that $\frac{\mdersmall_g q_u(h,x)}{\mdersmall_g Q(h,x)}$ is a phase polynomial in $x$ of degree $<k-1$.  Also, as $q_u$ is a quasi-cocycle of order $<k-j$, we see that $P_{u,g,h}$ is a $(\X,S^1)$-phase polynomial of degree $<k-j$.  Since $q_u$ is a cocycle in $u$, $P_{u,g,h}$ is also.  Applying Proposition \ref{exact_int} we conclude that $P_{u,g,h}(y,u_0)$ is a $(\X,S^1)$-phase polynomial of degree $<k-1$.  Putting this all together we see that $\frac{Q(g+h,x)}{Q(g,x) Q(h,T_g x)}$ is a $(\X,S^1)$-phase polynomial of degree $<k-1$, and the claim follows.

Write $f' := f / Q$, then (as in Proposition \ref{fin-prop}) $f'$ is a $(\Fw,\X,S^1)$-function with the property that $\mder_u f'$ is a $(\Fw,\X,S^1)$-coboundary for all $u \in U'$.  Applying Lemma \ref{straighten-lemma} just as in Proposition \ref{fin-prop}, we conclude that $f'$ is $(\Fw,\X,S^1)$-cohomologous to an $(\Fw,\X,S^1)$-function $f''$ which is invariant with respect to some open subgroup $U''$ of $U'$.  Thus we can write $f'' = \pi^* \tilde f$ for some $(\Fw,\X',S^1)$-function $\tilde f$, where $\X' := \Zcal_{<j}(\X) \times_{\sigma \circ \rho_{<j-1}} U_{j-1}/U''$, $\sigma: U_{j-1} \to U_{j-1}/U''$ is the quotient map, and $\pi: \X \to \X'$ is the factor map.

Since $Q$ takes values in $C_p$ and is of degree $<p$, we see from Lemma \ref{L:values-in-F_p} that $Q$ is a line cocycle. Since $\pi^* \tilde f$ is cohomologous to $f/Q$, we conclude that $\pi^* \tilde f$, and hence $\tilde f$, are also line cocycles.  Similarly, since $f, Q$ are $(\Fw,\X,S^1)$-quasicocycles of order $<k-1$, $\pi^* \tilde f$ is also, which implies in turn that $\tilde f$ is a $(\Fw,\X',S^1)$-quasicocycle of order $<k-1$ (cf. Lemma \ref{cocycle}).

From Lemma \ref{basic}(iii) see that $Q$ has type $<k$.  Since $f$ also has type $<k$, we conclude that $\pi^* \tilde f$ has type $<k$ also.  We now use the following variant of Lemma \ref{descent}, which does not concede any losses in the type:

\begin{proposition}[Exact descent]\label{exact_descent} Let $G$ be a discrete countable abelian group and $k \geq 0$.
Let $\X$ be an ergodic $G$-system of order $<k$. Let $\Y$ be a factor of $\X$, with factor map $\pi:\X \to \Y$.  Suppose that $f$ is a $(G,\Y,S^1)$-quasi-cocycle of order $<k$.  If $\pi^* f$ is of type $<k$, then $f$ is of type $<k$.
\end{proposition}

We will prove this proposition in Section \ref{xd}.  Assuming it for now, we conclude that $\tilde f$ is of type $<k$.  Observe that $\X'$ is an extension of $\Zcal_{<j-1}(\X) = \Zcal_{<j-1}(\X')$ by a finite abelian group by a cocycle of degree $<j-1$. If Theorem \ref{main-high-char2} (for this choice of $k$ and $j$) has already been established in the case when $U_{j-1}$ is finite, then $\tilde f$ is $(\Fw,\X',S^1)$-cohomologous to a $(\Fw,\X',S^1)$-phase polynomial of degree $<k$.  Pulling this back by $\pi$, we conclude that $f''$ is $(\Fw,\X,S^1)$-cohomologous to a $(\Fw,\X,S^1)$-phase polynomial of degree $<k$.  Since $f''$ is $(\Fw,\X,S^1)$-cohomologous to $f/Q$, and $Q$ is also a $(\Fw,\X,S^1)$-phase polynomial of degree $<k$, we conclude that $f$ is $(\Fw,\X,S^1)$-cohomologous to a 
$(\Fw,\X,S^1)$-phase polynomial of degree $<k$, and Theorem \ref{main-high-char2} follows.

The remaining tasks (in the low order case $j \leq k$) are to verify Theorem \ref{main-high-char2} in the case of finite $U_{j-1}$, and to also verify Proposition \ref{exact_int} and Proposition \ref{exact_descent}.

\subsection{The finite group case}\label{fingp}

We now establish Theorem \ref{main-high-char2} in the case when $U_{j-1}$ is finite.  This is the analogue of Proposition \ref{propfin}, but our arguments here are somewhat simpler thanks to the high characteristic (which allows us to use the full power of Lemma \ref{L:values-in-F_p}).

Since $U_{j-1}$ is finite and $p$-torsion, we can write $U_{j-1} = C_p^L$ for some finite $L$.  We will now induct on the dimension $L$.  The case $L=0$ is trivial, so suppose inductively that $L \geq 1$ and that the claim has already been proven for $L-1$.  We write $U_{j-1} = C_p^{L-1} \times \langle e \rangle$, where $e$ is a generator of $C_p$.  Recall that $\mder_e f$ is $(\Fw,\X,S^1)$-cohomologous to a $(\Fw,\X,S^1)$-phase polynomial $q_e$ of degree $<k-j+1$ which takes values in $C_p$.  We now extend this from $e$ to $\langle e \rangle$ in a manner which is a cocycle with respect to this parameter. 

Fix $q_e$, and then define the $(\Fw,\X,S^1)$-function $q_{e^s}$ for all $0 \leq s < p$ by the formula
$
q_{e^s}(g,x) := \prod_{i=0}^{s-1} q_e(g, V_e^i x).
$
Since $q_e$ is a $(\Fw,\X,S^1)$-phase polynomial of degree $<k-j+1$, $q_{e^s}$ is also.  Since $\mder_e f$ is $(\Fw,\X,S^1)$-cohomologous to $q_e$, we see from the cocycle identity 
\begin{equation}\label{mder}
\mder_{e^s} f(g,x) = \prod_{i=0}^{s-1} \mder_e f(g, V_e^i x)
\end{equation}
that $\mder_u f$ is $(\Fw,\X,S^1)$-cohomologous to $q_u$ for all $u \in \langle e \rangle$.

By repeated application of Lemma \ref{exactdiff}, we know that $q_e$ has degree $<p$ with respect to differentiation $\mder_e$ in the $e$ direction.  By Lemma \ref{L:values-in-F_p}, we conclude that
$\prod_{i=0}^{p-1} q_e(g, V_e^i x) = 1$,
and thus the $q_{e^s}$ form a cocycle in the $e^s$ variable, in the sense that
\begin{equation}\label{quv}
 q_{uv}(g,x) = q_u(g,x) q_v(g,V_u x)
\end{equation}
for all $g \in \Fw$, $u,v \in \langle e \rangle$, and $\mu$-a.e. $x$.  Applying Proposition \ref{exact_int}, we may find a $(\Fw,\X,S^1)$-phase polynomial $Q$ of degree $<k$ such that $q_u = \mder_u Q$ for all $u \in \langle e \rangle$.  Since $q_e$ takes values in $C_p$, $Q$ does also, and thus (by Lemma \ref{L:values-in-F_p}) is a line cocycle.

By repeating the arguments in the previous section, we also see that $Q$ is a $(\Fw,\X,S^1)$-quasi-cocycle of order $<k-1$. 

Recall that $\mder_e f$ is $(\Fw,\X,S^1)$-cohomologous to $q_e = \mder_e Q$, thus we can find an $(\X,S^1)$-function $F_e$ such that
$\mder_e f = (\mder_e Q) \mder F_e$.
Using the cocycle identity \eqref{mder} for $f$ and $Q$ we conclude that
$ \mder \prod_{i=0}^{p-1} V_e^i F_e = 1$
and thus by ergodicity, $\prod_{i=0}^{p-1} V_e^i F_e$ is equal to some constant in $S^1$.  Taking $p^{th}$ roots, we can express this constant as $c^p$ for some $c \in S^1$; dividing the $F_e$ by this constant (which does not affect any of the properties of $F_e$), we may take $c=1$, thus
$ \prod_{i=0}^{p-1} V_e^i F_e = 1.$
If we then define
$ F_{e^s} := \prod_{i=0}^{s-1} V_e^i F_e $
for $0 \leq s < p$, then the $F_{e^s}$ form a cocycle in the $e^s$ variable in the sense of \eqref{quv}.  Applying Lemma \ref{free-lem}, this cocycle is a coboundary, thus we can write $F_e = \mder_e F$ for some $(\X,S^1)$-function $F$.  We conclude that
$\mder_e \frac{f}{Q \mder F} = 1,$
thus $f/Q$ is cohomologous to a function which is $\langle e \rangle$-invariant.  We can now argue as in the preceding subsection (with $\langle e \rangle$ playing the role of $U''$) to deduce Theorem \ref{main-high-char2} for $U_{j-1} = C_p^L$ from the corresponding claim for $U_{j-1} = C_p^{L-1}$, which we have by induction.  This establishes Theorem \ref{main-high-char2} in the low order case $j \leq k$, contingent on Propositions \ref{exact_int} and \ref{exact_descent}.

\subsection{The high order case}

We now modify the above arguments to deal with the high order case $j > k$.  We need a key definition: we say that a $\2^k$-tuple $(f_\w)_{\w \in \2^k}$ of $(\Fw,\X,S^1)$-functions $f_\w$ for $\w \in \2^k$ is a \emph{good $(\Fw,\X,S^1)$-tuple} if the following properties hold:
\begin{itemize}
\item[(i)] Each $f_\w$ is a line cocycle and a $(\Fw,\X,S^1)$-quasi-cocycle of order $<k-1$.
\item[(ii)] For each face $\alpha$ of $\2^k$, $\prod_{\w \in \alpha} f_\w$ is a $(\Fw,\X^{[k]},S^1)$-cocycle.
\item[(iii)] $\prod_{\w \in \2^k} f_\w$ is a $(\Fw,\X^{[k]},S^1)$-coboundary.
\end{itemize}

From hypothesis we see that $(f^{\sgn(\w)})_{\w \in \2^k}$ is a good $(\Fw,\X,S^1)$-tuple.  Observe also that if $(f_\w)_{\w \in \2^k}$ is a good $(\Fw,\X,S^1)$-tuple, and we replace one or more of the $f_\w$ by a $(\Fw,\X,S^1)$-cohomologous function, then we still obtain a good $(\Fw,\X,S^1)$-tuple.

We now claim the following proposition:

\begin{proposition}[Descent of good tuples]\label{good-descent}  Let $1 \leq k < j \leq \charac(\F)$, and let $\X = \Zcal_{<j-1}(\X) \times_{\rho_{j-1}} U_{j-1}$ be an ergodic $\Fw$-system of order $<j$.  Suppose that $(f_\w)_{\w \in \2^k}$ is a good $(\Fw,\X,S^1)$-tuple.  Then there exists a good $(\Fw,\Zcal_{<j-1}(\X),S^1)$-tuple $(\tilde f_\w)_{\w \in \2^k}$ such that for every $\w$, $f_\w$ is $(\Fw,\X,S^1)$-cohomologous to $(\pi^\X_{\Zcal_{<j-1}(\X)})^* \tilde f_\w$, where $\pi^\X_{\Zcal_{<j-1}(\X)}: \X \to \Zcal_{<j-1}(\X)$ is the factor map.
\end{proposition}

\begin{proof}  As in previous arguments, we first reduce to the case when $U_{j-1}$ is finite, and then establish the finite case.

If $t \in U_{j-1}$, then by Lemma \ref{erglem}(iv), the transformation $(V_t)_{\{\w\}}^{[k]}$ preserves $\mu^{[k]}$ for every $\w \in \2^k$.  Arguing as in the proof of Lemma \ref{exactdiff}, we thus conclude that $\mder_t f_\w$ is a $(\Fw,\X,S^1)$-coboundary for every $\w \in \2^k$.  By repeatedly applying Lemma \ref{straighten-lemma}, we conclude that there exists an open subgroup $U''$ of $U_{j-1}$ such that each $f_\w$ is $(\Fw,\X,S^1)$-cohomologous to an $U''$-invariant function, which we can write as $\pi^* \tilde f_\w$ for some $(\Fw,\X',S^1)$-function $\tilde f_\w$, where $\X' := \Zcal_{<j}(\X) \times_{\sigma \circ \rho_{<j-1}} U_{j-1}/U''$, $\sigma: U_{j-1} \to U_{j-1}/U''$ is the quotient map, and $\pi: \X \to \X'$ is the factor map, thus we have
$ f_\w = (\pi^* \tilde f_\w) \mder F_\w$
for some $(\X,S^1)$-functions $\F_w$.

Arguing as in Section \ref{Red} (and using Lemma \ref{cocycle}), we see that the tuple $(\tilde f_\w)_{\w \in \2^k}$ obeys properties (i) and (ii) of being a good $(\Fw,\X',S^1)$-tuple.  Unfortunately, it need not obey (iii); we know that $\bigotimes_\w \pi^* \tilde f_\w$ is a $(\Fw,\X^{[k]},S^1)$-coboundary, but this does not automatically imply that $\bigotimes_\w \tilde f_\w$ is a $(\Fw,(\X')^{[k]},S^1)$-coboundary.  However, in the high order case $j > k$, we see from Lemma \ref{conditional-product} that $\X^{[k]}$ is an abelian extension $\X^{[k]} = \Zcal_{<j-1}(\X)^{[k]} \times_{\rho_{<j-1}^{[k]}} U_{j-1}^{[k]}$, where the $(\Fw,\Zcal_{<j-1}(\X)^{[k]},U_{j-1}^{[k]})$-cocycle $\rho_{<j-1}^{[k]} = \bigotimes_{\w \in \2^k} \rho_{<j-1}$ is the tensor product of $2^k$ copies of the $(\Fw,\Zcal_{<j-1}(\X),U_{j-1})$-cocycle $\rho$, and similarly for $(\X')^{[k]}$.  Applying Lemma \ref{descent-lem}, we conclude that $\bigotimes_{\w \in \2^k} \tilde f_\w$ is $(\Fw, (\X')^{[k]}, S^1)$-cohomologous to $\chi^{[k]} \circ \rho^{[k]} \circ (\pi^{\X'}_{\Zcal_{<j-1}(\X)})^{[k]}$ for some character $\chi^{[k]} \in \hat{U_{j-1}^{[k]}}$, where $\pi^{\X'}_{\Zcal_{<j-1}(\X)}: \X' \to \Zcal_{<j-1}(\X)$ is the factor map.  We can factorize the latter as a tensor product
$$ \chi^{[k]} \circ \rho^{[k]} \circ (\pi^{\X'}_{\Zcal_{<j-1}(\X)})^{[k]} = \bigotimes_{\w \in \2^k} \chi_\w \circ \rho_{<j-1} \circ \pi^{\X'}_{\Zcal_{<j-1}(\X)} =: \bigotimes_{\w \in \2^k} p_\w$$
for some characters $\chi_\w \in \hat U_{j-1}$.  We thus see that $\bigotimes_\w f'_\w$ is a $(\Fw,(\X')^{[k]},S^1)$-coboundary.  
Since $\rho_{<j-1} \circ \pi^{\X}_{\Zcal_{<j-1}(\X)}$ is a $(\Fw,\X,S^1)$-coboundary (being the derivative of the coordinate function $(y,u) \mapsto u$), we see that $\pi^* p_\w$ is also a $(\Fw,\X,S^1)$-coboundary.  Thus $f_\w$ is $(\Fw,\X,S^1)$-cohomologous to $\pi^* f'_\w$.  Also, since $(\tilde f_\w)_{\w \in \2^k}$ obeys properties (i), (ii) of being a $(\Fw,\X',S^1)$-good tuple, $(f'_\w)_{\w \in \2^k}$ does also.  Since we have already established (iii), we conclude that $(f'_\w)_{\w \in \2^k}$ is a $(\Fw,\X',S^1)$-good tuple.  Thus Proposition \ref{good-descent} for $U_{j-1}$ will follow from that for $U_{j-1}/U''$, thus reducing matters to the case when $U_{j-1}$ is finite.

We now repeat the arguments from Section \ref{fingp}.  As in that section, we can write $U_{j-1} = C_p^L = C_p^{L-1} \times \langle e \rangle$ for some $e \in C_p$ and induct on $L$.  Arguing as in the start of this proof (i.e. using the invariance of $\mu^{[k]}$ with respect to $(V_e)_{\{\w\}}^{[k]}$), we know that $\mder_e f_\w$ is a $(\Fw,\X,S^1)$-coboundary, thus we can find a $(\X,S^1)$-function $F_{e,\w}$ such that
$\mder_e f_\w = \mder F_{e,\w}$.
Using the cocycle identity \eqref{mder}, this implies that
$ \mder \prod_{i=0}^{p-1} V_e^i F_{e,\w} = 1$
and so (by ergodicity) $\prod_{i=0}^{p-1} V_e^i F_{e,\w}$ is constant.  Dividing $F_{e,\w}$ by the $p^{th}$ root of this constant as in Section \ref{fingp} we may thus take
$ \prod_{i=0}^{p-1} V_e^i F_{e,\w} = 1$.
Applying Lemma \ref{free-lem} as in Section \ref{fingp}, we can write $F_{e,\w} = \mder_e F_\w$ for some $(\X,S^1)$-function $F_\w$.  We conclude that $\mder_e (f_\w / \mder F_\w) = 1$, i.e. $f_\w$ is cohomologous to an $e$-invariant function.  From this we see (as in the reduction to the finite $U_{j-1}$ case) that Proposition \ref{good-descent} for $C_p^L$ will follow from Proposition \ref{good-descent} for $C_p^{L-1}$.  Since this proposition is trivial when $L=0$, the claim follows.
\end{proof}

Applying this proposition iteratively, starting with the good $(\Fw,\X,S^1)$-tuple $(f^{\sgn(\w)})_{\w \in \2^k}$ and decrementing $j$, we see that there exists a good $(\Fw,\Zcal_{<k}(\X),S^1)$-tuple $(\tilde f_\w)_{\w \in \2^k}$ such that $f^{\sgn(\w)}$ is $(\Fw,\X,S^1)$-cohomologous to $(\pi^\X_{\Zcal_{<k}(\X)})^* \tilde f_\w$ for all $\w \in \2^k$.

We now invoke the following lemma:

\begin{lemma}[Cauchy-Schwarz-Gowers for finite type]\label{fint-csg}  Let $G$ be a countable abelian group, let $k \geq 0$, let $\X$ be an ergodic $G$-system of order $<k$, and for each $\w \in \2^k$, let $f_\w$ be a $(G,\X,S^1)$-quasi-cocycle of order $<k$ such that $\bigotimes_\w f_\w$ is a $(G,\X^{[k]},S^1)$-coboundary.  Then each $f_\w$ is of type $<k$.
\end{lemma}

We prove this lemma later in Section \ref{xd}.  Assuming it for now, we conclude in particular that $\tilde f_{-\1}$ is of type $<k$.  It is also a line cocycle and a $(\Fw,\Zcal_{<k}(\X),S^1)$-quasi-cocycle of order $<k-1$.  Since we have already established Theorem \ref{main-high-char2} in the case $j=k$, we conclude that $\tilde f_{-\1}$ is is $(\Fw, \Zcal_{<k}(\X), S^1)$-cohomologous to a $(\Fw, \Zcal_{<k}(\X),S^1)$-phase polynomial $P$ of degree $<k$ that takes values in $C_p$.  Since $f$ is $(\Fw,\X,S^1)$-cohomologous to $(\pi^\X_{\Zcal_{<k}(\X)})^* \tilde f_{-\1}$, we conclude that $f$ is $(\Fw,\X,S^1)$-cohomologous to the $(\Fw,\X,S^1)$-phase polynomial $(\pi^\X_{\Zcal_{<k}(\X)})^* P$, and Theorem \ref{main-high-char2} then follows.

Our remaining tasks are to prove Proposition \ref{exact_int}, Proposition \ref{exact_descent}, and Lemma \ref{fint-csg}.

\subsection{Exact integration}\label{xac}
 
In this subsection we establish Proposition \ref{exact_int}.  We begin with the analogue of Lemma \ref{polycomp}.

\begin{lemma}[Refined composition of polynomials]\label{mod}
Let $G$ be a countable abelian group, let $j \geq 0$, let $U$ be a compact abelian group, and let $\X= \Y \times_\rho U$ be an ergodic $G$-system such that $\Y \geq \Zcal_{>j}(\X)$.  For any $t \in U$,
let $p_t$ be a $(\X,S^1)$-phase polynomial of degree $<l$ for some $l \geq 1$, and suppose that for any $t,s \in U$, we have the cocycle property
\begin{equation}\label{cocycle-p}
p_{t\cdot s}(x)=p_t(V_sx)\cdot p_s(x).
\end{equation}
Let $r$ be a $(\X,U)$-phase polynomial of degree $<m$ for some $1 \leq m \leq j+1$.  For any $x \in X$, write $x=(y,u)$, $y \in \Y$, $u \in U$.
\begin{itemize}
\item[(i)]  The function $(y,u) \mapsto p_{r(y,u)}(y,u)$ is a $(\X,S^1)$-phase polynomial of degree $<l+m-1$.
\item[(ii)]  More generally, if $q$ is a $(\X,U)$-phase polynomial of degree $<j+1$, then $(y,u) \mapsto p_{r(y,u)}(y,u q(y,u))$ is a $(\X,S^1)$-phase polynomial of degree $<l+m-1$.
\end{itemize}
\end{lemma}

\begin{proof}  We prove (i) and (ii) simultaneously by induction on $l$.  If $l=1$, then $p_t$ is a constant (by ergodicity), and so the map $t \mapsto p_t$ is a homomorphism, and the claim (i) then easily follows.  Also in this case (ii) is clearly equivalent to (i).

Now suppose inductively that $l \geq 2$, and that the claim has already been proven for $l-1$.  We now induct on $m$.  When $m=1$ then $r$ is constant, and (i) is clear.

Now we show (for any $m$) that (i) implies (ii).  Suppose first that $l \leq j$.  Then $p_t$ is measurable with respect to $\Abr_{<l}(\X) \leq \Abr_{<j}(\X)$, and hence measurable with respect to $\Zcal_{<j}(\X)$ (and hence $\Y$) by Lemma \ref{easy-thm}.  In particular, $p_t$ is $U$-invariant, and (i) and (ii) are clearly equivalent.  Now suppose instead that $l>j$.  The action of $U$ preserves $\Zcal_{<j}(\X)$, and thus by Lemma \ref{erglem}(iv), the measure $\mu^{[l]}$ is invariant under the face transformations $(V_u)^{[l]}_\alpha$ for any $(l-j)$-face $\alpha$ and any $u \in U$.  As $U$ is abelian, this implies that $\mu^{[l]}$ is invariant under any $(V_{u_\w})_{\w \in \2^l}$, where the $u_\w$ obey the conditions $\prod_{\w \in \alpha} u_{\w}^{\sgn(\w)}=1$ for every $(j+1)$-face $\alpha$.  Since $q$ has degree $<1$, we conclude (by Lemma \ref{basic}(iii)) that $\mu^{[l]}$ is invariant under $(V_{q(y_\w,u_\w)})_{\w \in \2^l}$ for $\mu^{[l]}$-a.e. $(y_\w,u_\w)_{\w \in \2^l}$.  On the other hand, since $p_t$ has degree $<l$ for every $t \in U$, we have $\prod_{\w \in \2^l} p_t( y_\w, u_\w )^{\sgn(\w)} = 1$ $\mu^{[l]}$-a.e. for such $t$ by Lemma \ref{basic}(iii).  We conclude that 
$\prod_{\w \in \2^l} p_t( y_\w, u_\w q(y_\w,u_\w) )^{\sgn(\w)} = 1$ $\mu^{[l]}$-a.e. By one last application of Lemma \ref{basic}(iii) we obtain that $(y,u) \mapsto p_t(y,uq(y,u))$ is of degree $<l$.  The claim (ii) now follows from (i).

Finally, we assume that (i) (and hence (ii)) have been proven for $m-1$, and then establish (i) for $m$.  Let $F$ be the $(\X,S^1)$-function $F(x) := p_{r(x)}(x)$.  From \eqref{cocycle-p} we have
$$ F(T_g x) = p_{r(x) \mdersmall_g r(x)}(T_g x) = p_{r(x)}(T_g x) p_{\mdersmall_g r(x)}(V_{r(x)} T_g x)$$
and thus (since the action of $U$ commutes with that of $G$)
$$ \mder_g F(x) = (\mder_g p_{r(x)})(x) (T_g p_{\mdersmall_g r(x)})( V_{r(x)} x ).$$
Observe that for each $t \in U$, $\mder_g p_t$ is a $(\X,S^1)$-phase polynomial of degree $<l-1$, and that $\mder_g p_t$ obeys \eqref{cocycle-p}.  By the induction hypothesis (i) with $l$ replaced by $l-1$, $(\mder_g p_{r(x)})(x)$ is a $(\X,S^1)$-phase polynomial of degree $<l+m-2$.  Similarly, $T_g p_t$ is a $(\X,S^1)$-phase polynomial of degree $<l$ that also obeys \eqref{cocycle-p}, $\mder_g r$ is a $(\X,U)$-phase polynomial of degree $<m-1$, and $r$ has degree $<j+1$ by hypothesis on $m$.  Applying the induction hypothesis (ii) with $m$ replaced by $m-1$, we conclude that $(T_g p_{\mdersmall_g r(x)})( V_{r(x)} x )$ is a $(\X,S^1)$-phase polynomial of degree $<l+m-2$.  Putting this all together, we see that $\mder_g F$ has degree $<l+m-2$ for all $g$, and thus (by \eqref{integ}) $F$ has degree $<l+m-1$, establishing (i) as required.
\end{proof}

Now we can prove Proposition \ref{exact_int}.

\begin{proof}[Proof of Proposition \ref{exact_int}]
Write $x=(y,u)$, $y \in \Y$, $u \in U$.   
Take $Q(y,u u_0):=p_u(y,u_0)$ for $u_0 \in U$ a generic point (actually we can take $u_0=1$ since polynomials are continuous).  Now as in Lemma \ref{poly-integ}, $\mder_t Q(x)=p_t(x)$.  It remains to show that $Q$ has degree $<l+j$.  But this follows from Lemma \ref{mod}(ii) with $r(y,u) := u/u_0$, $q(y,u) := u_0/u$, and $m:=j+1$; note that as $\rho_j$ has degree $<j$, $r$ and $q$ have degree $<j+1$.
\end{proof}

\subsection{Exact descent}\label{xd}

We now prove Proposition \ref{exact_descent} and Lemma \ref{fint-csg}.  As noted already in Remark \ref{descent-remark}, an exact descent result for cocycles already appears in \cite[Corollary 7.8]{hk-cubes}; Proposition \ref{exact_descent} can be viewed as an extension of that result to quasi-cocycles.

Our main tool for both of these tasks is the following equivalent characterization of the finite type condition.

\begin{lemma}[Characterization of finite type]\label{fint}  Let $G$ be a countable abelian group, let $k \geq 0$, let $\X$ be an ergodic $G$-system of order $<k$, and let $f$ be a $(G,\X,S^1)$-quasi-cocycle of order $<k-1$.  Let $\Phi_n$ be a F{\o}lner sequence for $G$.  Then the following two statements are equivalent.
\begin{itemize}
\item[(i)] $f$ is of type $<k$ (i.e. $d^{[k]} f$ is a $(G,\X^{[k]},S^1)$-coboundary).
\item[(ii)] For all $\x$ in a set of positive $\mu^{[k]}$-measure, $\limsup_{n \to \infty} |\E_{g \in \Phi_n} d^{[k]} f(g,\x)| \neq 0$.
\end{itemize}
\end{lemma}

\begin{proof}  We first show that (i) implies (ii).  By hypothesis, we have
$ d^{[k]} f(g,\x) = F( T_g^{[k]} \x ) \overline{F( \x )}$
for all $g \in G$, $\mu^{[k]}$-a.e. $\x$, and some $(\X^{[k]},S^1)$-function $F$.  We rearrange this as
$$ f(g,x_{-\1}) = F( (T_g x_\w)_{\w \in \2^k} ) \overline{F( (x_\w)_{\w \in \2^k} )} \prod_{\w \in \2^k \backslash -\1} f_\w(g,x_\w)$$
where each $f_\w$ is either equal to $f$ or its complex conjugate.

As $\X$ is of order $<k$, we see from Definition \ref{ucf-soft-def} that any measurable function on $\X^{[k]}$ that depends only on the first coordinate $x_{-\1}$, is equal $\mu^{[k]}$-a.e. to a function that is independent of this coordinate.  Since these two classes of functions together generate all measurable functions on $\X^{[k]}$, we conclude that all measurable functions on $\X^{[k]}$ are equal $\mu^{[k]}$-a.e. to a function independent of the first coordinate.  In particular, we may assume without loss of generality that $F$ is independent of the first coordinate.

Let $\eps > 0$ be a small number.  By definition of the product $\sigma$-algebra, we can approximate $F$ up to an error which is $O(\eps)$ in $L^2(\X^{[k]})$ by a function $\tilde F$ bounded in magnitude by $1$ of the form
$ \tilde F( (x_\w)_{\w \in \2^k} ) = \sum_{j=1}^N \prod_{\w \in \2^k \backslash -\1} F_{j,\w}(x_\w)$
for some functions $f_{j,\w} \in L^\infty(\X)$ with norm $\|f_{j,\w}\|_{L^\infty(\X)} \leq 1$ and some finite $N$.  We conclude that $f(g,x_{-\1})$ differs by $O(\eps)$ in $L^2(\X^{[k]})$ from the function
$$ \sum_{j=1}^N \sum_{j'=1}^N \prod_{\w \in \2^k \backslash -\1} F_{j,\w}(T_g x_\w) \overline{F_{j',\w}(x_\w)} f_\w(g,x_\w).$$
On the other hand, $f$ has magnitude $1$.  We conclude (for $\eps$ small enough) that
$$ |\sum_{j=1}^N \sum_{j'=1}^N \int_{\X^{[k]}} f(g,x_{-\1}) \prod_{\w \in \2^k \backslash -\1} \overline{F_{j,\w}(T_g x_\w)} F_{j',\w}(x_\w) \overline{f_\w(g,x_\w)}\ d\mu^{[k]}| \geq \frac{1}{2}$$
and thus by the pigeonhole principle that
$$ |\int_{\X^{[k]}} f(g,x_{-\1}) \prod_{\w \in \2^k \backslash -\1} \overline{F_{j,\w}(T_g x_\w)} F_{j',\w}(x_\w) \overline{f_\w(g,x_\w)}\ d\mu^{[k]}| \geq \frac{1}{2N^2}$$
for some $j,j'$ (depending on $g$).  Applying the Cauchy-Schwarz-Gowers inequality \eqref{csg} we conclude that
$\|f(g,\cdot) \|_{U^k(\X)} \geq \frac{1}{2N^2}$
for all $g \in G$; the point here is that the lower bound is uniform in $g$.  Applying Lemma \ref{welldefined}, we conclude that
$$ \int_{\X^{[k]}} d^{[k]} f(g,\x)\ d\mu^{[k]} \geq (\frac{1}{2N^2})^{2^k}$$
for all $g \in G$.  Averaging this over a F{\o}lner set $\Phi_n$, we conclude in particular that
$$ \int_{\X^{[k]}} |\E_{g \in \Phi_n} d^{[k]} f(g,\x)|\ d\mu^{[k]} \geq (\frac{1}{2N^2})^{2^k}$$
for all $n$.  From the monotone convergence theorem we conclude that
$$ \int_{\X^{[k]}} \limsup_{n \to \infty} |\E_{g \in \Phi_n} d^{[k]} f(g,\x)|\ d\mu^{[k]} \geq (\frac{1}{2N^2})^{2^k}$$
and (ii) follows.

Now we show that (ii) implies (i).  We will use some arguments related to those used to prove Proposition \ref{ext-type-k}.  We first observe from Definition \ref{quasicocycle} and Lemma \ref{basic}(iii) that $d^{[k]} f$ is a $(G,\X^{[k]},S^1)$-cocycle.  This allows us to build the circle extension $\X^{[k]} \times_{d^{[k]} f} S^1$ of $\X^{[k]}$.  Applying the ergodic theorem to the vertical function $(\x,u) \mapsto u$ in this extension, we conclude that the limit $F(\x) := \lim_{n \to \infty} \E_{g \in \Phi_n} d^{[k]} f(g,\x)$ exists $\mu^{[k]}$-a.e., and is invariant under the diagonal action of $G$.  By (ii), $F$ is non-zero on a set $A$ of positive measure in $\X^{[k]}$, which we can take to be invariant under the action of $G$.  In particular, $A$ corresponds to some set $B \in \I_k(\X)$ with $P_k(B) > 0$, where $P_k$ is the restriction of $\mu^{[k]}$ to $\I_k(\X)$.

Since $d^{[k]} f$ is a $(G,\X^{[k]},S^1)$-cocycle, we have
$d^{[k]} f(g+g',\x) = (d^{[k]} f(g,\x)) d^{[k]} f(g',T_{g}^{[k]} \x)$.
Averaging $g'$ over the F{\o}lner sequence $\Phi_n$ and taking limits, we conclude that
$F(\x) = (d^{[k]} f(g,\x)) F( T_g^{[k]} \x )$.
This implies that $d^{[k]} f$ is a $(G,B,S^1)$-coboundary.  

Now let $\alpha$ be any face of $\2^k$.  From Definition \ref{quasicocycle} and Lemma \ref{basic}(iii) we see that $d^{[k-1]} f$ is also a $(G,\X^{[k-1]},S^1)$-cocycle, and thus $(\partial(\alpha)_*)^* d^{[k-1]} f$ is a $(G,\X^{[k]},S^1)$-cocycle.  This implies that
$$ \frac{(T_h)^{[k]}_\alpha d^{[k]} f( g, \x )}{d^{[k]} f(g, \x)} = \mder_{g^{[k]}} (\partial(\alpha)_*)^* d^{[k-1]}f(h, \x )$$
for every $h \in G$, and so $(T_h)^{[k]}_\alpha d^{[k]} f$ is $(G,\X^{[k]},S^1)$-cohomologous to $d^{[k]} f$.  In particular, $d^{[k]} f$ is a $(G,(T_h)^{[k]}_\alpha B,S^1)$-coboundary.  Using the same gluing argument used in the proof of Proposition \ref{ext-type-k} we conclude that $d^{[k]} f$ is a $(G, \X^{[k]}, S^1)$-coboundary, and the claim follows.
\end{proof}

We can now prove Proposition \ref{exact_descent} and Lemma \ref{fint-csg}.

\begin{proof}[Proof of Proposition \ref{exact_descent}]  By hypothesis and Lemma \ref{fint}, we have
$$\limsup_{n \to \infty} |\E_{g \in \Phi_n} d^{[k]} \pi^* f(g,\x)| \neq 0$$
for all $\x$ in a set of positive measure in $\X^{[k]}$.

By Lemma \ref{functor}, $\Y^{[k]}$ is a factor of $\X^{[k]}$.  Let $\pi^{[k]}$ be the factor map, then we have $d^{[k]} \pi^* f = (\pi^{[k]})^* d^{[k]} f$.  We conclude that
$\limsup_{n \to \infty} |\E_{g \in \Phi_n} d^{[k]} f(g,\y)| \neq 0$
for all $\y$ in a set of positive measure in $\Y^{[k]}$.

Since $\X$ is of order $<k$, $\Y$ is also.  The claim then follows from another application of Lemma \ref{fint}.
\end{proof}

\begin{proof}[Proof of Lemma \ref{fint-csg}]  By hypothesis, we have
$ \bigotimes f_\w(g,\x) = F( T_g^{[k]} \x ) \overline{F( \x )}$
for some $(X^{[k]},S^1)$-function $F$.  Arguing exactly as in the proof of Lemma \ref{fint}, we see that for all $\x$ in a set of positive $\mu^{[k]}$-measure, $\limsup_{n \to \infty} |\E_{g \in \Phi_n} d^{[k]} f_{-\1}(g,\x)| \neq 0$.  Applying Lemma \ref{fint}, we conclude that $f_{-\1}$ is of type $<k$ as desired.  The corresponding claims for the other $f_\w$ are established similarly.
\end{proof}

The proof of Theorem \ref{main-thm-high} is now complete.

\appendix

\section{General theory of Gowers-Host-Kra seminorms}\label{basic-appendix}

In this appendix we set out the general ``soft'' theory of the Gowers-Host-Kra seminorms.  The theory for $\Z$-systems is discussed in detail in \cite{hk-cubes}.  This theory carries over without any difficulties to other discrete abelian groups, such as $\Fw$, but for the convenience of the reader we reproduce the theory from \cite{hk-cubes} here. 

Throughout this appendix, $(G,+)$ is a fixed countable abelian group, the system $\X = (X,\B_X,\mu_X,(T_g)_{g \in G})$ is a fixed $G$-system, and $k \geq 1$ is a fixed integer. 

\subsection{The cubic construction}

The first step is to construct the cubic spaces and measures from \cite{hk-cubes}, generalised from $\Z$-actions to $G$-actions. We slightly modify the terminology and notation of  \cite{hk-cubes} to make it compatible with standard notation of the theory of cubic complexes (see Section \ref{cubic-sec}).  

\begin{definition}[The discrete cube and its faces]\label{faces}
Let  $\2^k$ denote  the set $\{-1,1\}^k$, which we equip with the lexicographic order.  One can identify $\2^k$ with the vertices of the 
standard cube ${\bf I}^k := \{(x_1,\ldots,x_k) \in \R^k: -1 \leq x_i \leq 1\}$. Under this identification we will refer to $\2^k$ as the standard {\em discrete}
($k$-dimensional) cube. We will denote the elements of $\2^k$ by $\w=(w_1,\ldots,w_k)$.
For each $j=1,\ldots, k$, the sets $\beta_j^+ := \{\w \in \2^k: w_j=1\}$,  $\beta_j^- :=\{\w \in \2^k: w_j=-1\}$ correspond to the sets of vertices on opposite sides of the cube ${\bf I}^k$. We will  refer to $\beta_j^+, \beta_j^-$ as \emph{opposite sides} or \emph{parallel sides} of $\2^k$, and to the $\beta_j^+$ as the {\em positive sides}. 

More generally, for any $0 \leq l \leq k$, define an \emph{$l$-dimensional face} or \emph{$l$-face} to be any set formed by intersecting $k-l$ distinct non-parallel sides. Thus $\2^k$ has one $k$-face, $2k$ faces of dimension $k-1$ (i.e. the sides $\beta_j^{\pm}$), $k2^{k-l}$ faces of dimension $(k-l)$, and so forth down to $2^k$ faces of dimension $0$ (which are the vertices of the discrete cube).  

Let $\alpha$ be an $l$-face. Enumerating the elements of $\alpha$ in lexicographic order gives a natural bijection $\partial(\alpha): \alpha \to \2^l$, which we call the \emph{coordinate map} of $\alpha$, which maps the faces of $\2^k$, which are subsets of $\alpha$, to the faces of $\2^l$. 
\end{definition}
  
\begin{definition}[Cubic complexes]\label{complex} Let $S$ be an arbitrary set.
We write $S^{[k]}:=S^{\2^k}$ for the set of functions $\s:\2^k \to S$.  For each $0 \leq l \leq k$ and each $l$-face $\alpha$, we let $\partial(\alpha)_*: S^{[k]} \to S^{[l]}$ be the pushforward map given by the formula
$ \partial(\alpha)_*(\s)(\w) := \s( \partial(\alpha)^{-1}(\w) )$
for all $\w \in \2^l$.  For any $1 \leq j \leq k$ and sign $\pm$, we abbreviate the \emph{cubic boundary map} $\partial( \beta_j^\pm )_*: S^{[k]} \to S^{[k-1]}$ as $\partial_j^\pm = \partial_{j,k}^\pm$.  
\end{definition}

\begin{example}  If $k=2$, then
\begin{align*}
\partial_1^- (s_{(-1,-1)},s_{(-1,1)},s_{(1,-1)},s_{(1,1)})&=(s_{(-1,-1)},s_{(-1,1)}) \\
\partial_1^+ (s_{(-1,-1)},s_{(-1,-1)},s_{(1,-1)},s_{(1,1)})&=(s_{(1,-1)},s_{(1,1)}) \\
\partial_2^- (s_{(-1,-1)},s_{(-1,1)},s_{(1,-1)},s_{(1,1)})&=(s_{(-1,-1)},s_{(1,-1)})\\
\partial_2^+ (s_{(-1,-1)},s_{(-1,1)},s_{(1,-1)},s_{(1,1)})&=(s_{(-1,1)},s_{(1,1)}).
\end{align*}
\end{example}

\begin{remark}\label{iter} For future reference we make the trivial observation that the map $\s \mapsto (\partial_{k+1}^- \s, \partial_{k+1}^+ \s)$ is a bijection between $S^{[k+1]}$ and $S^{[k]} \times S^{[k]}$. 
\end{remark}

\begin{definition}[Face groups]\label{facebook}  Let $G$ be a (possibly non-abelian) group with identity $\id_G$, and let $\alpha$ be a face of $\2^k$.  For every $g \in G$, we let $g^{[k]}_\alpha \in G^{[k]}$ denote the group element whose components $(g^{[k]}_\alpha)_\w$ for $\w \in \2^k$ are defined to equal $g$ when $\w \in \alpha$, and to equal $\id_G$ otherwise.  The map $g \mapsto g^{[k]}_\alpha$ is a bijection from $G$ to the \emph{face group} $G^{[k]}_\alpha := \{ g^{[k]}_\alpha: g \in G \} \leq G^{[k]}$.  When $\alpha$ is a side (resp. a positive side) we refer to $G^{[k]}_\alpha$ as a \emph{side group} (resp. a \emph{positive side group}); when $\alpha = \2^k$ is the entire cube we refer to $G^{[k]}_{\2^k}$ as the \emph{diagonal group} and denote it as $\diag(G^{[k]})$, and abbreviate $g^{[k]}_{\2^k}$ as $g^{[k]}$.  We also write $\partial^{[k]} G$ (resp. $\partial^{[k]}_+$) for the subgroup of $G^{[k]}$ generated by all the side groups (resp. all the positive side groups).
\end{definition}

\begin{example} For $k=2$, the group  $\partial^{[2]}G$ is generated by 
$$\bigcup_{g\in G} \{(\id_G,\id_G,g,g),(g,g,\id_G,\id_G),(\id_G,g,\id_G,g),(g,\id_G,g,\id_G)\}_{g \in G}$$
while the group $\partial_+^{[2]} G$ is generated by 
$\bigcup_{g \in G} \{(\id_G,\id_G,g,g),(\id_G,g,\id_G,g)\}_{g \in G}.$ 
\end{example}

\begin{remark}\label{side-gen} For future reference we observe that the side group $\partial^{[k]} G$ is the group generated by the positive side group $\partial^{[k]}_+ G$ and the diagonal group $\diag(G^{[k]})$.
\end{remark}

\begin{definition}[Face actions]\label{faceact}  Let $G$ be a group acting on a space $X$ by transformations $T_g: X \to X$ for $g \in G$.  Then $G^{[k]}$ acts on $X^{[k]}$ in the obvious manner, with the action $T^{[k]}_\g$ of a group element $\g = (g_\w)_{\w \in \2^k} \in G^{[k]}$ mapping each point $(x_\w)_{\w \in \2^k} \in X^{[k]}$ to $(T_{g_\w}(x_\w))_{\w \in \2^k}$.  If $\alpha$ is a face, we abbreviate the \emph{face transformation} $T^{[k]}_{g^{[k]}_\alpha}$ as $(T_g)^{[k]}_\alpha$, thus $((T_g)^{[k]}_\alpha)_{g \in G}$ is an action of $G$ on $X^{[k]}$.  If $\alpha$ is a side (resp. a positive side), we refer to $(T_g)^{[k]}_\alpha$ as a \emph{side transformation} (resp. \emph{positive side transformation}), and if $\alpha = \2^k$ is the entire cube, we refer to $(T_g)^{[k]}_{\2^k}$ as a \emph{diagonal transformation} and abbreviate it further as $(T_g)^{[k]}$.
\end{definition}

\begin{example} If $k=2$, then
$$(T_g)^{[2]}_{\beta_1^-} (x_{(-1,-1)},x_{(-1,1)},x_{(1,-1)},x_{(1,1)})=(T_g x_{(-1,-1)},T_g x_{(-1,1)}, x_{(1,-1)},x_{(1,1)})$$
and similarly 
$$(T_g)^{[2]}_{\beta_2^+}   (x_{(-1,-1)},x_{(-1,1)},x_{(1,-1)},x_{(1,1)})=(x_{(-1,-1)},T_g x_{(-1,1)},x_{(1,-1)},T_g x_{(1,1)}).$$
Finally,
$$(T_g)^{[2]}
 (x_{(-1,-1)},x_{(-1,1)},x_{(1,-1)},x_{(1,1)})=(T_g x_{(-1,-1)},T_g x_{(-1,1)}, T_g x_{(1,-1)}, T_g x_{(1,1)}).$$
\end{example}

\begin{remark}\label{preserve} For future reference we observe that all positive side maps preserve the $(-1,\ldots,-1)$ coordinate of $X^{[k]}$.
\end{remark}

We recall the notion of a relative product:

\begin{definition}
Let $G$ be a countable group, and let $\X_1, \X_2$ be two $G$-systems with a common factor $\Y$.  Let $\pi_1: X_1 \to Y$ and $\pi_2: X_2 \to Y$ be the factor maps.  For $i=1,2$, let 
$\mu_{\X_i,y}$ represent the disintegration of $\mu_{\X_i}$ with respect to 
$Y$, thus 
$$ \int_{X_i} (\pi_i^* f) F\ d\mu_{\X_i} = \int_{Y_i} f(y) (\int_{X_i} F\ d\mu_{\X_i,y})\ d\mu_Y(y)$$
for all $f \in L^\infty(\Y)$, $F \in L^\infty(\X_i)$.  It is well known that this disintegration exists and is unique up to almost everywhere equivalence.

Let $\mu_{\X_1}\times_Y \mu_{\X_2}$ denote the measure  defined by the formula
\begin{equation}\label{relative_product}
 \mu_{\X_1}\times^{}_Y \mu_{\X_2}(A):=\int \mu_{\X_1,y}\times^{}_Y \mu_{\X_2,y}(A) d\mu_Y
\end{equation}
for $A \in \B_{\X_1} \times \B_{\X_2}$; this is known as the \emph{relatively independent joining} of $\mu_{\X_1}$ and $\mu_{\X_2}$ over $\mu_Y$. 
We refer to the $G$-system 
\[ \X_1 \times_\Y \X_2 := (X_1 \times X_2, \B_{\X_1} \times  \B_{\X_2}, 
   \mu_{\X_1}\times^{}_{\Y} \mu_{\X_2}, \diag(G \times G) )\] 
where $\diag(G \times G)$ denotes the diagonal action of $G$, as the \emph{relative product} of $\X_1$ and $\X_2$ with respect to $\Y$.
\end{definition}

Now we can introduce the cubic measure spaces of Host and Kra.

\begin{definition}[Cubic measure spaces]\label{cubichk-def}\cite[Section 3]{hk-cubes}  Let $\X = (X,\B_X,\mu_X,(T_g)_{g \in G})$ be a $G$-system.  For each $k \geq 0$, we endow $X^{[k]}$ with the product $\sigma$-algebra $\B^{[k]} := (\B_X)^{\2^k}$, and define the cubic measures $\mu^{[k]}$ and $\sigma$-algebras $\I_k \subset \B^{[k]}$ inductively as follows:
\begin{itemize}
\item Set $\I_0$ to be the $\sigma$-algebra of invariant sets in $\X$, and $\mu^{[0]}$ to be $\mu_X$.
\item Once $\mu^{[k]}$ and $\I_k$ are defined, we identify $X^{[k+1]}$ with $X^{[k]} \times X^{[k]}$ as per Remark \ref{iter}, and define $\mu^{[k+1]} := \mu^{[k]} \times_{\I_k} \mu^{[k]}$ to be the relatively independent joining of $\mu^{[k]}$ with itself over $\I_k$.  We then let $\I_{k+1}= \I_{k+1}(\X)$ be the $\sigma$-algebra of invariant sets of the system $\X^{[k+1]}$. 
\end{itemize}
We then define the $G$-system $\X^{[k]}$ as
$ \X^{[k]} := (X^{[k]}, \B^{[k]}, \mu^{[k]}, (T_g^{[k]})_{g \in G})$, i.e $\B_{\X^{[k]}}=\B^{[k]}$ and  $\mu_{\X^{[k]}}=\mu^{[k]}$.
\end{definition}

\begin{remark}\label{iterm} From construction we see that
\begin{equation}\label{def_measures}
\mu^{[k+1]} = \int_{X^{[k]}} (\mu^{[k]})_s \times (\mu^{[k]})_s dP_k(s) 
\end{equation}
where
\begin{equation}\label{decomposition}
\mu^{[k]} = \int (\mu^{[k]})_s dP_k(s) 
\end{equation}
is the ergodic decomposition of $\mu^{[k]}$ with respect to the diagonal action $(T_g^{[k]})_{g \in G}$ of $G$.  
\end{remark}

The cubic measures have a useful symmetry property:

\begin{lemma}[Symmetry of cubic measures]\label{symcube}\cite[Proposition 3.7]{hk-cubes} The measure $\mu^{[k]}$ is invariant under all the symmetries of the cube $\2^k$ (which act in an obvious manner on $\X^{[k]}$).
\end{lemma}

The cubic measures also behave well with respect to passage to subcubes.  

\begin{lemma}[Cubic complex structure]\label{ccs}\cite[Corollary 3.8]{hk-cubes}  Let $0 \leq l \leq k$, and let $\alpha$ be an $l$-face of $\2^k$.  Then the map $\partial(\alpha)_*: \X^{[k]} \to \X^{[l]}$ is a factor map from $\X^{[k]}$ to $\X^{[l]}$.
\end{lemma}

\begin{example} Let $\X = (X,\B_X,\mu_X,(T_g)_{g \in G})$ be an ergodic $G$-system.  For $k=1$ we have 
$\mu^{[1]}=\mu_X \times \mu_X$, 
$\X^{[1]}=(X \times X,\B_X \times \B_X,\mu_X \times \mu_X,(T_g  \times T_g)_{g \in G})$, and $\I_1$ is the $\sigma$-algebra of measurable subsets of $X \times X$ that are invariant under the action of $T_g^{[1]}=T_g  \times T_g$, for all $g \in G$. 
Now if
$\mu_X \times \mu_X = \mu^{[1]}  = \int ( \mu^{[1]})_s dP_1(s)$
is the ergodic decomposition of $\mu_X \times \mu_X$ with respect to the action of $(T_g^{[1]})_{ g \in G}$, then
$
 \mu^{[2]} =  \int ( \mu^{[1]})_s \times  ( \mu^{[1]})_s dP_1(s)$.
\end{example} 

\begin{example} Let $G$ be a finite abelian group, and let $\X$ be $G$ with the translation action, normalized counting measure, and the discrete $\sigma$-algebra.  Then $\mu^{[k]}$ is the normalized counting measure on the space of cubes
$$ \{ ( x + h_1 w_1 + \ldots + h_k w_k )_{(w_1,\ldots,w_k) \in \2^k}: x,h_1,\ldots,h_k \in G \}$$
and $\I_k$ is the $\sigma$-algebra consisting of subsets of $G^{[k]}$ which are invariant under the diagonal translations $(x_\w)_{\w \in \2^k} \mapsto (x_\w + h)_{\w \in \2^k}$ for $h \in G$.  Thus the probability space $(X^{[k]}, \B^{[k]}, \mu^{[k]})$ is measure isomorphic to the space $G^{k+1} = \{ (x,h_1,\ldots,h_k): x,h_1,\ldots,h_k \in G \}$ with the discrete $\sigma$-algebra and normalized counting measure, whilst $(X^{[k]}, \I_k, \mu^{[k]})$ is measure isomorphic to the space
$G^{k} = \{ (h_1,\ldots,h_k): h_1,\ldots,h_k \in G \}$ with the discrete $\sigma$-algebra and normalized counting measure. 
\end{example}

\subsection{Existence of the seminorms}

The objective of this section is to establish that the Gowers-Host-Kra seminorms from Definition \ref{ghk-uni} are in fact well-defined, and to relate them to the cubic measures just constructed.

For $\w=(w_1,\ldots,w_k) \in \2^k$ denote $\sgn(\w):=\prod_{i=1}^k w_i \in \{-1,1\}$.   
For any functions $f_{\w}: X \to \C$, $\w \in \2^k$ we denote by $\bigotimes_{\w \in \2^k} f_{\w}: X^{[k]} \to \C$ the tensor product
$$\bigotimes_{\w \in \2^k} f_{\w}( (x_\w)_{\w \in \2^k} ) :=
\prod_{\w \in 2^{\w}} f_{\w}(x_{\w}).$$

\begin{lemma}\label{welldefined}  Let $f \in L^\infty(\X)$.
The limits in Definition \ref{ghk-uni} exist and do not depend on the choice of F\o lner sequences. Furthermore if $f_{\w}:=f$ when $\sgn(\w)=1$, and $f_{\w}:=\bar{f}$ when $\sgn(\w)= -1$ ($\bar f$ denotes the complex conjugate of $f$),
then 
\[
\|f\|^{2^k}_{U^k(\X)}= (\pi^{\X^{[k]}}_{\pt})_*  (\bigotimes_{\w \in \2^k} f_{\w} )= \int_{X^{[k]}}  \bigotimes_{\w \in \2^k} f_{\w}  \ d \mu^{[k]} 
\]
\end{lemma}

\begin{proof}
We prove this by induction on $k$. For $k=1$ this follows from the mean ergodic theorem. 
Assume the induction hypothesis holds for $k-1$.  Then 
\[\begin{aligned}
 \E_{h \in {\Phi^k_n} } \| \mder_h f \|_{U^{k-1}(\X)}^{2^{k-1}}& = \pi^{\X^{[k-1]}}_{pt *}  (\bigotimes_{\w \in \2^{k-1}} f_{\w}\cdot T_h \bar{f}_{\w} )\\
&  = \E_{h \in {\Phi^k_n} } \int  (\bigotimes_{\w \in \2^{k-1}} f_{\w}) \cdot    (\bigotimes_{\w \in \2^{k-1}}  T_h \bar{f}_{\w} ) \ d \mu^{[k-1]}  \\
&  = \E_{h \in {\Phi^k_n} } \int  (\bigotimes_{\w \in \2^{k-1}} f_{\w}) \cdot   T_h^{[k]} (\bigotimes_{\w \in \2^{k-1}}  \bar{f}_{\w}) \  d \mu^{[k-1]} 
\end{aligned}
 \]
Since $\mu^{[k-1]} $ is invariant with respect to the action of  $(T_g^{[k]})_{ g \in G}$,  by the ergodic theorem the above averages converge to 
\[\begin{aligned}
\int  \pi^{\X^{[k-1]}}_{\I_{k-1}*}(\bigotimes_{\w \in \2^{k-1}} f_{\w}) \cdot   \pi^{\X^{[k-1]}}_{\I_{k-1}*} (\bigotimes_{\w \in \2^{k-1}}  \bar{f}_{\w})  \  d \mu^{[k-1]}  
&=\int (\bigotimes_{\w \in \2^{k-1}} f_{\w})(\bigotimes_{\w \in \2^{k-1}} \bar{f}_{\w})  \  d \mu^{[k]}\\
&= (\bigotimes_{\w \in \2^{k}} f_{\w}) \  d \mu^{[k]} = \pi^{\X^{[k]}}_{pt *}  (\bigotimes_{\w \in \2^{k}} f_{\w}) 
\end{aligned}\]
thus closing the induction.
\end{proof}

\begin{remark} The above expression is used in \cite{hk-cubes} as the \emph{definition} of the $U^k(\X)$ norm.  
\end{remark}

We record some basic properties of the Gowers-Host-Kra seminorms:

\begin{lemma}[Basic properties of $U^k$]\label{uk-basic}\cite[Lemma 3.9]{hk-cubes} (See also \cite[Lemmas 3.8, 3.9]{gowers})   
\begin{itemize}
\item[(i)] For any $\2^k$-tuple $(f_\w)_{\w \in \2^k}$ of functions $f_\w \in L^\infty(\X)$, we have the \emph{Cauchy-Schwarz-Gowers inequality}
\begin{equation}\label{csg}
|\int_{X^{[k]}}  \bigotimes_{\w \in \2^k} f_{\w}  \ d \mu^{[k]} | \leq \prod_{\w \in \2^k} \|f_\w\|_{U^k(\X)}.
\end{equation}
\item[(ii)] The function $f \mapsto \|f\|_{U^k(\X)}$ is a seminorm on $L^\infty(\X)$.
\item[(iii)] We have the monotonicity property
\begin{equation}\label{mono}
\| f \|_{U^k(\X)} \leq \|f\|_{U^{k+1}(\X)}.
\end{equation}
\end{itemize}
\end{lemma}

\begin{corollary}\label{uk-norm-I_k-cor} Let $f \in L^\infty(\X)$.  Then $\|f\|_{U_k(\X)}=0$ if and only if 
$(\pi^{\X^{[k-1]}}_{\I_{k-1}})_* \prod_{\w \in \2^{k-1}} f_\w = 0$
$\mu^{[k-1]}$-a.e., where $f_\w$ is as in Lemma \ref{welldefined}.
\end{corollary}

\begin{proof}  Using \eqref{def_measures} and Lemma \ref{welldefined}, we can write
$$ \|f\|^{2^k}_{U_k(\X)} = \int_{X^{[k-1]}} |(\pi^{X^{[k]}}_{\I_{k-1}})_* \prod_{\w \in \2^{k-1}} f_\w|^2\ d\mu^{[k-1]}$$
and the claim follows.
\end{proof}

The above construction is functorial:

\begin{lemma}[Functoriality]\label{functor}\cite[Lemma 4.5]{hk-cubes}  If $\Y = (Y, \B_Y, \mu_Y, \pi^X_Y)$ is a factor of $\X$, then $\Y^{[k]} = (Y^{[k]}, \B_Y^{[k]}, \mu^{[k]}_Y, (\pi^X_Y)^{[k]})$ is a factor of $\X^{[k]}$, and for every $f \in L^\infty(\Y)$ one has $\| (\pi^X_Y)^* f\|_{U^k(\X)} = \|f\|_{U^k(\Y)}.$ 
\end{lemma}

If $\X$ is ergodic with respect to $G$, it is certainly not the case in general that $\X^{[k]}$ is ergodic with respect to the diagonal action $\diag(G^{[k]})$.  Nevertheless there are some other important ergodicity-preserving properties of the above construction:

\begin{lemma}[Ergodicity-preserving properties]\label{erglem}\cite{hk-cubes}  Let $\X$ be an ergodic $G$-system, and let $k \geq 1$.
\begin{itemize}
\item[(i)] $\X^{[k]}$ is ergodic with respect to the action of $\partial^{[k]} G$.
\item[(ii)] The measure $P_k$ is ergodic with respect to the action of $\partial^{[k]}_+ G$.
\item[(iii)] For any measure-preserving transformation $u: X \to X$ that commutes with the $G$-action, and any side $\alpha$, the side transformation $u^{[k]}_\alpha$ preserves $\mu^{[k]}$.
\item[(iv)] More generally, if $u: X \to X$ is a measure-preserving transformation that commutes with the $G$-action and leaves $\Zcal_{<l}(\X)$ invariant for some $1 \leq l \leq k$, and $\alpha$ is a $k-l$-dimensional face of $\2^k$, then $u^{[k]}_\alpha$ preserves $\mu^{[k]}$.
\end{itemize}
\end{lemma}

\begin{proof} For (i), see \cite[Corollary 3.5]{hk-cubes}; for (ii), see \cite[Corollary 3.6]{hk-cubes}.  For (iii) and (iv), see \cite[Lemma 5.5]{hk-cubes}.  
\end{proof}

\subsection{Dual functions}

A ``soft'' way to describe the universal characteristic factors $\Zcal_{<k}(\X)$ is via the convenient device of \emph{dual functions}.

\begin{definition}[Dual functions]\label{dual_functions}
Let $\X = (X,\B_X,\mu_X,(T_g)_{g \in G})$ be a $G$-system, and let  $\{\Phi^i_n\}_{n=1}^{\infty}$,  $i=1,\ldots, k$ be  $k$  F\o lner sequences  in $G$.  
Define the nonlinear operators $\D_k:L^{\infty }(\X) \to L^{\infty }(\X)$ inductively by setting $\D_1 f  := 1$ and $\D_k f :=\lim_{n \to  \infty }\E_{h \in {\Phi^k_n} } T_h f \cdot \D_{k-1}( \overline{\mder_h f} ) $ for $k>1$, where the limit is in $L^2(\X)$.   
\end{definition}

\begin{remark}
The limits above exist, as in Lemma \ref{welldefined}, by repeated applications of the ergodic theorem.  Indeed, we easily verify that
\begin{equation}\label{dk}
\D_k f = (\pi^{\X^{[k]}}_{\X})_*  (\bigotimes_{\w \in \2^k \backslash \{-\1\}} \overline{f_{\w}} )
\end{equation}
where $-\1 := (-1,\ldots,-1)$ and the factor map is given by $(x_\w)_{\w \in \2^k} \mapsto x_{-\1}$ (i.e. the pushforward map $\partial(\{-\1\})_*$).  As a consequence we have
\begin{equation}\label{fdk}
\| f \|_{U^k(\X)}^{2^k}= \int_\X f \overline{\D_k f}\ d\mu_X
\end{equation}
and similarly (by repeated applications of the Cauchy-Schwarz inequality) that
\begin{equation}\label{fdkf}
\int_X f_1 \overline{\D_k f_2}\ d\mu_X \leq \|f_1\|_{U_k(\X)} \cdot \| f_2\|_{U_k(\X)}^{2^{k-1}}
\end{equation}
\end{remark}

\begin{example} When $k=2$, we have
$$
\D_2 f:= \lim_{m \to  \infty }\lim_{n \to  \infty } \E_{h_2 \in {\Phi^2_m} } \E_{h_1 \in {\Phi^1_n} }
(T_{h_2} f) (T_{h_1} f) \overline{T_{h_1 + h_2} f}.
$$
The limit above is a repeated limit, but \emph{a posteriori}, using Theorem \ref{main-thm}, one can show that the double (simultaneous) limit exists as well, and both limits coincide, by modifying the proof of \cite[Theorem 1.2]{hk-cubes}, and similarly for higher values of $k$.  We omit the details.
\end{example}

\begin{example} If $G$ is a finite abelian group, and $\X$ is $G$ with the translation action, then
for any $f: G \to \C$, the dual function $\D_k f: G \to \C$ is given by the formula
$$ \D_k f(x) = \E_{h_1,\ldots,h_k \in G} \prod_{(w_1,\ldots,w_k) \in \{0,1\}^k \backslash \{0\}^k} {\mathcal C}^{w_1+\ldots+w_k-1} f(x+w_1 h_1 +\ldots + w_k h_k)$$
where ${\mathcal C}: z \mapsto \overline{z}$ is the complex conjugation operator.  Dual functions in this setting play an important role in the finitary theory of arithmetic progressions and similar patterns; see \cite{gt-primes}.
\end{example}

\begin{remark} From Lemma \ref{functor} we have the functoriality property
$\D_k((\pi^X_Y)^* f) = (\pi^X_Y)^* \D_k f$
whenever $\Y$ is a factor of $\X$ and $f \in L^\infty(\Y)$.
\end{remark}

\begin{definition}[Universal characteristic factor]\label{ucf-soft-def}\cite[Definition 4.1]{hk-cubes}  We let $\Zcal_{<k} = \Zcal_{<k}(\X)$ be the sub-$\sigma$-algebra of $\B_X$ consisting of all sets $B \in \B_X$ such that $(\pi^{\X^{[k]}}_\X)^{-1}(B)$ is $\mu^{[k]}$-a.e. equivalent to a set $A$ in $\X^{[k]}$ which does not depend on the first coordinate $x_{-\1}$ (or equivalently, the set is invariant under the face transformation $g^{[k]}_{\{-\1\}}$ for all $g \in G$).
\end{definition}

\begin{example} Let $\X$ be an ergodic $G$-system. Then $\Zcal_{<1}(\X)$ is trivial.  One can use classical arguments to show that $\Zcal_{<2}(\X)$ is the Kronecker factor (i.e. the factor generated by the eigenfunctions of $\X$); see the discussion just before \cite[Lemma 4.2]{hk-cubes}.
\end{example}

\begin{remark}\label{Dualin}
It is easy to see that $\Zcal_{<k}(\X)$ is an invariant sub-$\sigma$-algebra and therefore a factor.  As the function on the right-hand side of \eqref{dk} does not depend on the first coordinate, we see that the dual function $\D_k f$ lies in $\Zcal_{<k}(\X)$ for all $f \in L^\infty(\X)$.  
\end{remark}

\begin{lemma}[$\Zcal_{<k}$ is universal]\label{softfactor}\cite[Lemma 4.3]{hk-cubes}  For any $f \in L^\infty(\X)$, we have $\|f\|_{U^k(\X)} = 0$ if and only if $(\pi^X_{Z_{<k}(X)})_* f = 0$.  
\end{lemma}

\begin{remark} From this lemma and Remark \ref{Dualin} it is not hard to show that $\Zcal_{<k}(\X)$ is in fact generated by the dual functions $\D_k f$ for $f \in L^\infty(X)$, although we will not use this fact here.
\end{remark}

Note that Lemma \ref{softfactor} immediately implies Proposition \ref{ucf-prop} in the introduction.  From this lemma and \eqref{mono} we also have 
\begin{equation}\label{mono-2}
\Zcal_{<j}(\Zcal_{<k}(\X))=\Zcal_{<j}(\X)
\end{equation}
for $0 < j \leq k$ (cf. \cite[Corollary 4.4]{hk-cubes}). 

From Lemma \ref{softfactor} and Lemma \ref{functor} one can show that universal characteristic factors are functorial:

\begin{lemma}[Functoriality]\label{functor2}\cite[Proposition 4.6]{hk-cubes}  If $\Y = (Y, \B_Y, \mu_Y, \pi^X_Y)$ is a factor of $\X$, then $\Zcal_{<k}(\Y)$ is a factor of $\Zcal_{<k}(\X)$ for any $k \geq 1$.  In fact, for any $f \in L^\infty(\Y)$, $(\pi^X_Y)^* f$ is $\Zcal_{<k}(\X)$-measurable if and only if $f$ is $\Zcal_{<k}(\Y)$-measurable.
\end{lemma}

If $\phi \in L^\infty(\X)$ is a phase polynomial of degree $<k$, an easy induction on $k$ using Definition \ref{phase-def} and Definition \ref{dual_functions} shows that $\phi = \D_k \phi$.  As a consequence of this and Remark \ref{Dualin} we obtain the easy direction of Theorem \ref{main-thm}, valid for any discrete abelian group $G$:

\begin{lemma}\label{easy-thm}  Every phase polynomial of degree at most $<k$ is $\Zcal_{<k}$-measurable, or in other words $\Phase_{<k}(\X) \subset L^2(\Zcal_{<k}(\X))$ (or $\Abr_{<k}(\X) \leq \Zcal_{<k}(\X)$).
\end{lemma}

Suppose that $f_\w$ is a function in $L^\infty(\X)$ for each $\w \in \2^k$.  From Definition \ref{ucf-soft-def} and Lemma \ref{symcube} we see that the quantity $\int_{X^{[k]}}  \bigotimes_{\w \in \2^k} f_{\w}  \ d \mu^{[k]}$ does not change if we replace one of the $f_\w$ by $\E(f_\w|\Zcal_{<k})$.  In particular we have
$$\int_{X^{[k]}}  \bigotimes_{\w \in \2^k} f_{\w}  \ d \mu^{[k]} = \int_{X^{[k]}}  \bigotimes_{\w \in \2^k} \E(f_{\w}|\Zcal_{<k})  \ d \mu^{[k]}.$$
Thus for instance $\D_k f = \D_k \E(f|\Zcal_{<k})$ for all $f \in L^\infty(\X)$ (and indeed $\Zcal_{<k}$ can be characterized as the minimal factor with this property).  Another corollary of the above formula is

\begin{lemma}\label{conditional-product}\cite[Proposition 4.7(1)]{hk-cubes}
The measure $\mu^{[k]}$ is a conditional relative product over $\Zcal_{<k}^{[k]}(\X)$.  In other words, one has a representation of the form
$ \mu^{[k]} = \int_{\Zcal_{<k}^{[k]}(\X)} (\bigotimes_{\w \in \2^k} \nu_{z,\w})\ d\sigma(z)$,
where $\sigma$ is the restriction of $\mu^{[k]}$ to $\Zcal_{<k}^{[k]}$, and for each $z \in \Zcal_{<k}^{[k]}(\X)$ and $\w \in \2^k$, $\nu_{z,\w}$ is a probability measure on $X$ which depends measurably on $z$.
\end{lemma}

\section{Abelian cohomology}\label{abcom}

In this section we collect some basic facts about abelian cohomology (as defined in Section \ref{abcom0}) that we will need in the paper. Much of this machinery is essentially from \cite{hk-cubes}, but for the convenience of the reader (and given that we are generalizing from $\Z$-actions to more general countable abelian  actions) we present the details here.  

The reader may wish to review the definitiosn in Definition \ref{cohom} before proceeding with the rest of this section.

We begin with the following trivial but useful lemma:

\begin{lemma}[Cocycles and pullbacks]\label{cocycle}  Let $G$ be a locally compact group, let $\X$ be a $G$-system, and let $\Y$ be a factor of $\X$ with factor map $\pi$.  Let $f$ be a $(G,\Y,U)$-function for some abelian group $U$.  Then $f$ is a $(G,\Y,U)$-cocycle if and only if $\pi^* f$ is a $(G,\X,U)$-cocycle.
\end{lemma}

Next, we recall that cohomology is trivial for free actions:

\begin{definition}[Free action]\label{free-def} Let $\X = (X, \B_X, \mu_X, (T_g)_{g \in G})$ be a $G$-system. The action of $G$ is said to be \emph{free} if $\X$ is measure-equivalent to a system of the form $Y \times G$, where the action of a group element $g \in G$ is given by the map $(y,h) \mapsto (y,gh)$.
\end{definition}

\begin{remark} If $G$ acts freely on $X$, then so does any compact abelian subgroup of $G$.
\end{remark}

\begin{lemma}[Free actions of compact abelian groups have no cohomology]\label{free-lem}\cite[Lemma C.8]{hk-cubes}  Let $G$ be a compact abelian group, and let $\X$ be a $G$-system in which the action of $G$ is free.  Then every $(G,\X,S^1)$-cocycle is a $(G,\X,S^1)$-coboundary.  In other words,
$Z^1(G,\X,S^1) = B^1(G,\X,S^1)$, or equivalently
$H^1(G,\X,S^1) = 0.$
\end{lemma}

There is an analogue of Lemma \ref{free-lem} in the polynomial category.  To state it, we first need a useful algebraic lemma.

\begin{lemma}[Composition of polynomials is again polynomial]\label{polycomp}
Let $G$ be a countable abelian  group, let $U, V$ be abelian groups, and let $\X = \Y \times_\rho U$ be an ergodic abelian extension of a $G$-system $\Y$ by a $(\Y,U)$-phase polynomial $\rho$ of degree $<k$ for some $k \geq 1$.   
\begin{itemize}
\item[(i)] If $p$ is a $(\X,V)$-phase polynomial of degree $<d$, and $u \in U$, then $\mder_u p$ is a $(\X,V)$-phase polynomial of degree $<d-1$.  
\item[(ii)] If $p$ is a $(\X,V)$-phase polynomial of degree $<d$, and $v_1, \ldots, v_j$ are a collection of $(\X,U)$-phase polynomials of degrees $<d_1,\ldots,<d_j$, then the $(\X,V)$-function $P(y,u) := (\mder_{v_1(y,u)} \ldots \mder_{v_j(y,u)} p)(y,u)$ is a $(\X,V)$-phase polynomial of degree $O_{d,j,d_1,\ldots,d_j}(1)$.  
\item[(iii)] If $p$ is a $(\X,V)$-phase polynomial of degree $<d$, $v_1,\ldots,v_j$ are 
a collection of $(\X,U)$-phase polynomials of degrees $<d_1,\ldots,<d_j$, and $s$ is a $(\X,U)$-phase polynomial of degree $<d'$, then the $(\X,V)$-function $$P(y,u) := (\mder_{v_1(y,u)} \ldots \mder_{v_j(y,u)} p)(y,s(y,u))$$ is a $(\X,V)$-phase polynomial of degree $O_{d,j,d_1,\ldots,d_j,d',k}(1)$.  
\item[(iv)] For each $u \in U$, let $q_u$ be a $(\X,V)$-phase polynomial of degree $<m$ which obeys the $U$-cocycle equation
\begin{equation}\label{coco}
q_{uv} = (V_u q_v) q_u
\end{equation}
for all $u, v \in U$.  Suppose that $v_1,\ldots,v_j$ are a collection of $(\X,U)$-phase polynomials of degrees $<d_1,\ldots,<d_j$, and $r, s$ are $(\X,U)$-phase polynomials of degree $<d'$, $<d''$ respectively, then the map $$P(y,u) := (\mder_{v_1(y,u)} \ldots \mder_{v_j(y,u)} q_{r(y,u)})(y, s(y,u))$$x is a $(\X,V)$-polynomial of degree $O_{d,j,d',d'',k}(1)$.  
\end{itemize}
\end{lemma}

\begin{proof}  We prove (i) by induction on $d$.  Indeed, the claim is trivial for $d=1$ by ergodicity, and for $d>1$ we have by induction that $\mder_g \mder_u p = \mder_u \mder_g p$ is a phase polynomial of degree $<d-2$ for all $g \in G$, and thus by \eqref{integ} $\mder_u p$ is a phase polynomial of degree $<d-1$ as claimed.  

We prove (ii) by a triple induction. First, the claim is trivial when $d=1$, so suppose that $d>1$ and the claim has already been shown for $d-1$.  For $j \geq d$ the claim follows from the previous claim, so now assume that $j < d$ and the claim has already been proven for $j+1$.  Finally, the claim is clear when $d_1+\ldots+d_j=0$, so suppose inductively that $d_1+\ldots+d_j>0$ and the claim has already been proven for smaller values of $d_1+\ldots+d_j$.  Consider the derivative $\mder_g P$ for some $g \in G$.  Some computation shows that this expression can be written as $(\mder_{v_1(y,u)} \ldots \mder_{v_j(y,u)} \mder_g p)(y,u)$, times a product of finitely many expressions of the form 
$(\mder_{v'_1(y,u)} \ldots \mder_{v'_{j'}(y,u)} p)(y,u)$ where $j'$ is either larger than $j$, or $j=j'$ and the total degree of $v'_1,\ldots,v'_j$ is less than that of $v_1,\ldots,v_j$.  Using the various induction hypotheses we conclude that $\mder_g P$ is a $(\X,S^1)$-phase polynomial of degree $O_{d,d_1,\ldots,d_j}(1)$, and the claim follows from \eqref{integ}.

Claim (iii) is proven by the same inductive argument as the previous claim; the non-linear nature of $s(y,u)$ introduced some new terms when one differentiates, but all such terms increase the number $j$ of vertical derivatives (and only involve polynomials in the subscripts, thanks to the polynomial nature of $s$ and $\rho$) and so can be safely handled by the induction hypothesis.

Finally, we prove claim (iv).  The case $d'=0$ follows from the previous claim, so suppose that $d'>0$ and the claim has already been proven for the smaller values of $d'$.  For $g \in G$, we take a derivative $\mder_g P$.  One obtains essentially the same terms that appeared in the previous claim, plus (thanks to the cocycle equation \eqref{coco}) some additional terms involving $q_{\mdersmall_g r(y,u)}$.  But such terms can be dealt with by the induction hypothesis.
\end{proof}

\begin{lemma}[Polynomial integration lemma]\label{poly-integ}  Let $G$ be a countable abelian group, let $m, k \geq 1$, and let $\X = \Y \times_\rho U$ be an ergodic abelian extension of a $G$-system $\Y$ by a $(G,\Y,U)$-phase polynomial cocycle $\rho$ of degree $<k$.  Let $V$ be a locally compact abelian group. For each $u \in U$, let $q_u$ be a $(\X,V)$-phase polynomial of degree $<m$ which obeys the $U$-cocycle equation \eqref{coco}
for all $u, v \in U$.  Then there exists a $(\X,V)$-phase polynomial $Q$ of degree $<O_{m,k}(1)$ such that $q_u = \mder_u Q$ for all $u \in U$.
\end{lemma}

\begin{proof}
Let $u_0$ be a generic element of $U$, and define the $(\X,V)$-function $Q$ by the formula
$ Q(y,v u_0) := q_v(y, u_0).$
for all $y \in \Y$ and $v \in U$.  Observe that for any $u \in U$, we have
$$ \mder_u Q(y,v u_0) = \frac{q_{uv}(y,u_0)}{q_v(y,u_0)} = q_u(y,v u_0)$$
thanks to \eqref{coco}.  Thus we have $q_u = \mder_u Q$ for all $u \in U$.  

The fact that $Q$ is a $(\X,V)$-phase polynomial of degree $O_{m,k}(1)$ follows from the $j=0$ case of Lemma \ref{polycomp}(iv).
\end{proof}

\begin{remark}  One can improve the degree bounds in Lemma \ref{polycomp} and Lemma \ref{poly-integ} if one assumes that $\Y=\Zcal_{<j}(\X)$ for some $j$; see Lemma \ref{mod} and Proposition \ref{exact_int}.
\end{remark}

We will also need another result in a similar spirit.

\begin{lemma}[Straightening nearly translation-invariant cocycles]\label{straighten-lemma}\cite[Lemma C.9]{hk-cubes} Let $G$ be a countable abelian group, let $\X$ be an ergodic $G$-system, let $K = (K,\cdot)$ be a compact abelian group acting freely on $X$ and commuting with the $G$ action, and let $\rho$ be an $(G,\X,H)$-function for some compact abelian $H$ such that $\mder_k \rho: (g,x) \mapsto \frac{\rho(g,T_k x)}{\rho(g,x)}$ is a coboundary for all $k \in K$.  Then $\rho$ is cohomologous to a $(G,\X,H)$-function which is invariant with respect to some open subgroup $U$ of $K$.
\end{lemma}

We will also take advantage of a useful splitting lemma.

\begin{lemma}[Splitting lemma]\label{split}\cite[Lemma C.5]{hk-cubes} (see also \cite{moore-schmidt})
Let $G$ be a countable abelian  group, let $\X$ be an ergodic $G$-system.
Let $f :\X \to S^1$ be a $(G,\X,S^1)$-cocycle such that $d^{[1]}f$ is a $(G,\X^{[1]},S^1)$-coboundary.
Then  $f$ is $(G,\X,S^1)$-cohomologous to a constant cocycle (i.e. a cocycle that is independent of the 
$\X$ coordinate). In other words, we have an exact sequence
\[
\begin{CD}
H^1(G,\pt,S^1)   @>>>   H^1(G,\X,S^1)  @>>> H^1(G,\X^{[1]},S^1)
\end{CD}\]
where the first map is the map induced by the factor map $\pi^\X_{\pt}$, and the second map is the map induced by the derivative map $f \mapsto d^{[1]} f$.
\end{lemma}

\begin{lemma}[Cohomology of $X^{[l]}$ injects into cohomology of $X^{[k]}$]\label{hkc-lem}\cite[Lemma C.7]{hk-cubes} Let $G$ be a countable abelian group, let $\X$ be an ergodic $G$-system, let $U$ be a compact abelian group, and let $k \geq l \geq 0$.  Let $\alpha$ be an $l$-face of $\2^k$, thus by Lemma \ref{ccs} $\partial(\alpha)_*: \X^{[k]} \to \X^{[l]}$ is a factor map.  Let $f$ be a $(G,\X^{[l]},U)$-cocycle such that the $(G,\X^{[k]},U)$-cocycle $(\partial(\alpha)_*)^* f$ is a $(G,\X^{[k]},U)$-coboundary.  Then $f$ is also a $(G,\X^{[l]},U)$-coboundary.  In other words, we have an exact sequence
$$
\begin{CD}
0  @>>> H^1(G,\X^{[l]},U) @>>> H^1(G,\X^{[k]},U).
\end{CD}
$$
\end{lemma}

Now we see how cohomology on an abelian extension relates to cohomology on the base space.

\begin{lemma}[Descent lemma]\label{descent-lem} Let $k \geq 1$, let $G$ be a countable abelian  group, and let $\X = (X, \B_X, \mu_X, (T_g)_{g \in G})$ be an ergodic $G$-system.  Let $\rho$ be an abelian $(G,\X,U)$-cocycle for some compact abelian $U$.  Let $\phi: U \to V$ be a surjective measurable homomorphism from $U$ to another compact abelian group $V$, and let $\pi^{\X \times_\rho U}_{\X \times_{\phi \circ \rho} V}$ be the associated factor map $(x,u) \mapsto (x,\phi(u))$. Suppose that $f$ is a $(G,\X \times_{\phi \circ \rho} V,S^1)$-function is such that $(\pi^{\X \times_\rho U}_{\X \times_{\phi \circ \rho} V})^* f$ is $(G,\X \times_\rho U,S^1)$-cohomologous to $(\pi^{\X \times_\rho U}_{\X \times_{\phi \circ \rho} V})^* p$ for some $(G,\X \times_{\phi \circ \rho} V,S^1)$-function $p$.  Then $f$ is $(G,\X \times_{\phi \circ \rho} V,S^1)$-cohomologous to $p (\chi \circ \rho \circ \pi^{\X \times_{\phi \circ \rho} V}_\X)$ for some character $\chi \in \hat U$.
\end{lemma}

\begin{proof} Let us first consider the case when $V$ (and $\phi$) is trivial, so that $f$ is now a $(G,\X,S^1)$-function.
By hypothesis, there exists a $(\X \times_\rho U,S^1)$-function $F$ such that
$$ f( g, x ) = p( g, x ) \frac{F( T_g x, \rho(g,x) u )}{F(x,u)}$$
for all $g \in G$ and almost every $x \in X$, $u \in U$.  We rearrange this as
\begin{equation}\label{ftg}
F(T_g x, \rho(g,x) u) = F(x,u) \overline{p}(g,x) f(g,x).
\end{equation}
We perform a Fourier expansion in $U$, obtaining
$$ F(x,u) = \sum_{\chi \in \hat U} F_\chi(x) \chi(u)$$
for some $F_\chi \in L^\infty(\X)$, not all identically zero. Comparing Fourier coefficients in \eqref{ftg}, we conclude that
$F_\chi(T_g x) \chi \circ \rho(g,x) = F_\chi(x) \overline{p}(g,x) f(g,x)$
for all $g \in G$ and $\chi \in \hat U$, and $\mu_X$-almost every $x \in X$.  In particular this shows that the function $|F_\chi|$ is $G$-invariant and thus (by ergodicity) constant.  Thus there exists $\chi \in \hat U$ such that $|F_\chi|$ is almost everywhere non-vanishing.  We can then write
$$ f(g,x) = p(g,x) \chi \circ \rho(g,x) \frac{F_\chi(T_g x)}{F_\chi(x)}.$$
This completes the proof in the case when $V$ is trivial.

To handle the general case we perform a lifting trick.  Observe that the system $\X' := \X \times_{\rho \oplus \phi \circ \rho} U \times V$ is a simultaneous extension of both $\X \times_\rho U$ and $\X \times_{\phi \circ \rho} V$.  Since $(\pi^{\X \times_\rho U}_{\X \times_{\phi \circ \rho} V})^* f$ is $(G,\X \times_\rho U,S^1)$-cohomologous to $(\pi^{\X \times_\rho U}_{\X \times_{\phi \circ \rho} V})^* p$, we see that $(\pi^{\X'}_{\X \times_{\phi \circ \rho} V})^* f$ is $(G,\X',S^1)$-cohomologous to $(\pi^{\X'}_{\X \times_{\phi \circ \rho} V})^* p$.  Writing $\X'$ as an extension of $\X \times_{\phi \circ \rho} V$ by the cocycle $\rho \circ \pi^{\X \times_{\phi \circ \rho}}_\X$, and applying the previous result, we obtain the claim. 
\end{proof}

\section{A measurable selection lemma}\label{measure-sec}

Suppose that $G$ is a countable abelian group  and  $\X$ is an ergodic $G$-system.
In our arguments we will frequently have a family of $(G,\X,S^1)$-functions $h_u$, parameterised in some measurable fashion by a parameter $u$ in a compact abelian group $U$, such that $h_u$ takes values in $\Phase_k(G,\X,S^1) \cdot B^1(G,\X,S^1)$. In other words (by the axiom of choice),  for each $u$, we may find $\psi_u \in \Phase_k(G,\X,S^1)$ and $F_u \in M(\X,S^1)$ such that for each $u \in U$, we have the equation
\begin{equation}\label{eq:measurable}
h_u =\psi_u \mder f_u.
\end{equation}

Unfortunately, due to the use of the axiom of choice, it is not necessarily the case that the functions $\psi_u$ and $F_u$ that arise here are measurable.  Fortunately, one can resolve this problem by using some separation properties of phase polynomials and the hypothesis that $\X$ is separable.  The basic tool here is the following (cf. \cite[Lemma 7.1]{gt-ff-ratner}):

\begin{lemma}[Separation lemma]\label{sep-lem} Let $G$ be a countable abelian group, let $\X=(X,\B_X,\mu_X,(T_g)_{g \in G})$ be an ergodic $G$-system, let $k \geq 1$, and let $\phi, \psi \in \Phase_{<k}(\X)$ be such that $\phi/\psi$ is non-constant.  Then $\| \phi-\psi\|_{L^2(\X)} \geq \sqrt{2} / 2^{k-2}$.
\end{lemma}

\begin{remark} The constant $\sqrt{2}/2^{k-2}$ can be improved slightly, but for our purposes any quantity that is independent of $\X$ would suffice here.
\end{remark}

\begin{proof} By dividing $\phi,\psi$ by $\psi$ we may assume $\psi=1$.

The claim is vacuous when $k=1$.  When $k=2$ we argue as follows.  For any $h \in G$ we have
\begin{equation}\label{phishift}
\int_X \phi\ d\mu_X = \int_X T_h \phi\ d\mu = \int_X (\mder_h \phi) \phi\ d\mu_X.
\end{equation}
If $\phi$ is a phase polynomial of degree $<2$, then $\mder_h \phi$ is constant; if $\phi$ is non-constant, then (by ergodicity) $\mder_h \phi$ is not identically $1$ for at least one $h$.  Thus $\int_X \phi\ d\mu_X = 0$ and so $\|\phi - 1 \|_{L^2(X)} = \sqrt{2}$, and the claim follows.

Now suppose inductively that $k \geq 3$ and the claim has already been proven smaller values of $k$.  Suppose for contradiction that there was a non-constant $\phi \in \Phase_{<k}(\X)$ such that $\|\phi-1\|_{L^2(\X)} < \sqrt{2}/2^{k-2}$.  Arguing as in \eqref{phishift} we conclude that
$\|(\mder_h \phi) \phi-1\|_{L^2(\X)} < \sqrt{2}/2^{k-2}$ for all $h$, and thus by the triangle inequality
$$ \| \mder_h \phi - 1\|_{L^2(\X)} = \| (\mder_h \phi)\phi - \phi\|_{L^2(\X)} < \sqrt{2}/2^{k-3}.$$
But $\mder_h \phi \in \Phase_{<k-1}(\X)$.  By the induction hypothesis we conclude that $\mder_h \phi$ is constant for every $h$, or in other words that $\phi \in \Phase_{<2}(\X)$.  The claim then again follows from the induction hypothesis.
\end{proof}

Since $L^2(\X)$ is separable, we conclude

\begin{corollary}[At most countably many polynomials modulo constants]\label{const}  Let $G$ be a countable abelian group, let $\X=(X,\B_X,\mu_X,(T_g)_{g \in G})$ be an ergodic $G$-system, let $k \geq 1$.  Then the collection $\Phase_{<k}(\X)$, after quotienting out by constants, is at most countable.
\end{corollary}

We are now ready to establish the measurable selection lemma.  

\begin{lemma}[Measurable selection lemma]\label{measurable_choice}  Let $G$ be a countable abelian group.
Let $\X$ be an ergodic $G$-system, and let $k \geq 1$.  Let $U$ be a compact abelian group. If $u \to h_u$ is a Borel measurable map from $U$ to $\Phase_{<k}(G,\X,S^1) \cdot B^1(G,\X,S^1) \subset M(G,\X,S^1)$ (where we give the latter the topology of convergence in measure), then there is a Borel measurable choice of $f_u, \psi_u$ (as functions from $U$ to $M(\X,S^1)$ and $\Phase_{<k}(G,\X,S^1)$ respectively) obeying \eqref{eq:measurable}. 
\end{lemma}

\begin{proof}  For each $u \in U$, write
$$\Omega_{u} := \{ f \in M(\X,S^1): h_u / \mder f \in \Phase_{<k}(G,\X,S^1) \},$$
then $\Omega_{u}$ is non-empty for each $u$ by hypothesis.  Also, if $f_u, f'_u \in \Omega_{u}$, then $\mder (f_u/f'_u) \in \Phase_{<k}(G,\X,S^1)$, and so (by \eqref{integ}) $f_u/f'_u \in \Phase_{<k+1}(\X)$; reversing this argument, we conclude  that $\Omega_{u}$ is a coset of $\Phase_{<k+1}(\X)$ in $M(\X,S^1)$ for each $u$.

As $L^2(\X)$ is separable, one can find a countable sequence $F_1, F_2, \ldots \in M(\X,S^1)$ which is dense in $M(\X,S^1)$.  For each $u$, let $n_u$ be the first integer such that there exists $k\geq 1$ and $f_u \in \Omega_{u}$ with $\|f_u - F_{n_u}\|_{L^2(\X)} < \sqrt{2} / 2^{k+1}$; this integer exists by density, and clearly depends in a measurable fashion on $u$.  By Lemma \ref{sep-lem} and the triangle inequality, we see that the $f_u \in \Omega_u$ which lie within $\sqrt{2} / 2^{k+1}$ of $F_{n_u}$ are all constant multiples of each other.  There is thus a unique $f_u \in \Omega_u$ which minimizes the distance $\|f_u - F_{n_u}\|_{L^2(\X)}$.  Selecting this $f_u$ (and then solving for $\psi_u$ using \eqref{eq:measurable}) we obtain the claim.
\end{proof}

\begin{remark} One can also establish this result using a general measure selection result of Dixmier (see e.g. \cite[Theorem 1.2.4]{kechris}) together with Lusin's theorem and Corollary \ref{const}; we omit the details.  One can also appeal to the descriptive set theory of Polish groups, see e.g. \cite[Appendix A]{hk-cubes}.
\end{remark}

\section{Finite characteristic algebra}

In this appendix we collect some algebraic facts that exploit the finite characteristic of the underlying field $\F$ (or the finite torsion of various abelian groups).

\subsection{Compact abelian torsion groups}

Recall that a group $U$ is \emph{$m$-torsion} if we have $u^m=1$ for all $u \in U$.

\begin{lemma}[Open sets of torsion groups contain open subgroups]\label{tor-lem} Let $U$ be a compact abelian $m$-torsion group for some $m \geq 1$.  Let $V$ be an open neighborhood of the identity in $U$.  Then $V$ contains an open subgroup $W$ of $U$.
\end{lemma}

\begin{proof}  We will use a Fourier-analytic method.  As $V$ is an open neighborhood of the origin, one can find another open neighborhood $V'$ of the origin such that $V' - V' \subset V$.

Let $\mu$ be the Haar measure on $U$, then $\mu(V') > 0$.  Let $\eps > 0$ be a small number (depending on $\mu(V')$) to be chosen later.  By Fourier analysis, we can approximate the indicator function $1_{V'}$ to within $\eps$ in $L^2(U)$-norm by some linear combination $F$ of finitely many characters $\chi_1,\ldots,\chi_n \in \hat U$, where $n$ is finite but potentially unbounded.  Since $U$ is $m$-torsion, each character $\chi_j$ takes on at most $m$ values, with each level set of $\chi_j$ being a coset of an open subgroup of $U$.  If we let $W$ be the intersection of the kernels of all the $\chi_j$, then $W$ is also an open subgroup of $U$, and $F$ is constant on every coset of $W$.  Since $F$ approximates $1_{V'}$ to within $\eps$, we conclude (if $\eps$ is sufficiently small depending on $\mu(V')$) that there exists a coset of $W$ on which $V'$ has density greater than $1/2$.  But then this forces $W \subset V'-V'$ and hence $W \subset U$, as desired.
\end{proof}

\begin{lemma}[Splitting lemma]\label{subgroup} Let $U$ be a compact abelian $m$-torsion group for some $m \geq 1$.  Let $W$ be an open subgroup of $U$.  Then there exists a splitting $U = W' \times Y$, where $W'$ is an open subgroup of $W$, and $Y$ is a finite abelian $m$-torsion group.
\end{lemma}

\begin{proof} It is known (see e.g. \cite[Chapter 5, Theorem 18]{morris}) that a compact abelian $m$-torsion group $U$ is topologically isomorphic to the direct product of cyclic $m$-torsion groups\footnote{In particular, the bounded torsion allows us to avoid having to deal with procyclic groups  which are not direct products of cyclic groups.}.  Thus $W$ must contain a cylinder neighbourhood $W'$ of the origin, i.e. a cofinite sub-product of these cyclic groups.  Since one clearly has the desired splitting $U = W' \times Y$, the claim follows.
\end{proof}

\subsection{Polynomials are discretely valued}

An important fact about phase polynomials over $\Fw$, which is not true for polynomials over some other groups $G$ (such as the integers $\Z$), is that such polynomials only take finitely many values.  More precisely, if we let $C_n := \{ z \in \C: z^n = 1\}$ denote the cyclic group of $n^\th$ roots of unity, we have 

\begin{lemma}[Phase polynomials over $\Fw$ are discretely valued]\label{L:values-in-F_p}
Let $\F$ be a finite field of characteristic $p$, and let $\X = (X, \B_X, \mu_X,(T_g)_{ g \in \Fw})$ be an
ergodic  $\Fw$-system. 
\begin{itemize}
\item[(i)] If $f \in \Phase_{<k}(\X,S^1)$ for some $k \geq p$, then $f^p \in \Phase_{<k-p+1}(\X,S^1)$.
\item[(ii)] If $f \in \Phase_{<k}(\X,S^1)$ for some $k \geq 1$, then (after multiplying $f$ by a constant), $f$ takes values in $C_{p^{\lfloor (k-2)/(p-1) \rfloor + 1}}$.  In other words,
$ \Phase_{<k}(\X) = S^1 \cdot \Phase_{<k}(\X, C_{p^{\lfloor (k-2)/(p-1) \rfloor+1}}).$
\item[(iii)] If $f \in \Phase_{<k}(\X,S^1)$ for some $k \geq 1$, and $f$ takes values in $C_{p^{\lfloor (k-2)/(p-1) \rfloor + 1}}$, then for any $g \in \Fw$, $\prod_{i=0}^{p-1} T_g^i f$ takes values in $C_{p^{\lfloor k/p \rfloor}}$.
\item[(iv)] If $f \in \Phase_{<k}(\Fw,\X,S^1)$ is a cocycle for some $k \geq 1$, then $f$ takes values in $C_{p^{\lfloor k/p \rfloor + 1}}$.
\end{itemize}
\end{lemma}

\begin{proof}  To prove (i), it suffices to verify it in the case $k=p$, since the higher cases then follow by induction and from the identity $\mder_g (f^p) = (\mder_g f)^p$.  Taking logarithms, it suffices to show that if $F: \X \to \R/\Z$ is a polynomial of degree $<p$, then $pf$ is constant.

Let $g \in G$.  Since $T_g^p f = f$ and $T_g = 1 + \ader_g$, we conclude using the binomial formula that
$\sum_{i=0}^p \binom{p}{i} \ader_g^i f = f.$
Since $f$ has degree $<p$, $\ader_g^p f = 0$.  We conclude that
$$ p \ader_g f + \binom{p}{2} \ader_g f + \ldots + p \ader_g^{p-1} f = 0$$
which we rewrite as
$$ (1 + \frac{p-1}{2} \ader_g + \ldots + \ader_g^{p-2}) \ader_g pf = 0.$$
Inverting the expression in brackets using Neumann series (and using the fact that $\ader_g^{p-1}$ annihilates $\ader_g pf$) we conclude that $\ader_g pf = 0$ for any $g$, thus by ergodicity $pf$ is constant as claimed.

To prove (ii), we first observe that it suffices to prove the claim for $k$ of the form $k=pm+1$ for integer $m$.  But the claim is trivial for $m=1$, and from (i), we see that the claim for $m$ implies the claim for $m+1$, and so (ii) follows by induction.

To prove (iii), it suffices to do so for $k$ of the form $k=pm-1$ for some integer $m$, since the claim is trivial otherwise.  By (i), the claim for $m$ implies the claim for $m-1$, so it suffices to verify the case $m=p-1$.  Taking logarithms, it suffices to show that if $F: \X \to \R/\Z$ is a polynomial of degree $<p-1$ with $pF=0$, then $\sum_{i=0}^{p-1} T_g^i F = 0$.  But writing $T_g = 1 + \ader_g$ we obtain the identity
$ \sum_{i=0}^{p-1} T_g^i = \sum_{i=0}^{p-1} \binom{p}{i+1} \ader_g^i$
and the claim follows, since $\binom{p}{i+1}$ is a multiple of $p$ for all $0 \leq i < p-1$, and $\ader_g^{p-1}$ annihilates $F$.

To prove (iv), observe from the cocycle equation that $\prod_{i=0}^{p-1} f( g, T_g^i x ) = 1$ for all $g \in \Fw$ and almost all $x \in X$.  On the other hand, from (ii) we know that for fixed $g$, $f( g, \cdot )$ is equal to a constant $c_g$ times a polynomial taking values in $C_{p^{\lfloor (k-2)/(p-1) \rfloor + 1}}$, so by (iii), $\prod_{i=0}^{p-1} f( g, T_g^i x )$ is equal to $c_g^p$ times a quantity in $C_{p^{\lfloor k/p \rfloor}}$.  Thus $c_g \in C_{p^{\lfloor k/p \rfloor + 1}}$, and the claim follows.
\end{proof}

\begin{remark} The claims are sharp.  For instance, in the characteristic $2$ space $\F_2^\omega$ (which acts on itself by translations), the function
$ \phi_k: (x_1,x_2,\ldots) \mapsto e^{2\pi i (\sum_{j=1}^\infty |x_j|)/2^k}$
for $k \geq 1$, where $x \mapsto |x|$ is the obvious map from $\F_2$ to $\{0,1\}$, is a phase polynomial of degree $k$ which takes values in $C_{2^k}$, but in no smaller group.  
\end{remark}

\subsection{Roots of phase polynomials}

We now develop some machinery that will allow us to take $m^\th$ roots of phase polynomials and still obtain a phase  polynomial.  It will be convenient to use the notation $O_{a_1,\ldots,a_k}(1)$ to denote a quantity bounded in magnitude by some constant $C(a_1,\ldots,a_k)$ depending only on $a_1,\ldots,a_k$.  Throughout this appendix, $\F$ is a finite field of characteristic $p$, and $\X$ is n $\Fw$-system.

For any cyclic $p$-group $\Z/p^m\Z$, and any $0 \leq j \leq m-1$, let $b_j = b_{j,m}: \Z/p^m\Z \to \{0,1,\ldots,p-1\}$ be the $j^{th}$ digit map, thus
$x = \sum_{j=0}^{m-1} b_j(x) p^j$
for all $x \in \Z/p^m \Z$.  

Recall that a map $P: X \to H$ into an \emph{additive} group $H$ is a polynomial of degree $<d$ if we have $\ader_{g_1} \ldots \ader_{g_d} P = 0$ for all $g_1,\ldots,g_d \in G$.

\begin{proposition}[Digits of polynomials are polynomials]\label{dig-prop} Let $P: X \to \Z/p^m\Z$ be a polynomial of degree $<d$, and let $\Z/p^l\Z$ be a cyclic group.  Embed $\{0,1,\ldots,p-1\}$ into $\Z/p^l\Z$ in the obvious manner.  Then for any $0 \leq j < m-1$, $b_j(P)$ is a polynomial of degree $<O_{l,d,p,j}(1)$.
\end{proposition}

\begin{proof} We may assume inductively that the claim is already proven for smaller values of $d$; for the same value of $d$ and smaller values of $l$; or the same value of $d$ and $l$ and smaller values of $j$.  We abbreviate $O_{l,d,p,j}(1)$ as $O(1)$.

From primary school arithmetic we know that we have a formula of the form
\begin{align*}
b_j(x+y) &= b_j(x) + b_j(y) + c_j( b_0(x), \ldots, b_{j-1}(x), b_0(y), \ldots, b_{j-1}(y) )\\
&\quad - p c_{j+1}( b_0(x), \ldots, b_{j}(x), b_0(y), \ldots, b_{j}(y) )
\end{align*}
for some ``carry bit'' functions $c_j: \{0,1,\ldots,p-1\}^{2j} \to \{0,1\}$.  Applying this with $P$ and $\ader_g P$ for some group element $g$ we conclude
\begin{align*}
b_j( T_g P ) &= b_j( P ) + b_j( \ader_g P ) + c_j( b_0(P), \ldots, b_{j-1}(P), b_0(\ader_g P), \ldots, b_{j-1}(\ader_g P) ) \\
&\quad - p c_{j+1}( b_0(P), \ldots, b_{j}(P), b_0(\ader_g P), \ldots, b_{j}(\ader_g P) )
\end{align*}
and so
\begin{align*}
\ader_g(b_j( P )) &= b_j( \ader_g P ) + c_j( b_0(P), \ldots, b_{j-1}(P), b_0(\ader_g P), \ldots, b_{j-1}(\ader_g P) ) \\
&\quad - p c_{j+1}( b_0(P), \ldots, b_{j}(P), b_0(\ader_g P), \ldots, b_{j}(\ader_g P) ).
\end{align*}
By the induction hypothesis on $d$, $b_j(\ader_g P)$ is already a polynomial of some degree $O(1)$.  By the induction hypothesis on $j$, $b_0(P), \ldots, b_{j-1}(P), b_0(\ader_g P), \ldots, b_{j-1}(\ader_g P)$ are also polynomials of degree $O(1)$.  The carry function $c_j$, by Lagrange interpolation, can be expressed as a polynomial in $\Z/p^l\Z$ of its arguments; the key point here is that as the arguments lie in $\{0,\ldots,p-1\}$, the denominators in the Lagrange interpolation formula contain no factors of $p$ and are thus invertible.  We thus conclude that 
$c_j( b_0(P), \ldots, b_{j-1}(P), b_0(\ader_g P), \ldots, b_{j-1}(\ader_g P) )$ is also a polynomial of degree $O(1)$.

Finally, by the induction hypothesis on $l$, we know that $$b_0(P), \ldots, b_{j}(P), b_0(\ader_g P), \ldots, b_{j}(\ader_g P)$$ are polynomials of degree $O(1)$ in $\Z/p^{l-1}\Z$, and thus by arguing as before $$c_{j+1}(b_0(P), \ldots, b_{j}(P), b_0(\ader_g P), \ldots, b_{j}(\ader_g P))$$
is a polynomial of degree $O(1)$ in $\Z/p^{l-1}\Z$.  This implies that
$$p c_{j+1}(b_0(P), \ldots, b_{j}(P), b_0(\ader_g P), \ldots, b_{j}(\ader_g P))$$ 
is a polynomial of degree $O(1)$ in $\Z/p^l\Z$.

Putting all this together we see that $\ader_g(b_j(P))$ is a polynomial of degree $O(1)$ for all $g$, and hence $b_j(P)$ is a polynomial of degree $O(1)$, thus closing the induction.
\end{proof}

\begin{corollary}[Functions of phase polynomials are phase polynomial]\label{func} Let $\phi_1,\ldots,\phi_m$ be $(\X,S^1)$-phase polynomials of degree $<d$ for some $d, m \geq 1$, let $n \geq 1$, and let $F(\phi_1,\ldots,\phi_m)$ be some function of $\phi_1,\ldots,\phi_m$ taking values in the cyclic group $C_{p^n}$.  Then $F(\phi_1,\ldots,\phi_m)$ is a $(\X,S^1)$-phase polynomial of degree $O_{p,d,m,n}(1)$.
\end{corollary}

\begin{proof}  We have the freedom to rotate each of the $\phi_j$ by a constant.  By Lemma \ref{L:values-in-F_p}, this allows us to assume that all the $\phi_1,\ldots,\phi_m$ take values in $C_{p^d}$ (say), thus $\phi_j = e( P_j/p^d )$ for some additive polynomials $P_j: X \to C_{p^d}$ of degree $<d$.  Now observe that $F$ can be viewed as a $e(G( (b_i(P_j))_{0 \leq i < d; 1 \leq j \leq m})/p^n)$ for some function $G: \{0,\ldots,p-1\}^{md} \to \Z/p^n\Z$.  By Lagrange interpolation, $G$ can be viewed as the restriction of a polynomial from $(\Z/p^n\Z)^{md}$ to $\Z/p^n\Z$ with degree $O_{d,m,p,n}(1)$.  The claim now follows from Proposition \ref{dig-prop}.
\end{proof}

We isolate one special case of Corollary \ref{func}:
 
\begin{corollary}[Phase polynomials have phase polynomial roots]\label{root}  Let $\phi$ be a $(\X,S^1)$-phase polynomial of degree $<d$ for some $d \geq 1$, and let $n \geq 1$.  Then there exists a phase polynomial $\psi$ of degree $O_{d,n,p}(1)$ such that $\psi^n = \phi$.
\end{corollary}

\begin{proof}  By rotating $\phi$ by a constant and using Lemma \ref{L:values-in-F_p}, we may assume that $\phi$ takes values in $C_{p^m}$ for some $m = O_{d,p}(1)$. If $n$ is not divisible by $p$, then $n$ is invertible in $C_{p^m}$ and the claim is immediate, so it suffices to check the case when $n$ is a power of $p$.  But then the claim follows immediately from Corollary \ref{func}.
\end{proof}

Another interesting consequence of Corollary \ref{func} (or Proposition \ref{dig-prop}) is that phase polynomials can always be expressed in terms of $C_p$-valued polynomials of higher degree.

\begin{corollary}[Representation of phase polynomials by $C_p$-valued phase polynomials]\label{phaserep} Let $\phi$ be a $(\X,S^1)$-phase polynomial of degree $<d$ for some $d \geq 1$.  Then $\phi$ can be expressed as a function of $O_{d,p}(1)$ many $C_p$-valued $(\X,S^1)$-phase polynomials of degree $<O_{d,p}(1)$.
\end{corollary}

\section{Connection with cubic complexes}\label{cubic-sec}

In this appendix we point out some connections between the notions of polynomiality and type in this paper with the theory of cubic complexes as used in topology, as set out in \cite{polyak}, in analogy with the more well-known simplicial complexes used in that field.  This material is not used elsewhere in this paper.

Abstractly, a cubic complex is a sequence of spaces $X^{[k]}, X^{[k-1]}, \ldots, X^{[0]}$, together with maps $\partial(\alpha)_*: X^{[k]} \to \X^{[l]}$ for every $l$-face $\alpha$ of $\2^k$, such that one has the relation $\partial(\alpha)_* \partial(\beta)_* = \partial( \partial(\beta)(\alpha) )_*$ whenever $k \geq l \geq m$, $\beta$ is an $l$-face of $\2^k$, and $\alpha$ is an $m$-face of $\2^l$ (so that $\partial(\beta)(\alpha)$ is an $m$-face of $\2^k$.  Note that the concrete cubic complex defined in Definition \ref{complex} is of this form.

Let $\X$ be an ergodic $G$-system for some countable abelian group $G$, and let $U=(U,+)$ a locally compact abelian group, which we now express additively for compatibility with \cite{polyak}.
By Lemma \ref{basic}(iii), a $(G,\X,U)$-function $f$ is an a $(G,\X,U)$-polynomial of degree $<k$ if and only if $d^{[k]} f = 0$ $\mu^{[k]}$-almost everywhere.  In the language of \cite{polyak}, this is equivalent to $f$ being of \emph{degree $<k$} in the sense of cubic complexes.

In the case that $U$ is uniquely divisible by $1,\ldots,k$ (i.e. for every $1 \leq j \leq k$ and $u \in U$ there is a unique solution $u/j\in U$ to the equation $j(u/j)=u$), one can express the operator $d^{[k]}: M(G,\X,U) \to M(G,\X^{[k]},U)$ as the composition of the $k$ differentiation operators $d: M(G,\X^{[j-1]},U) \to M(G,\X^{[j]},U)$ defined by
$ d f(g,x): = \frac{1}{j} \sum_{i=1}^j ( f(g,\partial^+_i {\bf x})-f(g,\partial^-_i {\bf x})).$

\end{document}